\documentclass[a4paper, 12pt]{article}
\usepackage[utf8]{inputenc}
\usepackage[a4paper, total={7in, 10in}]{geometry}
\usepackage{mathrsfs}
\usepackage{amsfonts}
\usepackage{amsmath}
\usepackage{amssymb}
\usepackage{amsthm}
\usepackage{dutchcal}
\usepackage{enumitem}
\usepackage{hyperref}
\usepackage{import}
\usepackage[toc]{appendix}
\usepackage{array}

\newcolumntype{P}[1]{>{\centering\arraybackslash}p{#1}}
\newcolumntype{M}[1]{>{\centering\arraybackslash}m{#1}}

\newcommand{\R}{\mathbb{R}}

\newcommand{\N}{\mathbb{N}}

\renewcommand{\L}{\mathbb{L}}

\newtheorem{theorem}{Theorem}[section]
\newtheorem{lemma}{Lemma}[section]
\newtheorem*{definition}{Definition}
\newtheorem{remark}{Remark}[section]
\newtheorem{proposition}{Proposition}[section]
\newtheorem{corollary}{Corollary}[section]

\numberwithin{equation}{section}

\title{Unilateral Problems for Quasilinear Operators with Fractional Riesz Gradients}
\author{
  Pedro Miguel Campos\footnote{CMAFcIO – Departamento de Matemática, Faculdade de Ciências, Universidade de Lisboa P-1749-016 Lisboa, Portugal. \,\,Email address: \texttt{pmcampos@fc.ul.pt}}\\
  \and
  José Francisco Rodrigues\footnote{CMAFcIO – Departamento de Matemática, Faculdade de Ciências, Universidade de Lisboa P-1749-016 Lisboa, Portugal. \,\,Email address: \texttt{jfrodrigues@ciencias.ulisboa.pt}}
}
\date{}

\begin{document}

\maketitle

\begin{abstract}
    In this work, we develop the classical theory of monotone and pseudomonotone operators in the class of convex constrained Dirichlet-type problems involving fractional Riesz gradients in bounded and in unbounded domains $\Omega\subset\mathbb{R}^d$. We consider the problem of finding $u\in K^s$, such that,
    \begin{equation*}
        \int_{\mathbb{R}^d}{\boldsymbol{a}(x,u,D^s u)\cdot D^s(v-u)}\,dx+\int_\Omega{b(x,u,D^s u)(v-u)}\,dx\geq 0 %\langle F, v-u\rangle
    \end{equation*}
    for all $v\in K^s$. Here $K^s\subset\Lambda^{s,p}_0(\Omega)$ is a non-empty, closed and convex set of a fractional Sobolev type space $\Lambda^{s,p}_0(\Omega)$ with $0\leq s\leq 1$ and $1<p<\infty$, and $D^s$ denotes the distributional Riesz fractional gradient, with two limit cases: $D^1=D$ representing the classical gradient in the classical Sobolev space $\Lambda^{1,p}_0(\Omega)=W^{1,p}_0(\Omega)$, and $-D^0=R$ denotes the vector-valued Riesz transform within $\Lambda^{0,p}_0(\Omega)=\{u\in L^p(\mathbb{R}^d):\, u=0 \mbox{ a.e. in } \mathbb{R}^d\setminus\Omega\}$.
    
    We discuss the existence and uniqueness of solutions in this novel framework and we obtain new results on the continuous dependence, with respect to the fractional parameter $s$, of variational solutions corresponding to several classical assumptions on the structural functions $\boldsymbol{a}$ and $b$ adapted to the fractional framework. We introduce an extension of the Mosco convergence for convex sets $K^s$ with respect to the parameter $s$, including the limit cases $s=1$ and $s=0$, to prove weak or strong convergences of the solutions $u_s$ and their fractional gradients $D^su$, according to different cases. Several applications are illustrated with examples of unilateral problems, including quasi-variational inequalities with constraints of obstacle type $u\geq \psi$ and $s$-gradient type $|D^su|\leq g$.
\end{abstract}

\tableofcontents

\section{Introduction}

In 2015, T.-T. Shieh and D. Spector have considered a new type of fractional differential operators in \cite{shieh2015On}, combining results of Harmonic Analysis, specially pseudodifferential operators, with Calculus of Variations. 
The functionals in \cite{shieh2015On} defining  those operators are given with the distributional Riesz fractional gradient, denoted by $D^s$. 

When applied to smooth functions, the $D^s$ can be computed as the gradient $D$ of the convolution of the function with the Riesz kernel.
More precisely, the operator $D^s$ can be expressed as
\begin{equation}\label{eq:gradient_of_riesz_potential}
    D^s u= D(I_{1-s}*u),
\end{equation}
where $I_{1-s}(x)= \frac{\mu_{d,s}}{d+s-1}|x|^{1-s-d}$ denotes the Riesz kernel  in $\R^d$ with $s\in(0,1)$ and $\mu_{d,s}$ is given by \eqref{eq:mu}. These fractional derivatives have some advantages in some types of fractional differential equations. 
Firstly, they can be extended to more general classes of function spaces of Sobolev type and they approach, in a natural way, the classical derivatives  due to the fact that the Riesz kernel $I_{1-s}$ is an approximation of the identity as $s\nearrow 1$. 
Secondly, being vector-valued like the classical gradient, contrasting with the scalar fractional Laplacian $(-\Delta)^s$, they are suitable for inhomogeneous and anisotropic problems, since duality methods apply to the Riesz fractional s-gradient. 
Moreover, it can also provide an eleventh characterization of the fractional Laplacian in $\R^d$ with respect to \cite{kwasnicki2017Ten}. 
Indeed, denoting by $\Phi\mapsto\mathrm{div}_s\Phi=D^s \cdot\Phi$ the $s$-divergence, we have in the distributional sense
\begin{equation}
    (-\Delta)^su=-D^s\cdot D^s u.
\end{equation}

Subsequently, A. Ponce made a brief introduction to the fractional gradient and divergence operators for smooth functions with compact support, \cite[pages 246-251]{ponce2016Elliptic}, and
M. \v{S}ilhav\'{y} published  in \cite{silhavy2020Fractional} a fractional vector calculus developing ideas presented in a seminar \cite{silhavy2016Beyond}. 
In the introduction of \cite{shieh2018On} it is mentioned the equivalence of the definition of the distributional Riesz fractional gradient \eqref{eq:gradient_of_riesz_potential} with the definition \eqref{eq:s_fractional_gradient} of $s$-fractional gradient  introduced by \v{S}ilhav\'{y} in \cite{silhavy2016Beyond} and \cite{silhavy2020Fractional}, which corresponds to 
\begin{equation}\label{N}
    D^s u= -N_{-s}*u,
\end{equation}
with $N_{-s}(x)= -D I_{1-s}(x)=\mu_{d,s}x|x|^{-1-s-d}$ being the vector-valued kernel of Horváth \cite{horvath1959On}. Actually, this definition can be seen as a special case of a definition used by Sobolev and Nikol\textquotesingle ski\u{\i} in \cite{sobolev1963Embedding} in a more general framework, with generalized homogeneous functions of degree $-d-s$ (a class of functions containing the Horváth kernel as a special case), under the name of \textit{Liouville derivatives of order $s$}. The proof of that equivalence was proven in \cite{ponce2016Elliptic} and in the work of G. Comi and G. Stefani \cite{comi2019BlowUp}, respectively, for smooth and Lipschitz functions with compact support.

\v{S}ilhav\'{y}'s treatment was different from the one of Shieh and Spector, as he was interested in showing that these nonlocal operators satisfy the ``natural qualitative requirements on the fractional operators''. 
Notably, \v{S}ilhav\'{y} has shown that $D^s$ is, up to a multiplicative constant, the only operator applying scalar functions into vector valued functions that satisfies: 
$(i)$ rotational and translational invariance, important for the operators to be independent of the chosen basis; 
$(ii)$ a mild continuity hypothesis, $D^s$ is sequential continuous for test functions in a suitable topology, and 
$(iii)$ $s$-homogeneity under isotropic scaling, which characterizes the fractional nature of the operator. 
In addition, using the weaker form of the definitions of the fractional gradient and divergence, he derived explicit formulas for these operators, in particular, showing that $(-\Delta)^{(s+r)/2} u=-D^s\cdot D^r u$ and computing directly the fractional gradient of distributions like the Dirac delta or the Heaviside function in $\R$, \cite[Examples 6.2 and 6.3]{silhavy2020Fractional}.

\vspace{2mm}

With this notion of fractional gradient, for $0<s<1$ and $1< p<\infty$, in \cite{shieh2015On} it was also introduced the following Sobolev-type space,  
\begin{equation}\label{L-C}
    \Lambda^{s,p}(\R^d)=\overline{C^\infty_c(\R^d)}^{\|\cdot\|_{\Lambda^{s,p}(\R^d)}},
\end{equation}
i.e., the completion of the space of test functions for the norm
\begin{equation}
    \|u\|_{\Lambda^{s,p}(\R^d)}=\left(\|u\|^p_{L^p(\R^d)}+\|D^s u\|^p_{L^p(\R^d;\R^d)}\right)^{1/p}.
\end{equation}

These spaces, that we call here Lions-Calderón spaces, were not new, as they were shown,  for $1<p<\infty$, also in \cite{shieh2015On}, to be equivalent to the so-called Bessel potential spaces in \cite{adams1975Sobolev, adams1996Function, stein1970Singular} or generalized (inhomogeneous) Sobolev spaces in \cite{grafakos2014Modern}. With this new approach, they are well suited to the study of some fractional partial differential equations, including nonhomogeneous and anisotropic variants of the fractional Laplacian. Therefore they have been well studied in the literature, as they have a long and rich history. Indeed, the oldest reference we could find for these spaces is due to N. Aronzajn and K. Smith, who in \cite{aronszajn1959theory, aronszajn1961theory} defined them in the Hilbertian case $p=2$. As it is well known, this case coincides with the Sobolev-Slobodeckij space $W^{s,2}(\R^d)=H^s(\R^d)$. Later, those spaces were independently generalized to the non-Hilbertian case $1<p<\infty$, by  J.-L. Lions in \cite{lions1960Une}, using complex interpolation between $L^p(\R^d)$ and $W^{1,p}(\R^d)$, and by A. P. Calderón in \cite{calderon1961Lebesgue}, using Bessel potentials. The initial definition in terms of Bessel potentials $G_s$, with their different notations $H^{s,p}(\R^d)=L^p_s(\R^d)=G_s(L^p(\R^d))=\Lambda^{s,p}(\R^d)$,  is not sufficient to justify their terminology as ``Bessel potential spaces''. In fact, both the Besov and the Triebel-Liozorkin spaces can be obtained also by the image of Bessel potential, respectively $B^{s,p}_q(\R^d)=G_s(B^{0,p}_q(\R^d))$ and $F^{s,p}_q(\R^d)=G_s(F^{0,p}_q(\R^d))$, and we have $B^{s,p}_p(\R^d)=W^{s,p}(\R^d))$ and $F^{s,p}_2(\R^d)=\Lambda^{s,p}(\R^d))$, see \cite{triebel1983Theory} for instance. 

On the other hand, as in \cite{comi2019BlowUp} under the name of \textit{distributional fractional Sobolev spaces} for $1\leq p\leq\infty$, the Lions-Calderón spaces were defined directly in terms of the $s$-fractional gradient, i.e.
 \begin{equation}\label{eq:first_lions_calderon}
	\Lambda^{s,p}(\R^d)=\{u\in L^p(\R^d):\, D^s u\in L^p(\R^d;\R^d)\},
\end{equation}
which equivalence with the definition \eqref{L-C} was proved for $1<p<\infty$ in \cite{brue2022AsymptoticsII} and independently in \cite{kreisbeck2022Quasiconvexity}, while for $p=1$ it has been earlier proven in \cite{comi2019BlowUp}.
Actually this definition was also suggested for general spaces of functions with Liouville derivatives by Sobolev and Nikol\textquotesingle ski\u{\i} in the survey \cite[page 150]{sobolev1963Embedding}. The alternative names ``generalized Sobolev spaces'' or ``distributional fractional Sobolev spaces'', although possibly more appropriate to this general class of functional spaces of the form $\{u\in L^p: \mathcal{D}u\in L^p\}$, with $\mathcal{D}$ being some broadened notion of derivatives, colides with the Sobolev-Slobodeckij spaces $W^{s,p}(\R^d)$, which are also well-known as the fractional Sobolev spaces and do not coincide with the spaces $\Lambda^{s,p}(\R^d)$, except when $p=2$ or when $s$ is an integer. 

Another important property of the fractional gradient  $D^s$ is the continuous dependence with respect to the parameter $s$. In fact, it was noted, in \cite{silhavy2016Beyond} and in \cite{ponce2016Elliptic} without proofs, in \cite{rodrigues2019On} using Kurokawa's results \cite{kurokawa1981On} on the approximating of the identity by Riesz kernels, and in \cite{bellido2020gamma} using Fourier Analysis and an interpolation inequality, that the following property for sufficiently regular functions $u$,
\begin{equation*}
    D^s u \to D u \quad \mbox{ in } L^p(\R^d;\R^d)\,\,\mbox{ as } s\nearrow 1, \mbox{ for } 1<p<\infty.
\end{equation*}
In \cite{comi2022asymptoticsI, brue2022AsymptoticsII}, using direct estimates and properties of the singular integrals, the authors extended these type of continuity results to the variation of $s$, by showing that
\begin{equation*}
    D^s u\to D^\sigma u \quad \mbox{ in } L^p(\R^d;\R^d)\,\,\mbox{ as }s\to \sigma\in (0,1],  \mbox{ for } 1\leq p<\infty,
\end{equation*}
and also
\begin{equation*}
    D^s u\to -R u\quad \mbox{ in } L^p(\R^d, \R^d)\,\,\mbox{ as } s\searrow 0, \mbox{ for } 1<p<\infty.
\end{equation*}
Here $R$ is the Riesz transform (see \cite{stein1970Singular, grafakos2014Classical}), which motivates the notation $D^0=-R$.

Concerning applications to partial differential equations, these spaces had also appeared in  \cite{lions1961problemi} and in \cite{jerison1995Inhomogeneous}, to study the regularity theory for linear elliptic partial differential equations, and in \cite{linares2015Introduction} to study a variety of questions about solutions to dispersive partial differential equations.  However the new characterization of these spaces using the distributional Riesz fractional gradient allows the study of fractional partial differential equations, as proposed in \cite{shieh2015On}, as well as extensions to unilateral problems, as in \cite{rodrigues2019On, azevedo2023constrain, azevedo2022transport, lo2021class} and other problems of the Calculus of Variations, as in \cite{shieh2018On}, \cite{bellido2020piola, bellido2020gamma} and \cite{kreisbeck2022Quasiconvexity}, in particular with applications in Peridynamics \cite{silling2000Reformulation, silling2007States} to model fractional hyperelasticity, see also \cite{schonberger2021characterization}.  On the other hand, this functional framework for the Riesz fractional gradient is also well suited for certain classes of evolution problems, such as fractional Stefan-type problems \cite{lo2023Stefan} and hyperbolic obstacle-type problems \cite{campos2023hyperbolic}.

Paraphrasing Jean Mawhin in his Foreword of \cite{bisci2016Variational}, \textit{it is therefore a natural question to see which results ``survive'' when the gradient
is replaced by the fractional gradient}, which \textit{is also fruitful because the
extension of classical results to new situations also sheds light on a better and deeper
understanding of the classical results}. In this work we are interested in more general types of stationary problems, with or without constraints, in (fractional) divergence form, including the variants of the $(s,p)$-Laplacian type:
\begin{equation}
    \begin{cases}
        -D^s\cdot\big(\alpha(u)|D^s u|^{p-2}D^s u+\boldsymbol{e}(u)\big)+b(u,D^s u)=f_0-D^s\cdot \boldsymbol{f} & \mbox{ in } \Omega\\
        u=0 &\mbox{ on } \R^d\setminus \Omega.
    \end{cases}
    \label{eq:goal_equation_distributional_form}
\end{equation}
This will be done in the general framework of monotone and pseudomonotone operators involving fractional Riesz gradients, in bounded and in unbounded domains $\Omega\subset\mathbb{R}^d$, in classes of convex constrained Dirichlet-type problems. In particular, this allow us to consider the existence of solutions to problems of the type \eqref{eq:goal_equation_distributional_form}, their extensions with obstacle or s-gradient constraints and to study their dependence with respect to the parameter $s$, including the limit cases $s=1$ and $s=0$.

In Section 2, we start with a summary of the natural functional framework of the Riesz fractional gradient $D^s$, with the precise definition of the Lions-Calderón spaces $\Lambda^{s,p}(\R^d)$ and $\Lambda^{s,p}_0(\Omega)$, with special emphasis on the differences between the properties of the latter spaces when $\Omega$ is bounded and when it is not. In particular, we recall the equivalent definitions of the fractional gradient and their important duality property in the Lions-Calderón spaces, the characterisation of their dual spaces $\Lambda^{-s,p'}(\R^d)$, $p'=p/(p-1)$, by showing the representation $F=f_0-D^s\cdot \boldsymbol{f}$, for all $F\in\Lambda^{-s,p'}(\R^d)$, the contiguity relation between $\Lambda^{s,p}(\R^d)$ and $W^{s,p}(\R^d)$ and the monotone inclusions of the spaces $\Lambda^{s,p}(\R^d)$ and $\Lambda^{s,p}_0(\Omega)$, with respect with the inverse monotone variation of the parameter $s\in[0,1]$, for $1<p<\infty$. We also recall the fractional Gagliardo-Nirenberg inequalities and the continuity of $D^s$ with respect to $s\in[0,1]$ in those spaces. We notice that these results also hold for $\Lambda^{s,p}_0(\Omega)$, since it may be considered as a subspace of $\Lambda^{s,p}(\R^d)$. Some properties of these spaces are then also complemented in the case of bounded $\Omega$, like in the classical case, as the fractional Poincaré inequality, emphasizing the dependence on the fractional parameter, and the fractional Rellich-Kondrachov compactness embedding. In particular, we prove the new explicit estimate $\|D^t u\|_{L^p(\R^d;\R^d)}\leq \frac{C}{t^{1+\frac{1}{p}}}\|D^s u\|_{L^p(\R^d;\R^d)}$ of the embedding $\Lambda^{s,p}_0(\Omega)\subset\Lambda^{t,p}_0(\Omega)$, $1<p<\infty$, depending on the fractional parameter $t$ but not on $s$, $0<t<s\leq1$, and
 we show the compact embedding
	$\Lambda^{s,p}_0(\Omega)\Subset \Lambda^{t,q}_0(\Omega)$ for all $1<p\leq q<\infty$, $0\leq t<s\leq 1$, satisfying $\frac{1}{p}-\frac{s}{d}<\frac{1}{q}-\frac{t}{d}$. We conclude this functional section with a relevant property on the density of the non negative smooth functions with compact support, in a bounded open set $\Omega$, in the non negative cone of $\Lambda^{s,p}_0(\Omega)$.

\vspace{2mm}

In the functional framework of the Lions-Calderón spaces, we study in Section \ref{sec:existence_uniqueness}, existence of solutions $u=u_s\in K^s \subset \Lambda^{s,p}_0(\Omega)$ to fractional variational inequalities of the type
\begin{equation}\label{eq:general_variational_inequality}
  \begin{split}
    &\int_{\R^d}{\boldsymbol{a}(x,u, D^s u)\cdot D^s(v-u)}\,dx+\int_{\Omega}{b(x,u, D^su)(v-u)}\,dx\geq 
    \int_\Omega{f_0(v-u)}+\int_{\R^d}{\boldsymbol{f}\cdot D^s(v-u)}\\ 
  \end{split}
\end{equation}
for all $v\in K^s$. Here $K^s$ may be general a non-empty closed convex subset of $\Lambda^{s,p}_0(\Omega)$ for $s\in[0,1]$ and $p\in(1,\infty)$, including the classical unilateral examples in the fractional setting 
\begin{equation}\label{eq:introduction_examples_convex_sets}
    K^s_{\geq\psi}=\{u\in \Lambda^{s,p}_0(\Omega): u\geq \psi\mbox{ in } \Omega\}\,\,\,\, \mbox{ and }\,\,\,\, \mathcal{K}^s_g=\{u\in \Lambda^{s,p}_0(\Omega): |D^s u|\leq g \mbox{ in } \R^d\}.
\end{equation}
In this new framework we develop the theory of monotone operators, initiated in \cite{kachurovskii1960monotone, minty1962Monotone}, and pseudomonotone operators, introduced by H. Brézis in \cite{brezis1968equations} and extended by F. Browder in \cite{browder1977PseudoMonotone}. These methods, combined with the functional properties of the fractional Riesz gradients, allow us to obtain a variety of interesting existence results for different types of problems, with $\Omega$ bounded or unbounded, in this novel framework. In the monotone case, it is possible to let $s=0$, which corresponds to a new class of nonlinear problems involving an important Calderón-Zygmund operator, namely the Riesz transform. When $\Omega$ is bounded and $s>0$, we apply the method of \cite{brezis1968equations} for pseudomonotone operators to obtain the novel existence results, while when $\Omega$ is unbounded and $s>0$ we extend the method \cite{browder1977PseudoMonotone} to the fractional framework. When $\Omega$ is unbounded but $s\geq 0$ we restrict ourselves to monotone operators and make use of more classical arguments. We observe that, as in the classical pseudomonotone case $s=1$, when $\Omega$ is unbounded the type of functions $\boldsymbol{a}$ and $b$ for which we can obtain existence results is more restrictive. 

In particular, we study specially fractional $p$-Laplacian type operators, for which the typical example is the variational $(s,p)$-Laplacian, or the $s$-fractional $p$-Laplacian
\begin{equation}\label{FractionalPLaplacian}
    u\mapsto-\Delta_p^su=-D^s\cdot(|D^s u|^{p-2}D^s u),
\end{equation}
that constitutes a continuous one parameter family of nonlinear operators as $s$ varies through $\Lambda^{s,p}_0(\Omega)$, from $0$ in $L^0_0(\Omega)$ till $1$ in $W^{1,p}_0(\Omega)$. In addition, we prove that the solution operator
\begin{equation*}
    \mathcal{S}: \Lambda^{-s,p'}(\Omega)\ni F=f_0-D^s\cdot \boldsymbol{f}\mapsto u\in K\subset\Lambda^{s,p}_0(\Omega),
\end{equation*}
for the variational inequality  \eqref{eq:general_variational_inequality} associated to fractional $p$-Laplacian type operators, for $0\leq s\leq 1$, is $\frac{1}{p-1}$-Hölder continuous for $p\geq2$, locally Lipschitz continuous for $1<p<2$ and, for $\Omega$ bounded, also compact in $\Lambda^{t,q}_0(\Omega)$ for $0<t<s$ and $1<p,q<\infty$ satisfying $\frac{1}{p}-\frac{s}{d}<\frac{1}{q}-\frac{t}{d}$. These results are new for $0\leq s\leq 1$ with $p\neq 2$ and $K\subsetneq\Lambda^{s,p}_0(\Omega)$, extending those of \cite{simon1981Regularity} for the classical Dirichlet problem corresponding to $s=1$.

\vspace{2mm}

The following section in this article, Section \ref{sec:stability}, contains the main results of this work and concerns the continuous dependence of the solutions to \eqref{eq:general_variational_inequality} with respect to the parameter $s$, which implies the variation of the convex sets $K^s$ and the respective functional spaces $\Lambda^{s,p}_0(\Omega)$. Introducing a generalized Mosco convergence \eqref{eq:GM} for the convex sets, we show that the solutions to those problems converge, up to subsequences, to a solution of the limit problems as $s\to\sigma\in[s^*,1]$, following the four settings of the Section \ref{sec:existence_uniqueness}, respectively the case of pseudomonotone operators with $0<s^*<s\leq1$, with $\Omega$ bounded, where we prove
 \begin{equation}\label{eq:introduction_convergence_brezis_stability}
        u_s\to u_\sigma \mbox{ in } \Lambda^{t,p}_0(\Omega), \mbox{ for } 0\leq t<s^*, \quad \mbox{ and } \quad D^{s}u_s\rightharpoonup D^\sigma u_\sigma \mbox{ in } L^p(\R^d;\R^d),
    \end{equation}
for pseudomonotone operators with $0<s^*<s\leq1$, with $\Omega$ unbounded, where the result is weaker
\begin{equation}\label{eq:introduction_convergence_browder_stability}
        u_s\rightharpoonup u_\sigma \mbox{ in } \Lambda^{t,p}_0(\Omega), \mbox{ for } 0\leq t\leq s^*, \quad \mbox{ and } \quad D^{s}u_s\rightharpoonup D^\sigma u_\sigma \mbox{ in } L^p(\R^d;\R^d),
    \end{equation}
and similarly for monotone operators with $0\leq s^* \leq \sigma\leq1$ and with $\Omega$ unbounded, where $\sigma=0$ is admissible with $\Lambda^{0,p}_0(\Omega)=L^p_0(\Omega)$ and $\sigma=1$ corresponds to $\Lambda^{1,p}_0(\Omega)=W^{1,p}_0(\Omega)$. In the case of $p$-Laplacian type coercive operators we have, in addition to the convergence \eqref{eq:introduction_convergence_browder_stability} for the whole sequence $s\to\sigma$, as the solutions are unique, also the stronger convergences for $0\leq \sigma\leq1$
\begin{equation}\label{eq:introduction_strong_convergence_stability}
        u_s\to u_\sigma \mbox{ in } L^p_0(\Omega) \quad \mbox{ and } \quad D^{s}u_s\to D^\sigma u_\sigma \mbox{ in } L^p(\R^d;\R^d),
    \end{equation}
and, in the case $\sigma\geq s^*>0$ with bounded $\Omega$, the first convergence in \eqref{eq:introduction_convergence_brezis_stability} also holds.

The proofs of these stability results are based in a refinement of the classical monotonicity and compactness methods used to study the existence problems and a generalization of the Mosco convergence \cite{mosco1969convergence} in a fixed Banach space $V$:
\begin{equation}
    K_n\xrightarrow{M} K \quad \Leftrightarrow \quad \text{w--}\limsup_n{K_n}\subset K\subset \liminf_n{K_n} \mbox{ in } V.
\end{equation}
Although this notion of convergence is a powerful tool to study stability problems associated to variational inequalities with respect to the variation of the convex sets, as one can see for example in
\cite{mosco1969convergence,mosco1976Implicit,rodrigues1987obstacle,antil2021on,boccardo2021some}, it cannot be used directly in our framework, since the spaces $\Lambda^{s,p}_0(\Omega)$ change with $s$. The particular case $s\nearrow \sigma$ is simpler, since the convex sets $K^s$ are non-increasing and $K^\sigma\subset K^s$, and it has been considered in \cite{lo2021class,campos2021Lions} for obstacle problems with $\sigma=1$, by using Mosco's arguments directly in the space $\Lambda^{s^*,p}_0(\Omega)$ where $0\leq s^*<s<\sigma\leq 1$.

The general case requires a generalization of Mosco's notion of convergence, which in our situation consists of working with the convergences of the pair $(u, D^s u)$ in the space $\L^p(\Omega)=L^p(\Omega)\times L^p(\R^d;\R^d)$, instead of working directly with the functions $u$ in $\Lambda^{s,p}_0(\Omega)$. This allows the ambient space $\L^p(\Omega)$ to be independent of $s$ without losing information about the regularity of $u$ and to introduce, for converging sequences  $[0,1]\ni s \to\sigma\in[0,1]$, the convergence of convex sets in the generalized sense of Mosco
\begin{equation}\label{eq:GM}
    \Lambda^{s,p}_0(\Omega)\supset K^s\xrightarrow{s-M} K^\sigma\subset \Lambda^{\sigma,p}_0(\Omega).
\end{equation}

With this approach, we are able to obtain the new results, with different combinations of weak and strong convergences of the solutions $u_s$ and their fractional gradients $D^su$, given in \eqref{eq:introduction_convergence_brezis_stability}, \eqref{eq:introduction_convergence_browder_stability} and \eqref{eq:introduction_strong_convergence_stability}. We emphasize that the convergences $s\to\sigma$ may be arbitrary, not requiring monotonicity even in the limit cases  $\sigma=1$ and $\sigma=0$. 

To complement the existence and stability results obtained in the Sections \ref{sec:existence_uniqueness} and \ref{sec:stability}, respectively, we provide some new examples and new applications in Section \ref{sec:applications}. We start in Subsection \ref{subsec:examples_mosco} with specific examples of sets that converge in the generalized sense of Mosco that was introduced in Section \ref{sec:stability}. We show that $\Lambda^{s,p}_0(\Omega)\xrightarrow{s-M} \Lambda^{\sigma,p}_0(\Omega)$, which implies, for example, that, under suitable assumptions on the nonlinear coefficients, the solutions $u^s\in \Lambda^{s,p}_0(\Omega)$ to \eqref{eq:goal_equation_distributional_form}, as $s\nearrow 1$, converge to solutions $u\in W^{1,p}_0(\Omega)$ of the classical quasilinear Dirichlet problem:
\begin{equation}
    \begin{cases}
        -D\cdot\big(\alpha(u)|Du|^{p-2}Du+\boldsymbol{e}(u)\big)+b(u,Du)=f_0-D\cdot \boldsymbol{f} & \mbox{ in } \Omega\\
        u=0 &\mbox{ on } \partial \Omega.
    \end{cases}
\end{equation}
In order to study also the stability, as $s\to\sigma\in[0,1]$, of the solutions to obstacle problems we give examples of $K^s_{\geq\psi_s} \xrightarrow{s-M} K^\sigma_{\geq\psi_\sigma}$, and also one example of $\mathcal{K}^s_{g_s} \xrightarrow{s-M} \mathcal{K}^\sigma_{g_\sigma}$, with constraint on the fractional gradient $D^s u$ instead of $u$, which we restrict to $g_s=I_{\sigma-s}g_\sigma$ and $s \nearrow \sigma\in(0,1]$.

In Subsection \ref{subsec:quasivariational}, we provide applications to quasivariational problems, i.e. problems where the convex sets $K=K(u)$ are allowed to depend on the solution $u$ of the problem, a class of problems considered, for instance, in the classical framework by \cite{bensoussan1973controle, mosco1976Implicit, joly1979propos, bensoussan1982Controlle, baiocchi1984Variational, gwinner2022uncertainty}. Fractional problems of this type were, as far as we know, first studied in Hilbertian framework by  H. Antil and C. Rautenberg in \cite{antil2018Fractional} using methods that are dependent on a comparison principle. Since it is not known if the fractional operators studied in our work satisfy any type of comparison principles, we use fixed-point arguments together with variational methods, generalizing some results that were obtained in \cite{rodrigues2019On, antil2021on, lo2021class} to the non-Hilbertian framework. We give two examples with strictly monotone coercive operators, one of obstacle type, where the obstacle is given by $\psi=\Psi(u)$, and another with a fractional gradient constraint $g=G(u)$, for suitable compact operators $\Psi$ and $G$, respectively. We conclude this subsection with two examples of continuous dependence with respect to $s$: \textit{i)} by showing the convergence \eqref{eq:introduction_convergence_brezis_stability} for $s\to\sigma\in[s^*,1]$ in the case of subsequences of solutions $u_s\in K^s_{\geq \Psi(u_s)}$ to the quasi-variational implicit obstacle problem towards a solution of the limit obstacle problem $u_\sigma\in K^s_{\geq \Psi(u_\sigma)}$, and \textit{ii)} applying a result of \cite{azevedo2022transport} proving the localization of solutions of fractional quasi-variational inequalities with s-gradient constraints in the Hilbertian framework $p=2$ with a linear operator when $s\nearrow 1$.

\section{The Lions-Calderón Spaces}\label{sec:functional_framework}

\subsection{The natural spaces for the s-fractional gradient}
Introducing the Riesz potentials
\begin{equation}\label{eq:riez_potential}
    I_{r}\varphi(x)=(I_{r}*\varphi)(x)=\frac{\mu_{d,1-r}}{d-r}\int_{\R^d}{\frac{\varphi(y)}{|x-y|^{d-r}}}\,dy.
\end{equation}
for $r=1-s\in(0,1)$ and $\varphi\in C^\infty_c(\R^d)$, with
\begin{equation}\label{eq:mu}
\mu_{d,s}=\frac{2^s\Gamma\left(\frac{d+s+1}{2}\right)}{\pi^{d/2}\Gamma\left(\frac{1-s}{2}\right)},
\end{equation}
we recall the following definition introduced in \cite{shieh2015On}:
\begin{definition}[Distributional Riesz fractional partial derivative]
    Let $s\in(0,1)$ and consider $\varphi\in C^\infty_c(\R^d)$. Then we define $(D^s\varphi)_j=\frac{\partial^s\varphi}{\partial x^s_j}$ by
    \begin{equation}\label{eq:distributional_riesz_fractional_partial_derivative}
        \left\langle\frac{\partial^s\varphi}{\partial x^s_j}, v\right\rangle=-\int_{\R^d}{(I_{1-s}\varphi)\frac{\partial v}{\partial x_j}}\,dx, \quad \forall v\in C^\infty_c(\R^d).
    \end{equation}
\end{definition}
Since $\varphi\in C^\infty_c(\R^d)$, the Riesz potential $I_{1-s}\varphi$ is actually a well-defined function belonging to $L^p(\R^d)\cap C^\infty(\R^d)$ for all $p\in(1,\infty)$, and so the distributional Riesz fractional partial derivative \eqref{eq:distributional_riesz_fractional_partial_derivative} makes sense in the distributional sense.

By definition of distributional Riesz partial derivative these operators are non-local, in the sense that in order to be computed they require information on the whole $\R^d$ and not just information along the direction of canonical vector $e_j$. Similarly to the classical derivatives, we can express the distributional Riesz fractional gradient, or simply the $s$-fractional gradient, in terms of its $s$-fractional partial derivatives \begin{equation}
    D^s \varphi=\sum_{j=1}^d\frac{\partial^s \varphi}{\partial x^s_j}e_j.
\end{equation}

When dealing with smooth functions with compact support, we can express the $s$-fractional gradient without the need of using distributions, since it is still a smooth vector function, although not having compact support. In fact we have the following characterization:
\begin{proposition}\label{prop:charact_Ds_with_riesz}
    Let $s\in (0,1)$ and $\varphi\in C^\infty_c(\R^d)$, then
    \begin{equation}\label{eq:shieh_spector_characterization}
        D^s\varphi=D(I_{1-s}\varphi)=I_{1-s}(D\varphi).
    \end{equation}
\end{proposition}
\begin{proof}
    See \cite[Theorem 1.2]{shieh2015On}.
\end{proof}
It is also possible to show that the $s$-fractional gradient $D^s$ corresponds, up to a sign, to the convolution with the Horváth's kernel \eqref{N}
\begin{proposition}
    Let $s\in (0,1)$ and $\varphi\in C^\infty_c(\R^d)$, then
    \begin{equation}\label{eq:s_fractional_gradient}
        D^s \varphi(x)=\mu_{d,s}\lim_{\varepsilon\to 0}{\int_{\{|y|>\varepsilon\}}{\frac{y \varphi(x+y)}{|y|^{d+s+1}}}\,dy},
    \end{equation}
\end{proposition}
\begin{proof}
    See  \cite[Lemma 15.9]{ponce2016Elliptic} or \cite[Proposition~2.2]{comi2019BlowUp}.
\end{proof}
One consequence of the characterizations \eqref{eq:shieh_spector_characterization} and \eqref{eq:s_fractional_gradient} is that we can extend them to the integral cases $s=1$ and $s=0$, respectively. Indeed, by observing that, in a certain sense \cite{kurokawa1981On}, the limit case $I_0$ corresponds to the identity operator, we can use \eqref{eq:shieh_spector_characterization} to extend the notion of $D^s$ to $s=1$, identifying $D^1$ and the classical gradient $D$. On the other hand, we can use the characterization in \eqref{eq:s_fractional_gradient} to extend $D^s$ to $s=0$, by identifying $D^0$ with $-R$, where $R$ is the vector-valued Riesz transform.

These $s$-fractional differential operators have some properties that are similar to the ones of the classical gradient. One of the most important, which allow us to make use of variational methods, is the duality property between the $s$-fractional gradient and the $s$-fractional divergence
\begin{equation}\label{eq:frac_divergence_in_terms_of_fra_partial_derivatives}
    D^s \cdot\Phi(x)=\mathrm{div}_s\Phi(x)=\sum_{j=1}^d{\frac{\partial^s \Phi_j}{\partial x^s_j}}.
\end{equation}

\begin{proposition}[Duality between the $s$-gradient and the $s$-divergence]\label{prop:duality_gradient_divergent}
    Let $s\in[0,1)$, $\varphi\in C_c^\infty(\R^d)$ and $\Phi\in C_c^\infty(\R^d;\R^d)$. Then
    \begin{equation}
        \int_{\R^d}{\varphi(x) D^s\cdot \Phi(x)}\,dx=-\int_{\R^d}{\Phi(x)\cdot D^s \varphi(x)}\,dx.
        \label{eq:duality_frac_divergence_frac_gradient}
    \end{equation}
\end{proposition}

\begin{proof}
	For $s\in(0,1)$, see \cite[Lemma~2.5]{comi2019BlowUp}. The limit case $s=0$ is shown in \cite[Lemma~26]{brue2022AsymptoticsII}.
\end{proof}

In fact, due to this duality property, following \cite{comi2019BlowUp} for $s>0$, we can then extend the notion of fractional differentiability to a wider class of functions, including the case $s=0$.

\begin{definition}[Weak $s$-fractional gradient]
    Let $0\leq s<1$ and consider $1\leq p\leq \infty$ if $s>0$ or $1<p<\infty$ if $s=0$. We define the weak $s$-fractional gradient of a function $f\in L^p(\R^d)$ the function $\Gamma\in L^1_\mathrm{loc}(\R^d;\R^d)$ that satisfies
	\begin{equation}
	    \int_{\R^{d}}{fD^s\cdot \Phi}\,dx=-\int_{\R^d}{\Gamma\cdot \Phi}\,dx\quad \forall \Phi\in C_c^\infty(\R^d;\R^d).
            \label{eq:def_weak_fractional_gradient}
	\end{equation}
	To simplify the notation, we write $D^s f=\Gamma$.
\end{definition}

It is important to observe that this definition only makes sense because $D^s\cdot=\mathrm{div}_s: C^\infty_c(\R^d;\R^d)\to {L^p}'(\R^d)$ continuously, see \cite[Corollary~2.3]{comi2019BlowUp} for the case $s>0$ and \cite[Corollary~5.2.8]{grafakos2014Classical} for the case $s=0$.

This notion of weak fractional gradient allows a definition of function spaces, similarly to the classical Sobolev space, suitable to study problems involving $s$-fractional gradients.

\begin{definition}[Lions-Calderón spaces]
    Let $0\leq s<1$ and consider $1\leq p\leq \infty$ if $s>0$ or $1<p<\infty$ if $s=0$. We define $\Lambda^{s,p}(\R^d)$ as the space of functions $f\in L^p(\R^d)$ with (weak) $s$-fractional gradient also in $L^p(\R^d;\R^d)$, i.e.
	\begin{equation}
		\Lambda^{s,p}(\R^d):=\{f\in L^p(\R^d):\, D^s f\in L^p(\R^d;\R^d)\}.
            \label{eq:def_lions_calderon_spaces}
	\end{equation}
	This space is endowed the norm
	\begin{equation}
		\|f\|_{\Lambda^{s,p}(\R^d)}:=\left(\|f\|_{L^p(\R^d)}^p+\|D^s f\|_{L^p(\R^d;\R^d)}^p\right)^{1/p}.
            \label{eq:norm_lions_calderon_spaces}
	\end{equation}
	We also define the completion of $C^{\infty}_c(\R^d)$ for the norm ${\|\cdot\|_{\Lambda^{s,p}(\R^d)}}$ as the Banach space
    \begin{equation}    
        \Lambda^{s,p}_0(\R^d):=\overline{C^\infty_c(\R^d)}^{\|\cdot\|_{\Lambda^{s,p}}}.
    \end{equation}
\end{definition}

This last definition was introduced in \cite{shieh2015On}, where the authors have show that $\Lambda^{s,p}_0(\R^d)$, for $1<p<\infty$, coincides with the Bessel potential spaces or generalized Sobolev spaces
\begin{equation}
	H^{s,p}(\R^d)=\{((1+4\pi^2|\xi|^2)^{-s/2}\hat{f})^\vee\,: f\in L^p(\R^d)\}, 
    \label{eq:def_bessel_potential_spaces}
\end{equation}
 where it is also known the density of $C^\infty_c(\R^d)$ for $s\geq 0$ and $p>1$, see \cite{lions1961problemi}. 
 
 On the other hand, it was proven
 that $C^\infty_c(\R^d)$ is also dense in $\Lambda^{s,p}(\R^d)$, in \cite{brue2022AsymptoticsII} and in \cite{kreisbeck2022Quasiconvexity}, showing, as in the classical Sobolev space $W^{1,p}(\R^d)$ (corresponding to $s=1$), the identity of all those three spaces. 
 
The case $s=0$ in the Definition \ref{eq:def_lions_calderon_spaces} only includes the cases $1<p<\infty$ and we can identify $\Lambda^{0,p}(\R^d)$ with $L^p(\R^d)$. These results can be summarized (and written more precisely) in the following proposition.
\begin{proposition}[Lions-Calderón spaces and Bessel Potential Spaces]\label{prop:distributional_characterization_lions_calderon_spaces}
    The following identities hold:
    \begin{itemize}
        \item[i)] $\Lambda^{s,p}_0(\R^d)=H^{s,p}(\R^d)$ with equivalent norms, when $0<s<1$ and $1<p<\infty$;
        \item[ii)] $\Lambda^{s,p}(\R^d)=\Lambda^{s,p}_0(\R^d)$ when $0<s<1$ and $1\leq p<\infty$; and
        \item[iii)] $\Lambda^{0,p}(\R^d)=L^p(\R^d)$ with equivalent norms when $1<p<\infty$.
    \end{itemize}
\end{proposition}

\begin{proof}
    \begin{itemize}
        \item[i)] It was proven in \cite[Theorem~1.7]{shieh2015On}.
        \item[ii)] It was proven in \cite[Theorem~2.7]{kreisbeck2022Quasiconvexity} when $1<p<\infty$ and in \cite[Theorem~A.1]{brue2022AsymptoticsII} when $1\leq p<\infty$. 
        \item[iii)] On the one hand, since $D^0=-R:L^p(\R^d)\to L^p(\R^d;\R^d)$ continuously for $p\in(1,\infty)$, see \cite[Corollary~5.2.8]{grafakos2014Classical}, we have that $\Lambda^{0,p}(\R^d)\subset L^p(\R^d)$. On the other hand, using the identity $-D^0\cdot D^0=-R\cdot R=I$ and the $L^p$ continuity of the Riesz transform, we get that $L^p(\R^d)\subset \Lambda^{0,p}(\R^d)$.
    \end{itemize}
\end{proof}

An interesting consequence of the previous result is that we can establish a relationship between the Lions-Calderón spaces and the classical Sobolev spaces. In fact, unlike the Besov spaces, the class of Bessel Potential spaces $H^{s,p}(\R^d)$ contain the (classical) Sobolev spaces $W^{k,p}(\R^d)$ when $s=k\in\N$ and $p\in(1,\infty)$, and the Lebesgue spaces $L^p(\R^d)$ when $s=0$ and $p\in(1,\infty)$. Consequently, it is natural to extend the definition of $\Lambda^{s,p}(\R^d)$ to $s=1$ by setting
\begin{equation}
    \Lambda^{1,p}(\R^d):=W^{1,p}(\R^d)
\end{equation}
with $D^1:=D$, where $D$ is the classical gradient.

On the other hand, from the point of view of the theory of interpolation, the Lions-Calderón spaces $\Lambda^{s,p}(\R^d)$, when $0<s<1$ and $1<p<\infty$, is the complex interpolated space $(L^p(\R^d),W^{1,p}(\R^d))_{[s]}$ (see \cite{lions1960Une} and \cite[Theorem 2.4.7]{triebel1983Theory}), while the Sobolev-Slobodeckij spaces $W^{s,p}(\R^d)$ is the real interpolated space $(L^p(\R^d),W^{1,p}(\R^d))_{s,p}$ (see \cite[Theorem 2.4.2]{triebel1983Theory}). For the definition of the spaces $W^{s,p}(\R^d)$, as well as some of their properties, we refer to \cite{diNezza2012Hitchhiker}.

Although in general the Lions-Calderón spaces $\Lambda^{s,p}(\R^d)$ do not coincide with the Sobolev-Slobodeckij spaces $W^{s,p}(\R^d)$, they are in a certain sense close.

\begin{proposition}[Contiguity between the Lions-Calderón spaces and the Sobolev-Slobodeckij spaces] \label{prop:contiguity_fractional_sobolev_and_lions_calderon}
Let $1<p<\infty$, $0<s<1$ and $0<\varepsilon\leq \min\{s, 1-s\}$, then
\begin{equation*}
    \Lambda^{s+\varepsilon,p}(\R^d)\subset W^{s,p}(\R^d)\subset\Lambda^{s-\varepsilon,p}(\R^d).
\end{equation*}
\end{proposition}
\begin{proof}
    See \cite[Theorem 3.2]{lions1961problemi}, which contains the more general case $s\in\R$ in the framework $H^{s,p}(\R^d)$.
\end{proof}

Similarly to the classical Sobolev spaces, the Lions-Calderón spaces inherit some natural properties, such as continuous embeddings with respect the fractional parameter $s$.

\begin{proposition}\label{prop:order_lions_calderon}
    Let $0\leq s<t\leq 1$ and $1<p<\infty$. Then, $\Lambda^{t,p}(\R^d)\subset \Lambda^{s,p}(\R^d)$, being the inclusion map continuous. Moreover, in the limit case $\Lambda^{1,p}(\R^d)= W^{1,p}(\R^d)$ there exists a constant $C>0$, independent of $0<s<1$, such that
    \begin{equation}\label{eq:estimate_Ds_wrt_classical_Sobolev}
        \|D^s u\|_{L^p(\R^d;\R^d)}\leq \frac{C}{s}\|u\|_{W^{1,p}(\R^d)}, \quad \forall u\in W^{1,p}(\R^d) \mbox{ with } 1\leq p<\infty.
    \end{equation}
\end{proposition}
\begin{proof}
    The inclusions among the Lions-Calderón spaces for $1<p<\infty$ are from \cite[Theorem~5]{calderon1961Lebesgue} and \cite[Prop. 3.3]{lions1961problemi}, where the notations $L^p_s(\R^d)$ and $H^{s,p}(\R^d)$ were used for $\Lambda^{s,p}(\R^d)$, respectively. The inequality \eqref{eq:estimate_Ds_wrt_classical_Sobolev} is from \cite[Proposition 2.7]{bellido2020gamma}
\end{proof}

\begin{proposition}[Fractional Gagliardo-Nirenberg inequalities for $\Lambda^{s,p}(\R^d)$]\label{GagliardoNirenbergLionsCalderon}
    Let $0\leq r\leq s \leq t\leq 1$ such that $s=\theta r+(1-\theta)t$ with $\theta\in[0,1]$. Consider also $p,p_1,p_2\in[1,\infty]$ such that
    \begin{equation}\label{eq:hypothesis_for_frac_gagliardo_nirenberg}
        \frac{1}{p}=\frac{\theta}{p_1}+\frac{1-\theta}{p_2}.  
    \end{equation}
    Then, there exists a positive constant $C=C(d,p_1,p_2,r,t,\theta)\geq 0$ such that for all $f\in \Lambda^{r,p_1}(\R^d)\cap\Lambda^{t,p_2}(\R^d)$, one has
    \begin{equation}\label{eq:gagliardo_nirenberg_inequality_full_lions_calderon_spaces}
        \|f\|_{\Lambda^{s,p}(\R^d)}\leq C\|f\|^\theta_{\Lambda^{r,p_1}(\R^d)}\|f\|^{1-\theta}_{\Lambda^{t,p_2}(\R^d)}.
    \end{equation}
\end{proposition}

\begin{proof}
    This is a simple application of \cite[Proposition~5.6]{brezis2018Gagliardo} and Proposition \ref{prop:distributional_characterization_lions_calderon_spaces} for the identification between $\Lambda^{s,p}(\R^d)$ and $H^{s,p}(\R^d)$.
\end{proof}

\begin{remark}\label{rem:Gagliardo-NirembergALaBrueCalziComiStefani}
    In \cite[Theorem~13]{brue2022AsymptoticsII}, it was proved that for $1<p_1=p_2=p<\infty$ and for $0\leq r\leq s\leq t\leq 1$, one has
    \begin{equation}
        \|D^s f\|_{L^p(\R^d;\R^d)}\leq c_{d,p}\|D^r f\|^{\frac{t-s}{t-r}}_{L^p(\R^d;\R^d)}\|D^t f\|^{\frac{s-r}{t-r}}_{L^p(\R^d;\R^d)}, \quad \mbox{ for all } f\in \Lambda^{t,p}(\R^d).
            \label{eq:gagliardo_nirenberg_inequality_fractional_gradients}
    \end{equation}
    This result is better than Proposition \ref{GagliardoNirenbergLionsCalderon} in the particular case $p_1=p_2=p$ because it provides an inequality just in terms of the fractional gradient and the constant $c_{d,p}$ does not depend on $r,s,t$. 
\end{remark}

To conclude this subsection, let us provide some properties of the $s$-fractional gradient in the context of the Lions-Calderón spaces. We start with the continuous dependence of $D^s$ with respect to the parameter $s$, provided that the function where we are computing these fractional gradients is sufficiently regular.

\begin{proposition}[Continuity $s\mapsto D^s$]\label{prop:continuous_dependence_fractional_gradiente_on_s}
	Let $f\in\Lambda^{s,p}(\R^d)$ with $0<s\leq1$ and $1\leq p<\infty$. Then, for any sequence $t\in(0,s]$ (not converging to $0$) we have that $t\mapsto D^t f$ is continuous in $L^p(\R^d;\R^d)$. Moreover, if we impose the strict condition $p>1$, then this continuity result can be extended to the case in which the sequence $t$ can converge to any value in $[0,s]$.
\end{proposition}
\begin{proof}
	See \cite[Theorem~34]{brue2022AsymptoticsII}.
\end{proof}

Lastly, we provide a similar characterization to the one in Proposition \ref{prop:charact_Ds_with_riesz} but for functions that satisfy the minimal regularity assumptions, that is, belong to the smallest Lions-Calderón space where the respective fractional gradient is well-defined.

\begin{proposition}\label{prop:characterization_fractional_gradient_for_functions_in_lions_calderon}
    Let $0<s<1$, $1<p<\infty$ and $u\in\Lambda^{s,p}(\R^d)$. Then $D^s u=D(I_{1-s} u)$, which means that, in these spaces, the notion of weak fractional gradient and the notion of distributional Riesz fractional gradient coincide.
\end{proposition}
\begin{proof}
    Let $\varphi\in C^\infty_c(\R^d)$. Using the weak definition of fractional gradient as well as the Proposition \ref{prop:charact_Ds_with_riesz}, we get that
    \begin{equation*}
        \int_{\R^d}{\varphi D^s u}\,dx=-\int_{\R^d}{u D^s \varphi}\,dx=-\int_{\R^d}{u (I_{1-s}D \varphi)}\,dx.
    \end{equation*}
    Noticing that
    \begin{equation*}
        \int_{\R^d}{\int_{\R^d}{\frac{|D\varphi(y)||u(x)|}{|y-x|^{d+s-1}}}\,dy}\,dx<+\infty,
    \end{equation*}
   by Fubini's theorem we can change the order of the integrals and so
    \begin{equation*}
        \int_{\R^d}{u (I_{1-s} D \varphi)}\,dx=\int_{\R^d}{(I_{1-s}u) D \varphi}\,dx.
    \end{equation*}
    Consequently, we conclude that $D^s u= D(I_{1-s}u)$, first in the sense of distributions and afterwords in $L^p(\R^d;\R^d)$.
\end{proof}
\begin{remark}\label{rem:dualityp-p´} By density, this equivalence in the definition of the fractional gradient implies the general duality property this operator $D^s$:
\begin{equation*}
        \int_{\R^d} {vD^s u}\,dx=-\int_{\R^d}{u D^s v}\,dx.
    \end{equation*}
for all $u\in\Lambda^{s,p}(\R^d)$, $v\in\Lambda^{s,p'}(\R^d)$ with $0<s<1<p<\infty$ and $p'=p/(p-1)$. Similarly, we notice that the duality between $D^s$ and $D^s\cdot=div_s$ in \eqref{eq:duality_frac_divergence_frac_gradient} still holds for all $\varphi \in\Lambda^{s,p}(\R^d)$ and all $\Phi\in\Lambda^{s,p'}(\R^d)^d$.

\end{remark}

\begin{proposition}\label{prop:fractionalGradientInTermsOfRieszPotentialAndFractionalGradient}
    Let $0<s\leq\sigma\leq 1$, $1<p<\infty$ and $u\in\Lambda^{\sigma,p}(\R^d)$. Then $D^s u=I_{\sigma-s}D^\sigma u$.
\end{proposition}
\begin{proof}
    Let $\varphi\in C^\infty_c(\R^d)$. Then, by the semigroup property of Riesz potentials,
    \begin{equation*}
        \int_{\R^d}{\varphi D^s u}\,dx=-\int_{\R^d}{u D^s \varphi}\,dx=-\int_{\R^d}{u D(I_{1-s} \varphi)}\,dx=-\int_{\R^d}{u D(I_{1-\sigma}I_{\sigma-s} \varphi)}\,dx.
    \end{equation*}
    Since $\varphi\in C^\infty_c(\R^d)$, then $I_{\sigma-s}\varphi\in L^p(\R^d)$. Moreover, for every $\psi\in C^\infty_c(\R^d)$ we have
    \begin{multline*}
        \int_{\R^d}{\psi D^\sigma(I_{\sigma-s}\varphi)}\,dx=-\int_{\R^d}{(I_{\sigma-s}\varphi) D^\sigma \psi}\,dx=-\int_{\R^d}{\varphi (I_{\sigma-s}D^\sigma \psi)}\,dx\\
        =-\int_{\R^d}{\varphi (I_{\sigma-s}(I_{1-\sigma}D\psi))}\,dx=-\int_{\R^d}{\varphi (I_{1-s}D\psi)}\,dx=-\int_{\R^d}{\varphi D^s \psi}=\int_{\R^d}{\psi D^s\varphi}\,dx.
    \end{multline*}
    which means that $D^\sigma(I_{\sigma-s}\varphi)=D^s\varphi\in L^p$, and therefore $I_{\sigma-s}\varphi\in\Lambda^{\sigma,p}(\R^d)$.
    Using Proposition \ref{prop:characterization_fractional_gradient_for_functions_in_lions_calderon}, we have $D(I_{1-\sigma}I_{\sigma-s} \varphi)=D^\sigma(I_{\sigma-s}\varphi)$. Consequently, from the properties of the fractional gradient, we are able to conclude that
    \begin{equation*}
        \int_{\R^d}{\varphi D^s u}\,dx=-\int_{\R^d}{u D(I_{1-\sigma}I_{\sigma-s} \varphi)}\,dx=-\int_{\R^d}{u D^\sigma(I_{\sigma-s} \varphi)}\,dx=\int_{\R^d}{I_{\sigma-s} \varphi D^\sigma u}\,dx.
    \end{equation*}
    Finally, since 
    \begin{equation*}
        \int_{\R^d}{\int_{\R^d}{\frac{|D^\sigma u(x)||\varphi(y)|}{|y-x|^{d+s-\sigma}}}\,dy}\,dx\leq \|D^s u\|_{L^p(\R^d;\R^d)}\|I_{\sigma-s}|\varphi|\|_{{L^p}'(\Omega)}\leq C\|D^s u\|_{L^p(\R^d;\R^d)}\|\varphi\|_{L^q(\R^d)}<+\infty,
    \end{equation*}
    for $q=\frac{Np}{N(p-1)+(\sigma-s)p}$, we can apply Fubini's theorem to obtain
    \begin{equation*}
        \int_{\R^d}{\varphi D^s u}\,dx=\int_{\R^d}{I_{\sigma-s} \varphi D^\sigma u}\,dx=\int_{\R^d}{\varphi (I_{\sigma-s}D^\sigma u)}\,dx,
    \end{equation*}
    concluding the proof.
\end{proof}

\subsection{A fractional framework for the homogeneous Dirichlet condition}
As it was observed in Proposition \ref{prop:distributional_characterization_lions_calderon_spaces}, $C^\infty_c(\R^d)$ is dense in $\Lambda^{s,p}(\R^d)$. This prompts us to define the Lions-Calderón spaces $\Lambda^{s,p}_0(\Omega)$ for arbitrary open sets $\Omega$ also by density of smooth functions in $\R^d$ with compact support in $\Omega$, i.e. $C^\infty_c(\Omega)$. Those spaces are suited to fractional Dirichlet problems.

\subsubsection{The space $\Lambda^{s,p}_0(\Omega)$ with arbitrary $\Omega$}
\begin{definition}
	Let $\Omega\subset\R^d$ be a any open set, and consider $0\leq s\leq 1$ and $1<p<\infty$. We define the space 
\begin{equation}\label{eq:Lions_calderon_compact_support}
    \Lambda^{s,p}_0(\Omega)= \overline{C^\infty_c(\Omega)}^{\|\cdot\|_{\Lambda^{s,p}(\R^d)}}
 \end{equation}
 and we endow this space with the norm of $\Lambda^{s,p}(\R^d)$. We also denote the dual space of $\Lambda^{s,p}_0(\Omega)$ by $\Lambda^{-s,p'}(\Omega)$, where $p'=\frac{p}{p-1}$, and denote their duality by $\langle \cdot, \cdot\rangle=\langle \cdot, \cdot\rangle_s$.
\end{definition}

We observe that, by construction,  we have 
\begin{equation}\label{eq:space_subspace_zero_outside_omega}
    \Lambda^{s,p}_0(\Omega)\subset \{f\in \Lambda_0^{s,p}(\R^d):\, f=0 \mbox{ on }\Omega^c\}.
\end{equation}
Due to this inclusion, it makes sense to compute the weak fractional gradient $D^s$ of any function $u\in\Lambda^{s,p}_0(\Omega)$. Moreover, it is easy to show, that this fractional gradient $D^s u$ coincides with the $L^p$-limit of $D^s u_n$ where $u_n$ are the approximating functions in $C^\infty_c(\Omega)$. In addition, we also have that for any $u\in\Lambda^{s,p}_0(\Omega)$ and $v\in\Lambda^{s,p'}_0(\Omega)$, 
\begin{equation*}
    \int_{\R^d}{vD^s u}\,dx=\lim_{n\to\infty}{\int_{\R^d}{vD^s u_n}\,dx}=-\lim_{n\to\infty}{\int_{\R^d}{u_nD^s v}\,dx}=-\int_{\R^d}{uD^s v}\,dx.
\end{equation*}

When $s=0$, for any open subset $\Omega\subset\R^d$, we notice that 
\begin{equation}\label{eq:characterization_lions_calderon_compact_support_s_eq_0}
    \Lambda^{0,p}_0(\Omega)=L^p_0(\Omega)=\{f\in L^p(\R^d):\, f=0 \mbox{ on }\Omega^c\}.
\end{equation}

Among the important properties of these spaces, we recall the generalization of the classical Sobolev inequalities to this fractional framework. We define the fractional Sobolev exponent
\begin{equation}
	p^*_s=\frac{dp}{d-sp}
\end{equation}
whenever $sp<d$. Notice that when $s=0$, $p_0^*=p$ for all $p\in(1,\infty)$. 

\begin{proposition}[Sobolev embeddings for $\Lambda^{s,p}_0(\Omega)$]\label{prop:sobolev_embeddings_lions_calderon}
	Let $0< s\leq 1$, $1<p<\infty$. Then,
\begin{itemize}
    \item[1)] if $sp<d$, we have $\Lambda^{s,p}_0(\Omega)\subset L^q(\Omega)$ for $p\leq q\leq p^*_s$. Moreover, when $q=p^*_s$ we have
    \begin{equation} \label{eq:subcritical_fractional_sobolev_inequality_with_Ds}
        \|f\|_{L^{p^*_s}(\Omega)}\leq C\|D^s f\|_{L^p(\R^d;\R^d)};
    \end{equation}
    \item[2)] if $sp=d$, then $\Lambda^{s,p}_0(\Omega)\subset L^q(\Omega)$ for every $q\in[p,+\infty)$; and
    \item[3)] if $sp>d$, then we have $\Lambda^{s,p}_0(\Omega)\subset C^{0,s-\frac{d}{p}}(\overline{\Omega})$ with the estimate
    \begin{equation}
        |f(x)-f(y)|\leq C|x-y|^{s-\frac{N}{p}}\|D^s f\|_{L^p(\R^d)}.
        \label{eq:fractional_morrey_inequality_with_Ds}
    \end{equation}
    \item[4)] for $t<s$ and $1<p\leq q\leq \frac{dp}{d-(s-t)p}$, we have $\Lambda^{s,p}_0(\Omega)\subset \Lambda^{t,q}_0(\Omega)$.
\end{itemize}
\end{proposition}
\begin{proof}
Throughout this proof we are using the embedding $\Lambda^{s,p}_0(\Omega)\subset \Lambda_0^{s,p}(\R^d)$. 
\begin{itemize}
    \item[1)] The inequality \eqref{eq:subcritical_fractional_sobolev_inequality_with_Ds} was proven in \cite[Theorem~1.8]{shieh2015On}, which implies that $\Lambda^{s,p}(\R^d)\subset L^{p^*_s}(\R^d)$. Then, the embedding $\Lambda^{s,p}_0(\Omega)\subset L^q(\R^d)$ for $p\leq q\leq p^*_s$ follows from Proposition \ref{prop:order_lions_calderon}, by noticing that
    \begin{equation*}
        \Lambda^{s,p}_0(\Omega)\subset \Lambda^{s,p}(\R^d)\subset\Lambda^{s-\varepsilon,p}(\R^d)\subset L^{p^*_{s-\varepsilon}}(\R^d)
    \end{equation*}
    with $p\leq p^*_{s-\varepsilon}<p^*_s$ when $0<\varepsilon\leq s$. The limit case $s=0$, when $p^*_s=p$, corresponds to the continuity in $L^p$ of the Riesz transform \cite[Corollary 5.2.8]{grafakos2014Classical}. 
    \item[2)] This follows by a similar argument to the last part of 1).
    \item[3)] It was proved in \cite[Theorem~1.11]{shieh2015On}.
    \item[4)] We start by observing that $q=\frac{dp}{d-(s-t)p}$ was already proven in \cite[Theorem~6]{calderon1961Lebesgue}. To extend the inclusions for $p\leq q\leq \frac{dp}{d-(s-t)p}$, we write $s=t+\delta$ for some $\delta>0$, and we consider $0\leq \varepsilon\leq\delta$. In this case,
    \begin{equation*}
        \Lambda^{s,p}_0(\Omega)\subset \Lambda^{t+\delta-\varepsilon,p}(\R^d)\subset\Lambda^{t,\frac{dp}{d-(\delta-\varepsilon)p}}(\R^d).
    \end{equation*}
    Finally, since for every $u\in \Lambda^{s,p}_0(\Omega)\subset\Lambda^{t,\frac{dp}{d-(\delta-\varepsilon)p}}(\R^d)$, then there exists a sequence $\{u_n\}\subset C^\infty_c(\Omega)$ such that $u_n\to u$ in $\Lambda^{s,p}(\R^d)\subset\Lambda^{t,\frac{dp}{d-(\delta-\varepsilon)p}}(\R^d)$, and consequently $u\in\Lambda^{t,\frac{dp}{d-(\delta-\varepsilon)p}}_0(\Omega)$.
\end{itemize}
\end{proof}

If instead of considering just the inclusion operator like in Sobolev embeddings, we also restrict the domain of the functions to one subdomain with finite measure, then the resulting operator is compact. More precisely, we have the following result.

\begin{proposition}[Local Rellich-Kondrachov for $\Lambda^{s,p}_0(\Omega)$]\label{prop:local_rellich_kondrachov}
    Let $0<s\leq 1$, $p\in(1,\infty)$ and $\Omega\subset \R^d$, where $\Omega$ is any open set. If we also consider the open set with finite measure $\omega\subset\Omega$ and the compact set $K\subset \Omega$, we have following compact embeddings (denoted by $\Subset$) :
    \begin{itemize}
        \item[1)] if $sp<d$, then $\Lambda^{s,p}_0(\Omega)\Subset L^q(\omega)$ for $q<p^*_s$;
        \item[2)] if $sp=d$, then $\Lambda^{s,p}_0(\Omega)\Subset L^q(\omega)$ for every $q<+\infty)$; and
        \item[3)] if $sp>d$, then $\Lambda^{s,p}_0(\Omega)\Subset C^{0,\beta}(K)$ for $\beta<s-\frac{d}{p}$.
    \end{itemize}
\end{proposition}
\begin{proof}
    We observe that, since $\Lambda^{s,p}_0(\Omega)\subset\Lambda^{s,p}(\R^d)$, we just need to prove this result for $\Omega=\R^d$. The idea is to use the contiguity between the Lions-Calderón spaces and the Sobolev-Slobodeckij spaces, and then apply the compactness results that are already known for the latter. 
    
    In fact, if $sp\leq d$, we can consider $\varepsilon\in(0,s)$ and then due to Proposition \ref{prop:contiguity_fractional_sobolev_and_lions_calderon} and \cite[Theorem~2.1]{delTeso2020Estimates}, we have that
    \begin{equation}
        \Lambda^{s,p}(\R^d)\subset W^{s-\varepsilon,p}(\R^d)\Subset L^q(\omega), \quad \mbox { with } q<p^*_{s-\varepsilon}.
    \end{equation}
    Since $\varepsilon$ ranges in $(0,s)$, then $p^*_{s-\varepsilon}$ ranges in $(p,p^*)$. Then the conclusion follows using the fact that $\omega$ has finite measure.

    We can use the same arguments for $sp>d$ but with $\varepsilon\in(0,s-\frac{d}{p})$ and with the inclusions
    \begin{equation}
        \Lambda^{s,p}(\R^d)\subset W^{s-\varepsilon,p}(\R^d)\Subset C^{0,\beta}(K), \quad \mbox { with } \beta<s-\varepsilon-\frac{d}{p}.
    \end{equation}
\end{proof}

Similarly to the classical Sobolev spaces, the Lions-Calderón spaces are closed for the composition with uniformly Lipschitz functions. In fact, as an immediate consequence of the following theorem, we have, in particular,
\begin{equation}
    f^+, f^- \mbox{ and } |f|\in \Lambda^{s,p}_0(\Omega), \mbox{ whenever } f\in \Lambda^{s,p}_0(\Omega).
\end{equation}

\begin{theorem}[Composition with Lipschitz functions]
    Let $\Omega\subset\R^d$ be any open set, $s\in(0,1)$ and $p\in(1,\infty)$. If $f\in \Lambda^{s,p}_0(\Omega)$ and $\varphi\in \mathrm{Lip}(\R)$ with $\varphi(0)=0$ and $\|\varphi'\|_{L^\infty(\R)}\leq \lambda$, then $\varphi\circ f\in \Lambda^{s,p}_0(\Omega)$.
\end{theorem}
\begin{proof}
    Let $f_n\in C^\infty_c(\Omega)$ be such that $f_n\to f$ in $\Lambda^{s,p}_0(\Omega)$. Using Strichartz's characterization of $H^{s,p}(\R^d)=\Lambda^{s,p}(\R^d)$, \cite[Section 2.3]{strichartz1967Multipliers} or \cite[Theorem 4.8.1]{adams1996Function}, we know that a function $g\in \Lambda^{s,p}(\R^d)$ if and only if $g\in L^p(\R^d)$ and
    \begin{equation*}
        S_sg(x)=\left(\int_0^\infty{\left(\int_{B_1(0)}{|g(x+ry)-g(x)|}\,dy\right)^2}\frac{dr}{r^{1+2s}}\right)^{1/2}\in L^p(\R^d).
    \end{equation*}
    Moreover, the norm of $g$ in $\Lambda^{s,p}(\R^d)$ is equivalent to $\|g\|_{L^p(\R^d)}+\|S_s g\|_{L^p(\R^d)}$.
    
    We notice that $\varphi\circ f_n\in L^p(\R^d)$, from the known properties of $L^p$ functions. We have also $S_s(\varphi\circ f_n)\in L^p(\R^d)$ since $\|S_s(\varphi(f_n))\|_{L^p(\R^d)}\leq \lambda\|S_s f_n\|_{L^p(\R^d)}$, by taking the $L^p$ norm on both sides of the pointwise estimate
    \begin{multline*}
        S_s(\varphi(f_n))(x)=\left(\int_0^\infty{\left(\int_{B_1(0)}{|\varphi(f_n(x+ry))-\varphi(f_n(x))|}\,dy\right)^2}\frac{dr}{r^{1+2s}}\right)^{1/2}\\
        \leq \lambda\left(\int_0^\infty{\left(\int_{B_1(0)}{|f_n(x+ry)-f_n(x)|}\,dy\right)^2}\frac{dr}{r^{1+2s}}\right)^{1/2}=\lambda S_s f_n(x).
    \end{multline*}
   Therefore $\varphi\circ f_n\in \Lambda^{s,p}(\R^d)$. Since $\mathrm{supp}\{\varphi\circ f_n\}\subset \mathrm{supp}\{f_n\}\subset \Omega$, then $\varphi\circ f_n\in \Lambda^{s,p}_0(\Omega)$. Taking the limit as $n\to \infty$ and using Vitalli's convergence theorem and the above inequality, we conclude $\varphi\circ f\in \Lambda^{s,p}_0(\Omega)$.
\end{proof}
\begin{remark}
    Although this result was first mentioned, without proof, in \cite{shieh2018On} for the case $\Omega=\R^d$, it was shown in \cite{campos2021Lions} for the special cases $\varphi(t)=|t|$ and the truncation function $\varphi_k(t)=(t\wedge k)\vee(-k)$. Moreover, in \cite{campos2021Lions} it was also proved that the truncation $T_kf=\varphi_k\circ f$ for a function $f\in \Lambda^{s,p}(\R^d)$ is such that  $T_kf\to f$ in $\Lambda^{s,p}(\R^d)$ as $k\to\infty$.
\end{remark}

The definition \eqref{eq:Lions_calderon_compact_support} and the interpretation \eqref{eq:space_subspace_zero_outside_omega} of $\Lambda^{s,p}_0(\Omega)$  as subspaces of $\Lambda^{s,p}(\R^d)$ yields the following interesting characterization of these spaces: 
\begin{theorem}[Netrusov's theorem]\label{thm:netrusovsTheorem}
    Let $\Omega\subset\R^d$ be any open set, $s\in(0,1)$ and $p\in(1,\infty)$, and consider $f\in\Lambda^{s,p}(\R^d)$. Then, we have $f\in \Lambda^{s,p}_0(\Omega)$ if and only if, for every $\varepsilon>0$ there exists a function $\eta$ such that $\eta=0$ on a neighborhood of $\Omega^c$, $0\leq \eta\leq 1$, and $\|f-\eta f\|_{\Lambda^{s,p}(\R^d)}<\varepsilon$.
\end{theorem}
\begin{proof}
    See \cite[Theorem~10.1.1]{adams1996Function}.
\end{proof}

A useful property for applications to the theory of fractional differential equations is the characterization of $\Lambda^{-s,p'}(\Omega)$ using fractional partial derivatives. Consider the operator
\begin{equation}
    \begin{aligned}
        \Pi_s:\Lambda^{s,p}_0(\Omega)&\to L_0^p(\Omega)\times L^p(\R^d;\R^d)=:\L^p(\Omega)\\
        v&\mapsto (v, D^s v).
    \end{aligned}
    \label{eq:def_of_pi_s}
\end{equation}
\begin{theorem}[Characterization of $\Lambda^{-s,p'}(\Omega)$]\label{CharacterizationDualInTermsOfFractionalDerivatives}
Let $\Omega\subset\R^d$ be any open set and let $s\in (0,1)$, $p\in(1,\infty)$ and $p'=\frac{p}{p-1}$. If $F\in \Lambda^{-s,p'}(\Omega)$, then there exist functions $f_0\in {L^p}'(\Omega)$ and $f_1,...,f_d\in {L^p}'(\R^d)$ such that
\begin{equation*}
    \langle F,g\rangle=\int_{\Omega}{f_0 g}\,dx+\sum_{j=1}^d{\int_{\R^d}{f_j\frac{\partial^s g}{\partial x_j^s}}\,dx} \quad \forall g\in\Lambda^{s,p}_0(\Omega).
\end{equation*}
\end{theorem}
\begin{proof}
    Let $\Pi_s:\Lambda^{s,p}_0(\Omega)\to L^p(\Omega)\times L^p(\R^d, \R^d)$ defined as $\Pi_s(g)=(g, D^s g)$ for any $g\in \Lambda^{s,p}_0(\Omega)$. Note that 
    \begin{equation*}
        \|\Pi_s g\|_{\L^p(\Omega)}=\|g\|_{\Lambda^{s,p}_0(\R^d)}
    \end{equation*}
    which means that $\Pi_s$ is an isometry of $\Lambda^{s,p}_0(\R^d)$ onto a subspace $W\subset \L^p(\Omega)$. Then we define the linear functional $F^*$ on $W$ as $\langle F^*, \Pi_s g\rangle=\langle F,g\rangle$. By the isometry between $\Lambda^{s,p}_0(\Omega)$ and $W$ we obtain that $\|F^*\|_{W'}=\|F\|_{\Lambda^{-s,p'}(\Omega)}$. Consequently, using Hahn-Banach theorem, there exists an extension $\Tilde{F}$ of $F^*$ to all $\L^p(\Omega)$ such that $\|F^*\|_{W'}=\|\Tilde{F}\|_{(\L^p(\Omega))'}$. Then, by the Riesz representation theorem, there exist $f_0\in {L^p}'(\Omega)$ and $f_1,..., f_n\in {L^p}'(\R^d)$, such that
    \begin{equation*}
        \langle \Tilde{F}, h\rangle=\int_\Omega{f_0 h_0}\,dx+\sum_{j=1}^d{\int_{\R^d}{f_j h_j}\,dx}, \quad \forall h=(h_0, h_1,..., h_d)\in L^p(\Omega)\times L^p(\R^d,\R^d).
    \end{equation*}
    Hence, for any $g\in \Lambda^{s,p}_0(\Omega)$, we conclude
    \begin{equation*}
        \langle F,g\rangle=\langle F^*, \Pi_s g\rangle=\langle\Tilde{F}, \Pi_s g\rangle=\int_{\Omega}{f_0 g}\,dx+\sum_{j=1}^d{\int_{\R^d}{f_j\frac{\partial^s g}{\partial x_j^s}}\,dx}.
    \end{equation*}
\end{proof}
\begin{remark}\label{rem:characterization_dual}
    Due to this characterization, we may write any element $F\in\Lambda^{-s,p'}(\Omega)$ in the form
    \begin{equation*}
        F=f_0-D^s\cdot \boldsymbol{f}
    \end{equation*}
    with $f_0\in {L^p}'(\Omega)$ and $\boldsymbol{f}=(f_1,...,f_d)\in {L^p}'(\R^d;\R^d)$, similarly to the classical Sobolev case $s=1$, cf. {\cite[Theorem~3.8]{adams1975Sobolev}}.
\end{remark}

\begin{remark}\label{rem:caracterization_of_dual_wrt_sobolev_dual}
    It is also possible to write $F=f_0-D^s\cdot \boldsymbol{f}\in\Lambda^{-s,p'}(\Omega)$ with $f_0\in L^{p^*_s}(\Omega)$ and $\boldsymbol{f}\in {L^p}'(\R^d;\R^d)$. This is a simple consequence of the fractional Sobolev inequality \eqref{eq:subcritical_fractional_sobolev_inequality_with_Ds}, which allows us to use the application $\Lambda^{s,p}_0(\Omega)\ni g\mapsto (g,D^s g)\in L^{p^*_s}(\Omega)\times L^p(\R^d, \R^d)$ in the proof of Theorem \ref{CharacterizationDualInTermsOfFractionalDerivatives}.
\end{remark}

\subsubsection{The space $\Lambda^{s,p}_0(\Omega)$ with $\Omega$ bounded}
So far, we have been providing some properties that hold for any open set $\Omega\subset\R^d$, but now we are going to focus on the case in which $\Omega$ is bounded. We start by providing a characterization of these spaces when the boundary of $\Omega$ is sufficiently smooth.

\begin{proposition}[Extension by $0$]
    Let $\Omega\subset\R^d$ be a bounded and open subset with Lipschitz boundary, $0< s\leq 1$ and $1<p<\infty$. Then,
    \begin{equation}
        \Lambda^{s,p}_0(\Omega)=\{f\in\Lambda^{s,p}(\R^d):\, \mathrm{supp}f\subset\overline{\Omega}\}.
        \label{eq:characterization_lions_calderon_compact_support_s_general}
    \end{equation}
\end{proposition}
\begin{proof}
    Since this is a property about the elements of $\Lambda^{s,p}_0(\Omega)$, we can use the identification $\Lambda^{s,p}(\R^d)=H^{s,p}(\R^d)$ and then just apply \cite[Remark~2.7]{jerison1995Inhomogeneous}.
\end{proof}

There are also some results, similar to those valid for arbitrary $\Omega$ that can be improved when $\Omega$ is bounded. One such example is the Proposition \ref{prop:sobolev_embeddings_lions_calderon} concerning the Sobolev embeddings.

\begin{proposition}[Sobolev embeddings for $\Lambda^{s,p}_0(\Omega)$]\label{prop:sobolev_embeddings_compact_support}
    Let $\Omega\subset\R^d$ be any bounded open set and consider $s\in[0,1]$ and $p\in(1,+\infty)$. Then,
    \begin{itemize}
        \item[1)] if $sp<d$, we have that $\Lambda^{s,p}_0(\Omega)\subset L^q(\Omega)$ for every $q\in[1,p^*_s]$ with 
        \begin{equation}
            \|f\|_{L^q(\Omega)}\leq C\|D^s f\|_{L^p(\R^d;\R^d)};
            \label{eq:sobolev_subcritical_compact_support}
        \end{equation}
        \item[2)] $sp=d$ implies that $\Lambda^{s,p}_0(\Omega)\subset L^q(\Omega)$ for every $q< +\infty)$;
        \item[3)] $sp>d$ implies that $\Lambda^{s,p}_0(\Omega)\subset C^{0,\beta}(\overline{\Omega})$ for any $\beta\leq s-d/p$. In particular, $\Lambda^{s,p}_0(\Omega)\subset L^q(\Omega)$ for all $q\leq +\infty$;
        \item[4)] when $t\leq s$ and $\frac{1}{p}-\frac{s}{d}\leq\frac{1}{q}-\frac{t}{d}$, we have $\Lambda^{s,p}_0(\Omega)\subset \Lambda^{t,q}_0(\Omega)$.
    \end{itemize}
\end{proposition}
\begin{proof}
    The items 1), 2), 3) and 4) are a simple application of the Proposition \ref{prop:sobolev_embeddings_lions_calderon} together with some well known properties of the $L^p$ and $C^{0,\alpha}$ spaces.
\end{proof}

The inequality \eqref{eq:sobolev_subcritical_compact_support} is important as it also implies a fractional Poincaré inequality. In fact, in \cite{bellido2020gamma}, it was raised the question of how does the constant $C$ in this inequality depends on $s$ and it was proved that one could take this constant to be inversely proportional to $s$, at least when $q=p$. To be more precise, let us now recall their result (with a small improvement).

\begin{theorem}[Fractional Poincaré's inequality]\label{FracPoincareInequality}
    Let $p\in[1,\infty)$ and $s\in(0,1)$ and let $\Omega\subset\R^d$ be a bounded open subset. Then, there exists a sufficiently large real number $R>1$ for which $\Omega\subset B_R(0)$, and a positive constant $C_P>0$ depending only on $p$, $d$, $\Omega$ and $R$, such that for all open sets $\Omega_1\supset B_{2R}(0)$ ($\Omega_1$ can be unbounded) and for all $f\in \Lambda^{s,p}_0(\Omega)$, we have 
    \begin{equation}\label{eq:poincare_inequality}
        \|f\|_{L^p(\Omega)}\leq \frac{C_P}{s}\|D^s f\|_{L^p(\Omega_1;\R^d)}.
    \end{equation}
\end{theorem}
\begin{proof}
    Most of the proof can be found in \cite[Theorem~2.9]{bellido2020gamma}. In fact, when $p\in(1,\infty)$, one only needs to change the last estimate of (13) of the same article, with a $L^p(\Omega_1;\R^d)$-norm instead of doing it with the $L^p(\R^d;\R^d)$-norm.
    Similarly, when $p=1$, one follows the same lines of the proof but with the estimates
    \begin{align*}
        \int_{\Omega}{\left(\int_{\{|y|<2R\}}{\frac{|D^s f(y)|}{|x-y|^{N-s}}}\,dy\right)}\,dx
        &\leq \int_{\{|y|<2R\}}{|D^s f(y)|\left(\int_{\Omega}{\frac{1}{|x-y|^{N-s}}}\,dx\right)}\,dy\\
        &\leq \frac{C(N)}{s}R\|D^s f\|_{L^1(\Omega_1;\R^n)}.
    \end{align*}
    and
    \begin{align*}
        \int_{\Omega}{\left(\int_{\{|y|>2R\}}{\frac{|D^s f(y)|}{|x-y|^{N-s}}}\,dy\right)}\,dx
        &\leq\int_{\{|y|>2R\}}{\left(\int_{\Omega}{\frac{1}{|x-y|^{N-s}}\left(\int_{\Omega}{\frac{f(z)}{|y-z|^{N+s}}}\,dz\right)}\,dx\right)}\,dy\\
        &\leq\int_{\{|y|>2R\}}{\left(\int_{\Omega}{\frac{C(N)}{|y|^{N+s}|y|^{N-s}}\|f\|_{L^1(\Omega)}}\,dx\right)}\,dy\\
        &\leq C(N)R^{-N}|\Omega|\|f\|_{L^1(\Omega)},
    \end{align*}
    where for the last estimate we used the fact that $|y|/2\leq |y-x|$ when  $x\in \Omega$ and $y\in \Omega_1^c$.
\end{proof}

\begin{remark}\label{rem:uniform_boundness_poincare_constant_p2}
    It is interesting to see that the Poincaré constant obtained in the previous theorem blows-up as $s\searrow 0$. On the other hand, from part iii) of Proposition \ref{prop:distributional_characterization_lions_calderon_spaces}, we know that $\|f\|_{L^p(\Omega)}\leq \|D^0 f\|_{L^p(\R^d;\R^d)}$. This raises the question whether it is possible or not to obtain a uniform bound to the best Poincaré constants independent of $s\in(0,1)$. This is known to be possible for $p=2$ due to the results of \cite{feulefack2022Small}, where it is shown, in particular, that the best Poincaré constant $c_{2,s}=1/\sqrt{\lambda_s} \to 1$ as $s\to 0^+$, where $\lambda_s=\inf\left\{\frac{\|D^s w\|_{L^2(\R^d;\R^d)}}{\|w\|_{L^2(\Omega)}}:\, w\in H^s_0(\Omega)\right\}>0$ is the first eigenvalue of the Dirichlet problem for $(-\Delta)^s=-D^s\cdot D^s$.
\end{remark}

This result has several implications. The first one is that the $L^p$-norm of the $s$-fractional gradient is equivalent to the usual norm of $\Lambda^{s,p}_0(\Omega)$. In fact, we have:
\begin{corollary}[An equivalent norm for $\Lambda^{s,p}_0(\Omega)$]\label{EquivalenceOfNorms}
    Let $s\in[0,1]$ and $p\in[1,\infty)$, and let $\Omega\subset\R^d$ be an open and bounded subset. Then there exists a constant $C=C(s,p,d,\Omega)$ such that
    \begin{equation*}
        \|D^s f\|_{L^p(\R^d;\R^d)}\leq \|f\|_{\Lambda^{s,p}_0(\Omega)}\leq C\|D^s f\|_{L^p(\R^d;\R^d)}, \quad \mbox{ for all } f\in\Lambda^{s,p}_0(\Omega).
    \end{equation*}
\end{corollary}
\begin{remark}
    Notice that it also possible to prove that $\|D^s f\|_{L^p(\Omega_1;\R^d)}$ is equivalent to $\|f\|_{L^p(\Omega)}+\|D^s f\|_{L^p(\R^d;\R^d)}$. The only extra step that one needs to take care of is the one of finding an estimate for $\|D^s f\|_{L^p(\Omega_1^c;\R^d)}$ in terms of $\|f\|_{L^p(\Omega)}$, which can be done using some arguments that are present in the proof of \cite[Theorem~2.9]{bellido2020gamma}. This means one can consider the norm of a function in $\Lambda^{s,p}_0(\Omega)$ just knowing the $D^s$ of that function on a certain bounded domain $\Omega_1$, which can be of interest for numerical methods.
\end{remark}

Moreover we may compare the norms of the embedding $\Lambda^{s,p}_0(\Omega)\subset\Lambda^{t,p}_0(\Omega)$, with $1<p<\infty$ fixed, with an explicit dependence with respect to the fractional parameter $t<s$. 
\begin{corollary}\label{cor:order_between_Ds}
    Let $0<t<s\leq 1$. Let $\Omega\subset\R^d$ be a bounded and open set. Then, $\Lambda^{s,p}_0(\Omega)\subset \Lambda^{t,p}_0(\Omega)$ with the inequality
    \begin{equation}\label{eq:order_between_Ds_explicit_constant}
        \|D^t u\|_{L^p(\R^d;\R^d)}\leq \frac{C}{t^{1+\frac{1}{p}}}\|D^s u\|_{L^p(\R^d;\R^d)}.
    \end{equation}
    for some $C=C(\Omega, d,p)>0$.
\end{corollary}
\begin{proof}
    This is a refinement of the proof of \cite[Proposition 4.1]{bellido2020gamma}. In fact, if we proceed as in the proof of \cite[Proposition 4.1]{bellido2020gamma}, and making use of the precise constants, in particular not estimating in \cite[eq. (22)]{bellido2020gamma} $\frac{\mu_{d, 1-s+t}}{d-s+t}=\frac{1}{\gamma(t-s)}\leq C(d)$, we get 
    \begin{equation}
        \|D^t u\|_{L^p(\R^d;\R^d)}\leq \left(\frac{3\sigma_{d-1}R}{(s-t)}\frac{\mu_{d, 1-s+t}}{d-s+t}+\mu_{d,t}\frac{C_P}{s}\left(\frac{2^{d+s}\sigma_{d-1}|\Omega|}{d+t}+|\Omega|^{1/p'}\left(\frac{\sigma_{d-1}}{t}\right)^{1/p}\right)\right)\|D^s u\|_{L^p(\R^d;\R^d)},
    \end{equation}
    where $C_P=C_P(\Omega,d,p)>0$ is the fractional Poincaré's constant, $\sigma_{d-1}$ is the measure of the $d-1$ dimensional sphere $\partial B(0,1)$ and $R>1$ is sufficiently large number such that $\Omega\subset B(0,R)$.
    Now, we estimate in such a way that we can remove the dependence of the constant with respect to $t$. Indeed, since $0<1-s+t<1$ and $t\leq d-1+t$ then \cite[Lemma 2.4]{bellido2020gamma} yields
    \begin{equation}
        \frac{1}{(t-s)}\frac{\mu_{d, 1-s+t}}{d-s+t}\leq\frac{1}{d-1+t}\frac{\mu_{d,1-s+t}}{1-(1-s+t)}\leq \frac{1}{s}\sup_{\tau\in[-1,1)}{\frac{\mu_{d,\tau}}{1-\tau}}<\infty.
    \end{equation}
    Consequently,
    \begin{equation}
        \begin{aligned}
            \|D^t u\|_{L^p(\R^d;\R^d)}&\leq \left(\frac{3\sigma_{d-1}R}{t}\sup_{\tau\in[-1,1)}{\frac{\mu_{d,\tau}}{1-\tau}}+\mu_{d,t}\frac{C_P}{t}\left(\frac{2^{d+s}\sigma_{d-1}|\Omega|}{d+t}+|\Omega|^{1/p'}\left(\frac{\sigma_{d-1}}{t}\right)^{1/p}\right)\right)\|D^t u\|_{L^p(\R^d;\R^d)}\\
            &\leq \left(\frac{C_1}{t}+\frac{C_2}{t}+\frac{C_3}{t^{1+\frac{1}{p}}}\right)\|D^s u\|_{L^p(\R^d;\R^d)}\leq \frac{C_4}{t^{1+\frac{1}{p}}}\|D^s u\|_{L^p(\R^d;\R^d)}.
        \end{aligned}
    \end{equation}
    for some constants $C_1>0$ depending on $\Omega$ and $d$, and constants $C_2, C_3,C_4>0$, depending only on $\Omega$, $d$ and $p$.
\end{proof}

\begin{corollary}[Characterization of $\Lambda^{-s,p'}(\Omega)$, when $\Omega$ is bounded] \label{cor:characterization_dual_with_Omega_bounded}
    Let $\Omega$ be a bounded subset of $\R^d$ and let $s\in (0,1)$, $p\in(1,\infty)$ and $p'=\frac{p}{p-1}$. If $F\in \Lambda^{-s,p'}(\Omega)$, then there exist functions $f_1,...,f_d\in {L^p}'(\R^d)$ such that
    \begin{equation*}
        \langle F,g\rangle=\sum_{j=1}^d{\int_{\R^d}{f_j\frac{\partial^s g}{\partial x_j^s}}\,dx} \quad \forall g\in\Lambda^{s,p}_0(\Omega).
    \end{equation*}
\end{corollary}
\begin{proof}
    By using the equivalent norm $f\mapsto \|D^s f\|_{L^p(\R^d;\R^d)}$ in $\Lambda^{s,p}_0(\Omega)$, the result follows by applying the same arguments as those used in the proof of Proposition \ref{CharacterizationDualInTermsOfFractionalDerivatives}.
\end{proof}
\begin{remark}
    As in the general case, when $\Omega$ is bounded, any $F\in \Lambda^{-s,p'}(\Omega)$ may also be given as in Proposition \ref{CharacterizationDualInTermsOfFractionalDerivatives}
    in the form $F=f_0-D^s\cdot\boldsymbol{f}$
    with $\boldsymbol{f}=(f_1,...,f_d)\in {L^p}'(\R^d;\R^d)$ and $f_0\in L^q(\Omega)$, where, by the fractional Sobolev embeddings, $q\geq(p^*_s)'=\frac{Np}{Np-N+sp}$ if $sp<d$, $q>1$ if $sp=d$ and $q\geq1$ if $sp>d$.
\end{remark}

Our improved version of the fractional Poincaré inequality, Theorem \ref{FracPoincareInequality}, also allows us to prove that, for fixed $s$, there is an increasing order of inclusions of $\Lambda^{s,p}_0(\Omega)$  with respect to the decreasing of $p$. This generalizes the already known results for $s=0$ and $s=1$, that are simple consequences of the inclusions between the  $L^p(\Omega)$, for bounded $\Omega$, and the characterisation of $\Lambda^{0,p}_0(\Omega)=\{u\in L^p(\mathbb{R}^d):\, u=0 \mbox{ a.e. in } \mathbb{R}^d\setminus\Omega\}$, in the first case, and of the fact that the (classical) gradient is local in the second case, respectively. 

\begin{corollary}[Order with respect to $p$]\label{InclusionsForPLionsCalderonCompactSupport}
    Let $\Omega\subset\R^d$ be a bounded open subset of $\R^d$, $s\in(0,1)$ and $1\leq p<q<+\infty$. Then $\Lambda^{s,q}_0(\Omega)\subset \Lambda^{s,p}_0(\Omega)$.
\end{corollary}
\begin{proof}
    Let $\Omega_1\supset\Omega$ be an open bounded subset of $\R^d$ as in the statement of the theorem of the Poincaré inequality, Theorem \ref{FracPoincareInequality}. Then, the proof of this corollary is just a simple consequence of the fact that $f\mapsto \|D^s f\|_{L^p(\Omega_1;\R^d)}$ is an equivalent norm to $\Lambda^{s,p}_0(\Omega)$ and that $L^q(\Omega_1;\R^d)\subset L^p(\Omega_1; \R^d)$.
\end{proof}

We also have an estimate for the product of functions of $\Lambda^{s,p}_0(\Omega)$ with test functions.

\begin{corollary}[Estimate for product with test function]\label{prop:EstimateProductTestFunction}
    Let $s\in(0,1)$ and $p\in[1,\infty)$, and consider $\phi\in C^\infty_c(\Omega)$ and $u\in\Lambda^{s,p}(\R^d)$. Then $\phi u\in \Lambda^{s,p}_0(\Omega)$ with
    \begin{equation}
        \|D^s(\phi u)\|_{L^p(\R^d;\R^d)}\leq C\|D^s u\|_{L^p(\R^d;\R^d)},
    \end{equation}
    where $C=C(d,s,\Omega,\phi)>0$ is a constant that blows-up as $s\searrow 0$.
\end{corollary}
\begin{proof}
    See \cite[Lemma~2.7]{kreisbeck2022Quasiconvexity}.
\end{proof}

The last direct application that we make of the fractional Poincaré inequality that we present is the following compactness result, which was first proven in \cite{bellido2020gamma} in the case $\sigma=1$. This result is going to play a very important role in the theory of stability of solutions in Section \ref{sec:stability}.

\begin{theorem}[]\label{CompactnessFromUniformBoundOnDs}
    Let $\Omega\subset \R^d$ be a bounded open subset of $\R^d$,  and let $p\in(1,+\infty)$. Consider a convergent sequence in $(0,1]$ such that $s_n\to\sigma\in(0,1]$, and a sequence of functions $u_n\in \Lambda^{s_n,p}_0(\Omega)$ with $\sup_n{\|D^{s_n} u_n\|_{L^p(\R^d;\R^d)}}<\infty$. Then, there exists $u\in\Lambda^{\sigma,p}_0(\Omega)$ such that for every $t\in(0,\sigma)$ one can extract a subsequence also denoted by $u_n$ satisfying
    \begin{equation*}
        u_n\to u \mbox{ in } \Lambda^{t,p}_0(\Omega) \quad \mbox{ and }\quad D^{s_n} u_n\rightharpoonup D^\sigma u \mbox{ in } L^p(\R^d;\R^d).
    \end{equation*}
\end{theorem}

\begin{proof}
    The arguments in \cite[Theorem~4.2]{bellido2020gamma} may be very generalized to prove that exists a function $u\in\Lambda^{\sigma,p}_0(\Omega)$ and a subsequence $\{u_s\}_{n_j}\subset\Lambda^{t+\varepsilon,p}_0(\Omega)$ such that $u_s\to u$ in $L^p(\Omega)$, where $\varepsilon>0$ is sufficiently small to have $t+2\varepsilon<\min\{s,\sigma\}$. Now, one only needs to prove that in fact $u_s\to u$ in $\Lambda^{t,p}_0(\Omega)$. Using Gagliardo-Nirenberg interpolation inequality, Lemma \ref{GagliardoNirenbergLionsCalderon}, and the order between Lions-Calderón spaces with $p$ fixed, item 4) of Proposition \ref{prop:sobolev_embeddings_compact_support}, together with the fact that $t+2\varepsilon<\min{s,\sigma}$, we obtain
    \begin{align*}
		\|D^t(u_s-u)\|_{L^p(\R^d;\R^d)}&\leq c_{d,p}\|u_s-u\|_{L^p(\R^d)}^{\frac{\varepsilon}{t+\varepsilon}}\|D^{t+\varepsilon}(u_s-u)\|_{L^p(\R^d;\R^d)}^{\frac{t}{t+\varepsilon}}\\
        &\leq c_{d,p,t,\varepsilon}\|u_s-u\|_{L^p(\R^d)}^{\frac{\varepsilon}{t+\varepsilon}}\left(\|D^s u_s\|_{L^p(\R^d;\R^d)}+\|D^\sigma u\|_{L^p(\R^d;\R^d)}\right)^{\frac{t}{t+\varepsilon}}.
    \end{align*}
    Using the hypothesis that $\|D^su_s\|_{L^p(\R^d;\R^d)}$ is uniformly bounded, and that $u_s\to u$ in $L^p(\R^d)$ with $u\in\Lambda^{s,p}_0(\Omega)$, we conclude the proof. 
\end{proof}

\begin{corollary}\label{cor:fracRellichKondrachovFixedP}
    Let $\Omega\subset \R^d$ be a bounded open set, and let $0\leq t<s\leq 1$ and $p\in(1,\infty)$. Then,
    \begin{equation*}
        \Lambda^{s,p}_0(\Omega)\Subset \Lambda^{t,p}_0(\Omega).
    \end{equation*}
\end{corollary}
\begin{proof}
    This is a simple consequence of Theorem \ref{CompactnessFromUniformBoundOnDs} when $s_n\equiv s$, since by taking a sequence $u_n$ bounded in $\Lambda^{s,p}_0(\Omega)$ there is a $u\in\Lambda^{s,p}_0(\Omega)$ and a subsequence $u_n$ such that $D^t u_n\to D^t u$ in $L^p(\R^d;\R^d)$.
\end{proof}

\begin{theorem}[Fractional Rellich-Kondrachov]\label{CompactnessResultALaBellido}
    Let $\Omega\subset \R^d$ be a bounded open set. Then, for every $0\leq t<s\leq 1$ and $1<p\leq q<\infty$ satisfying
    \begin{equation*}
        \frac{1}{p}-\frac{s}{d}<\frac{1}{q}-\frac{t}{d}
    \end{equation*}
    we have the compact embedding
	\begin{equation*}
		\Lambda^{s,p}_0(\Omega)\Subset \Lambda^{t,q}_0(\Omega).
	\end{equation*}
\end{theorem}
\begin{proof}
    From the Corollary \ref{cor:fracRellichKondrachovFixedP} we have $\Lambda^{s,p}_0(\Omega)\Subset\Lambda^{s-\varepsilon,p}_0(\Omega)$ for every $\varepsilon\in(0,s)$. However, at the same time, due to the item 5) of Proposition \ref{prop:sobolev_embeddings_compact_support}, we have that for $t<s-\varepsilon$ and $\frac{1}{p}-\frac{s-\varepsilon}{d}\leq\frac{1}{q}-\frac{t}{d}$ we have $\Lambda^{s-\varepsilon,p}_0(\Omega)\subset \Lambda^{t,q}_0(\Omega)$. By choosing $\varepsilon$ arbitrarily small, and combining the two embeddings, we conclude the proof.
\end{proof}

We conclude this section with an interesting result that we can only prove in the case where $\Omega$ is bounded, although it is probably true in general domains.

\begin{theorem}\label{thm:densityPositiveCone}
    Let $\Omega\subset\R^d$ be an open and bounded set, $s\in(0,1)$, $p\in(1,\infty)$ and $f\in \Lambda^{s,p}(\Omega)$ with $f\geq 0$. Then, there exists a sequence of functions $f_n\in C^\infty_c(\Omega)$ with $f_n\geq 0$ for all $n\in\N$ such that $f_n\to f$ in $\Lambda^{s,p}_0(\Omega)$
\end{theorem}
\begin{proof}
    Let us fix $\varepsilon>0$. By Theorem \ref{thm:netrusovsTheorem}, there exists a function $\eta$ such that $\eta=0$ on a neighborhood of $\Omega$, $0\leq \eta\leq 1$, and $\|f-\eta f\|_{\Lambda^{s,p}(\R^d)}<\varepsilon$. Let us now consider an approximation of the identity $\{\varphi_n\}\subset L^1(\R^d)$ with $\|\varphi_n\|_{L^1(\R^d)}=1$, $\varphi_n\geq 0$ and compactly supported in $B_{1/n}$. Now, observe that
    \begin{equation*}
        \|f-\varphi_n*(\eta f)\|_{L^p(\R^d)}\leq \|f-\eta f\|_{L^p(\R^d)}+\|\eta f-\varphi_n*(\eta f)\|_{L^p(\R^d)}\leq \varepsilon+\|\eta f-\varphi_n*(\eta f)\|_{L^p(\R^d)}.
    \end{equation*}
    As noted in \cite[Theorem~3.22]{comi2019BlowUp}, $D^s(\varphi_n*(\eta f))=\varphi_n*D^s(\eta f)$, and so, we also have that
    \begin{align*}
        \|D^s(f-\varphi_n*(\eta f))\|_{L^p(\R^d)}&\leq \|D^s(f-\eta f)\|_{L^p(\R^d)}+\|D^s(\eta f)-\varphi_n*D^s(\eta f)\|_{L^p(\R^d)}\\
        &\leq \varepsilon+\|D^s(\eta f)-\varphi_n*D^s(\eta f)\|_{L^p(\R^d)}.
    \end{align*}
    Now we may choose the index $n$ sufficiently big in $\varphi_n$, such that
    \begin{equation*}
        \|\eta f-\varphi_n*(\eta f)\|_{L^p(\R^d)}+\|D^s\eta f-\varphi_n*D^s(\eta f)\|_{L^p(\R^d)}\leq \varepsilon
    \end{equation*}
    and conclude the proof by setting $f_n=\varphi_n*(\eta f)$ and observing that
    \begin{equation*}
        \mathrm{supp}(\varphi_n*(\eta f))\subset \mathrm{supp}\varphi_n+\mathrm{supp}(\eta f)\subset B_{1/n}+ \mathrm{supp}(\eta f)\subset\mathrm{supp}f.
    \end{equation*}
\end{proof}

\begin{remark}
    This argument only works when $\Omega$ is bounded because the boundary of any neighborhood of $\Omega^c$ has to be at a positive distance of the boundary of $\Omega^c$. However, this seems to be a technical restriction, since in the classical case $s=1$, this result is also valid for $\Omega$ unbounded (see \cite[Corollary 9.1.5]{adams1996Function}).
\end{remark}

\begin{remark}
    Some of the results, that we have included in this section, are well known to have a counterpart on the Sobolev-Slobodeckij spaces $W^{s,p}_0(\Omega)$. For example, the fractional Sobolev and the Rellich-Kondrachov embeddings can be found for example in \cite{adams1975Sobolev, diNezza2012Hitchhiker, triebel1983Theory}, and the fractional Poincaré inequality can be found \cite[Proposition 2.5]{brasco2016Stability}.
\end{remark}

\section{Existence results}\label{sec:existence_uniqueness}
In this section, we study the existence of solutions $u=u_s\in K$ to the variational inequalities of the form
\begin{equation}\label{eq:general_equation}
    \langle \mathscr{A}_s(u), v-u\rangle\geq \langle F, v-u\rangle, \quad \forall v\in K,
\end{equation}
where $K$ is a non-empty closed convex subset of $\Lambda^{s,p}_0(\Omega)$, $F\in \Lambda^{-s,p'}(\Omega)$ with $\Omega\subset\R^d$ being an open set, bounded or unbounded, and

\begin{equation}\label{VariationalEquation}
    \langle \mathscr{A}_s(u), v\rangle=\int_{\R^d}{\boldsymbol{a}(x,u,D^s u)\cdot D^s(v-u)}\,dx+\int_\Omega{b(x,u, D^su)v}\,dx, \quad \forall v\in \Lambda^{s,p}_0(\Omega).
\end{equation}
Notice that the operator $\mathscr{A}_s:\Lambda^{s,p}_0(\Omega)\mapsto \Lambda^{-s,p'}(\Omega)$, defined by 
\begin{equation}\label{eq:quasiLinearOperator}
    \mathscr{A}_s(v)=-D^s\cdot \boldsymbol{a}(x,v,D^s v)+b(x,v,D^s v),
\end{equation}
is of fractional divergence form if $0<s<1$, classical if $s=1$ and of novel type if $s=0$. As in the classical case, we assume that $\boldsymbol{a}:\R^d\times\R\times\R^d\to\R^d$ and $b:\R^d\times\R\times\R^d\to\R^d$ are Carathéodory functions, that is, $\boldsymbol{a}(\cdot,r,\xi)$ and $b(\cdot,r,\xi)$ are measurable for fixed $r\in\R$ and $\xi\in\R^d$, and $\boldsymbol{a}(x,\cdot,\cdot)$ and $b(x,\cdot,\cdot)$ are continuous in the variables $(r,\xi)$ for a.e. $x\in\R^d$, satisfying suitable growth and coercivity conditions.

From Remark \ref{rem:characterization_dual}, setting $F=f_0-D^s\cdot\boldsymbol{f}$ and replacing the nonlinear functions $\boldsymbol{a}$ and $b$ by $\boldsymbol{a}-\boldsymbol{f}$ and $b-f_0$, respectively, one can rewrite the inequality \eqref{eq:general_equation} in the integral form
\begin{equation}\label{integralGeneralEquation}
	\int_{\R^d}{\boldsymbol{a}(u,D^s u))\cdot D^s(v-u)}\,dx+\int_\Omega{b(u, D^s u)(v-u)}\,dx\geq 0,\quad \forall v\in K,
\end{equation}
where we have omitted the dependence of $\boldsymbol{a}$ and $b$ on $x$ to simplify the notation. On the other hand, the condition $K\subset\Lambda^{s,p}_0(\Omega)$ means that we have a homogeneous Dirichlet condition $u=0$ in $\R^d\setminus\Omega$. If instead, we have $u=g$ in $\R^d\setminus\Omega$ with $g\in\Lambda^{s,p}(\R^d)$, we can set $\overline{u}=u-g\in\Lambda^{s,p}_0(\Omega)$ and consider the variational inequality \eqref{eq:general_equation} for $\overline{u}\in \overline{K}$ and $\overline{\mathscr{A}}_s$, with $\overline{K}=K-g$ and $\overline{\mathscr{A}}_s(\overline{u})=\mathscr{A}_s(u-g)$.

With this general structure we can apply classical abstract methods in the framework of Lions-Calderón spaces, in particular the general theory of pseudo-monotone operators introduced by Brézis in \cite{brezis1968equations} (see also the classic book \cite{lions1969Quelques} or the more modern references \cite{roubivcek2013nonlinear, showalter2013monotone}), to show the existence of solutions to general nonlinear operators. Similarly, we can also study the more restrictive class of monotone operators, which in certain cases, with additional assumptions, allow us to obtain stronger results, including the uniqueness of solutions in the strictly monotone case and to handle the novel case of $D^0$.

\subsection{Pseudomonotone operators with $\Omega$ bounded and $0<s\leq 1$}\label{sec:pseudomonotone_brezis}

In this subsection we assume $\Omega$  open and bounded. For simplicity we consider first the case where $b$ does not depend on $\xi$ and later, in Remarks \ref{remark:3.1} and \ref{remark:3.2}, we discuss the general case. We assume the growth conditions
\begin{equation}\label{eq:upperboundForA}
    |\boldsymbol{a}(x,r,\xi)|\leq \gamma_1(x)+C'_1|r|^{\frac{q_1}{p'}}+C_1|\xi|^{p-1},
\end{equation}
\begin{equation}\label{eq:upperboundForB}
    |b(x,r)|\leq\gamma_2(x)+C_2|r|^{q_2-1},
\end{equation}
and the coercivity condition
\begin{equation}\label{eq:weakCoercivity}
    \boldsymbol{a}(x,r,\xi)\cdot\xi+b(x,r)r\geq\alpha|\xi|^p-\beta|r|^{q_3}-k(x)
\end{equation}
for every $r\in\R$, $\xi\in\R^d$ and a.e. $x\in\R^d$, with exponents 
\begin{equation}\label{eq:constraints_exponents}
    1<q_1<\min\{p_s^*,p^2/(p-1)\}, \quad 1<q_2<\min\{p_s^*,p+1\}, \quad \mbox{ and } \quad 1<q_3< p,
\end{equation}
constants $C_1,C'_1, C_2, \alpha>0$ and $\beta\in\R$ and functions $\gamma_1,\gamma_2\in {L^p}'(\R^d)$ and $k\in L^1(\R^d)$. 

In this case we assume that $\boldsymbol{a}$ is  only monotone with respect to the last variable, i.e.,
\begin{equation}\label{eq:monotonicity}
    (\boldsymbol{a}(x,r,\xi)-\boldsymbol{a}(x,r,\eta))\cdot(\xi-\eta)\geq 0,
\end{equation}
for every $\xi,\eta\in\R^d$, $r\in\R$ and for almost all $x\in\R^d$.

\begin{lemma}[Pseudomonotonicity]\label{lemma:pseudomonotone}
    Let $\Omega\subset\R^d$ be a bounded and open set, $s\in(0,1]$, $p\in(1,\infty)$. Let $\boldsymbol{a}$ and $b$ satisfy the growth conditions \eqref{eq:upperboundForA} and \eqref{eq:upperboundForB}, with $q_1$ and $q_2$ as in \eqref{eq:constraints_exponents}, and the monotonicity condition \eqref{eq:monotonicity} hold. Then, the operator $\mathscr{A}_s:\Lambda^{s,p}_0(\Omega)\to \Lambda^{-s,p}(\Omega)$, as defined in \eqref{eq:quasiLinearOperator}, is pseudomonotone, i.e., it is bounded and
    \begin{equation}\label{ImplicationPseudomonotone}
        \left.
        \begin{aligned}
            u_k\rightharpoonup u\\
            \limsup_{k\to\infty}{\langle\mathscr{A}_s(u_k),u_k-u\rangle}\leq 0.
        \end{aligned}
        \right\}
        \Rightarrow\langle\mathscr{A}_s(u),u-v\rangle\leq \liminf_{k\to\infty}{\langle \mathscr{A}_s(u_k),u_k-v\rangle},\quad\forall v\in V.
    \end{equation}
\end{lemma}
\begin{proof}
    We can apply a similar argument to \cite[Lemma 2.31]{roubivcek2013nonlinear} and \cite[Lemma 2.32]{roubivcek2013nonlinear}, for the boundedness and for the implication \eqref{ImplicationPseudomonotone}, respectively, by replacing $D$ by $D^s$. In fact, as observed in the previous section, the framework of the Lions-Calderón spaces has also a fractional version of the Sobolev and Poincaré inequalities as well as of the Rellich-Kondrachov's compactness embeddings. 
\end{proof}

\begin{theorem}[Existence solutions of fractional variational inequalities]\label{ExistenceFractionalVariationalInequality}
    Let $\Omega\subset\R^d$ be a bounded open set, $K$ is a non-empty closed convex subset of $\Lambda^{s,p}_0(\Omega)$, $s\in(0,1]$, $p\in(1,\infty)$ and let the coercivity condition \eqref{eq:weakCoercivity}, with $q_3$ as in \eqref{eq:constraints_exponents}, hold in addition to the assumptions present in Lemma \ref{lemma:pseudomonotone}. Then, there exists a function $u\in K$ such that
    \begin{equation}\label{VariationalInequality}
        \int_{\R^d}{\boldsymbol{a}(u,D^s u)\cdot D^s(v-u)}\,dx+\int_\Omega{b(u)(v-u)}\,dx\geq 0, \quad \forall v\in K.
    \end{equation}
\end{theorem}
\begin{proof}
    By Lemma \ref{lemma:pseudomonotone} we know that $\mathscr{A}_s$ is pseudo-monotone and, by the well-known existence theory, see \cite[Thm. 8.2, page 247]{lions1969Quelques}, it is sufficient to show that the coercivity assumption \eqref{eq:weakCoercivity} implies, for some $v_0\in K$,
    \begin{equation}\label{eq:coercivity}
        \lim_{\|v\|_{\Lambda^{s,p}_0(\Omega)}\to +\infty}{\frac{\langle \mathscr{A}_s(v),v-v_0\rangle}{\|v\|_{\Lambda^{s,p}_0(\Omega)}}}=\infty.
    \end{equation}

    In fact, \eqref{eq:coercivity}  is an immediate consequence of the following estimate 
    \begin{equation}\label{eq:estimate_for_coercivity_omega_bounded}
        \begin{aligned}
            &\left\langle \mathscr{A}_s(v),v-v_0\right\rangle=\langle\mathscr{A}_s(v),v\rangle+\langle\mathscr{A}_s(v),v_0\rangle\\
            &\quad\geq\alpha\|D^s v\|^p_{L^p(\R^d;\R^d)}
            -C\|D^s v_0\|_{L^p(\R^d;\R^d)}\left(\|D^s v\|^{\frac{q_1}{p'}}_{L^p(\R^d;\R^d)}+\|D^s v\|^{p-1}_{L^p(\R^d;\R^d)}+\|D^s v\|^{q_2-1}_{L^p(\R^d;\R^d)}\right)\\
            &\qquad\quad -C\|D^s v\|^{q_3}_{L^p(\R^d;\R^d)}-\|k\|_{L^1(\R^d)}-\|\gamma_1\|_{{L^p}'(\R^d)}\|D^s v_0\|_{L^p(\R^d;\R^d)}-\|\gamma_2\|_{{L^p}'(\R^d)}\|v_0\|_{L^p(\Omega)}
        \end{aligned}
    \end{equation}
    which is easily obtained using \eqref{eq:weakCoercivity}, \eqref{eq:upperboundForA}, \eqref{eq:upperboundForB}, and the fractional Sobolev inequalities together with the inclusions between Lebesgue spaces due to $\Omega$ being bounded.

    In fact, using \eqref{eq:weakCoercivity}, \eqref{eq:upperboundForA}, \eqref{eq:upperboundForB}, the coercivity \eqref{eq:coercivity} follows by recalling that from \eqref{eq:constraints_exponents}, we have $q_1/p^{'}<p$, $q_2-1<p$ and $q_3<p$.
\end{proof}

\begin{remark}\label{remark:3.1}
    In the case in which $0\in K$, the existence result of Theorem \ref{ExistenceFractionalVariationalInequality} can be improved for more general functions $\boldsymbol{a}$ and $b$. By choosing $v_0=0$, the estimate \eqref{eq:estimate_for_coercivity_omega_bounded} can be simplified to
    \begin{equation*}
        \left\langle \mathscr{A}_s(v),v-v_0\right\rangle=\langle\mathscr{A}_s(v),v\rangle\geq\alpha\|D^s v\|^p_{L^p(\R^d;\R^d)}
        -C\|D^s v\|^{q_3}_{L^p(\R^d;\R^d)}-\|k\|_{L^1(\R^d)}
    \end{equation*}
    which does not depend on $q_1$ and $q_2$, and so, these exponents can be extended to 
    \begin{equation*}
        1<q_1,q_2<p^*_s=
        \begin{cases}
            \frac{dp}{d-sp} & \mbox{ if } sp<d\\
            \forall q<\infty & \mbox{ if }
            sp\geq d.
        \end{cases}
    \end{equation*}
\end{remark}
\begin{remark}\label{remark:3.2}
    When $0\in K$ we can also consider $b$ to be dependent on $\xi$, see for instance \cite[Section 2.4.3]{roubivcek2013nonlinear}. One particular example is when $b$ is an affine transformation of $\xi\in\R^d$, i.e.
    \begin{equation*}
        b(x,r,\xi)=b_0(x,r)+\mathbf{b}(x, r)\cdot \xi,
    \end{equation*}
    satisfying, in conjugation with $\boldsymbol{a}$, the coercivity condition
    \begin{equation}\label{eq:coercivity_b_depend_Ds}
        \boldsymbol{a}(x,r,\xi)\cdot\xi+b(x,r,\xi)r\geq\alpha|\xi|^p-\beta|r|^{q_3}-k(x)
    \end{equation}
    for $1<q_3< p$, with $b_0$ subjected to \eqref{eq:upperboundForB} and $\mathbf{b}$ having the growth condition
    \begin{equation*}
        |\mathbf{b}(x,r)|\leq \gamma(x)+C|r|^{p^*_s/(q-\varepsilon)}.
    \end{equation*}
    with $\gamma\in L^{q+\varepsilon}(\R^d)$ for some $\varepsilon>0$, and 
    \begin{equation*}
        q=
        \begin{cases}
            \frac{dp}{dp-2d+sp} & \mbox{ if } sp<d\\
            \forall r<\infty & \mbox{ if }
            sp\geq d.
        \end{cases}
    \end{equation*}
    
    However, to obtain an existence result for \eqref{integralGeneralEquation} with a more general function $b(x,r,\xi)$ we not only need to impose a growth condition of the type
    \begin{equation*}
        |b(x,r,\xi)|\leq \gamma(x)+C\left(|r|^{p^*_s-\varepsilon-1}+|\xi|^{(p-\varepsilon)/{p^*_s}'}\right)
    \end{equation*}
    with $\varepsilon>0$ and $\gamma\in L^{{p^*_s}'+\varepsilon}(\R^d)$, but we also need to require further restrictions on the main part of $\boldsymbol{a}$. In fact, we need the strict monotonicity on $\boldsymbol{a}$ with respect to the last variable, also called Leray-Lions condition, i.e., 
    \begin{equation*}
        (\boldsymbol{a}(x,r,\xi)-\boldsymbol{a}(x,r,\eta))\cdot(\xi-\eta)> 0 \qquad \mbox{ whenever }\xi\neq \eta,
    \end{equation*}
    with the coercivity conditions
    \begin{equation*}
        \forall \eta \in\R^d:\, \lim_{|s|\to\infty}{\frac{\boldsymbol{a}(x, r, \xi)\cdot(\xi-\eta)}{|\xi|}}=+\infty \quad \mbox{uniformly for } r \mbox{ bounded}
    \end{equation*}
    and \eqref{eq:coercivity_b_depend_Ds}. Under these hypothesis we can use similar arguments as in \cite[Section 2.4.3]{roubivcek2013nonlinear} to get existence results, which is also a special case of the next subsection.
\end{remark}

\begin{remark}\label{remark:3.3}
    In the particular case when the operator $\mathscr{A}_s$ is of potential type and derives from a Gâteaux differentiable functional $\mathscr{F}_s : \Lambda^{s,p}_0(\Omega)\to\R$ given by
    \begin{equation*}
        \mathscr{F}_s(u)= \int_{\R^d} F(x,u,D^s u) dx,
    \end{equation*}
    for some differentiable F, their coefficients are given by $\boldsymbol{a}(x,u,D^s u)=F_\xi(x,u,D^s u)$ and $b(x,u, D^su)=F_r(x,u,D^s u)$. General integral functionals of this gradient type have been considered in \cite{shieh2018On} and extensions to functionals of vector valued functions $u$ in \cite{bellido2020gamma} and in \cite{kreisbeck2022Quasiconvexity}.
\end{remark}

\subsection{Pseudomonotone operators with $\Omega$ arbitrary and $0<s\leq 1$}\label{sec:existence_browder}

When we allow $\Omega$ to be an arbitrary open subset of $\R^d$, possibly unbounded, we cannot apply the same arguments as in the previous subsection since we do not have Poincaré inequalities nor Rellich-Kondrachov compactness embeddings.

To overcome the lack of Poincaré inequalities, we impose a stronger coercivity condition on the pair of functions $\boldsymbol{a}$ and $b$

\begin{equation}\label{eq:strongCorcivity}
    \boldsymbol{a}(x,r,\xi)\cdot\xi+b(x,r, \xi)r\geq \alpha|\xi|^p+\beta|r|^p-k(x),
\end{equation}
for every $r\in\R$, $\xi\in\R^d$ and a.e. $x\in\R^d$, where $\alpha>0$ and $\beta> 0$ are constants, $k\in L^1(\R^d)$, and $s\in(0,1]$. Moreover, we also impose the classical growth conditions
\begin{equation}\label{eq:upper_bound_for_A_when_omega_arbitrary}
    |\boldsymbol{a}(x,r,\xi)|\leq \gamma_1(x)+C'_1|r|^{p-1}+C_1|\xi|^{p-1},
\end{equation}
and
\begin{equation}\label{eq:upper_bound_for_B_when_omega_arbitrary}
    |b(x,r,\xi)|\leq\gamma_2(x)+C'_2|r|^{p-1}+C_2|\xi|^{p-1},
\end{equation}
with $\gamma_1,\gamma_2\in {L^p}'(\R^d)$.

 When $\Omega$ is unbounded, the faillure of the Rellich-Kondrachov compactness embeddings can be handled by using a different argument due to Browder, \cite{browder1977PseudoMonotone} for the classical case $s=1$, by assuming also the Leray-Lions condition of strict monotonicity with respect to the last variable on $\boldsymbol{a}$, i.e., for almost all $x\in\R^d$, for all $r\in\R$ and for all $\xi,\eta\in\R^d$
\begin{equation}\label{eq:strictMonotonicity}
	(\boldsymbol{a}(x,r,\xi)-\boldsymbol{a}(x,r,\eta))\cdot(\xi-\eta)> 0 \qquad \mbox{ whenever }\xi\neq \eta.
\end{equation}

As in the previous subsection, we start by proving that the operator $\mathscr{A}_s$ is pseudomonotone.

\begin{lemma}[Pseudomonotonicity when $\Omega$  is arbitrary]\label{lemma:pseudomonotone_arbitrary_domain}
    Let $\Omega\subset\R^d$ be any open set, $s\in(0,1]$ and $p\in(1,\infty)$. Let $\boldsymbol{a}:\R^d\times\R\times\R^d\to\R^d$ and $b:\R^d\times\R\times\R^d\to\R^d$ satisfy the growth conditions \eqref{eq:upper_bound_for_A_when_omega_arbitrary} and \eqref{eq:upper_bound_for_B_when_omega_arbitrary}, the strong coercivity condition \eqref{eq:strongCorcivity} and assume also that $\boldsymbol{a}$ is strictly monotone with respect to the last variable as in \eqref{eq:strictMonotonicity}. Then, the operator $\mathscr{A}_s:\Lambda^{s,p}_0(\Omega)\to \Lambda^{-s,p}(\Omega)$, as defined in \eqref{eq:quasiLinearOperator}, is pseudomonotone.
\end{lemma}
\begin{proof}
    First, we notice that assumptions \eqref{eq:upper_bound_for_A_when_omega_arbitrary} and \eqref{eq:upper_bound_for_B_when_omega_arbitrary} imply that $\mathscr{A}_s:\Lambda^{s,p}_0(\Omega)\to \Lambda^{-s,p}(\Omega)$ is bounded. 
    In order to prove \eqref{ImplicationPseudomonotone} let $u_k\rightharpoonup u$ in $\Lambda^{s,p}_0(\Omega)$ and $\limsup_{k\to\infty}{\langle\mathscr{A}_s(u_k), u_k-u\rangle}\leq 0$. We may adapt the proof of \cite[Theorem 1]{browder1977PseudoMonotone}: 
    \begin{itemize}
        \item[$i)$] applying local Rellich-Kondrachov compactness result, Proposition \ref{prop:local_rellich_kondrachov}, to an increasing sequence of bounded open sets exhausting $\Omega$, we can extract a subsequence $\{u_k\}$ that converges to $u$ a.e. in $\R^d$;
        \item[$ii)$] for the almost everywhere convergence of a subsequence $\{D^s u_k\}$ to $D^s u$ a.e. in $\R^d$, we consider the function for a.e. $x\in\R^d$
        \begin{equation*}
            p_k(x)=(\boldsymbol{a}(x,u_k, D^s u_k)-\boldsymbol{a}(x,u, D^s u))\cdot (D^s u_k-D^s u)+(b(x,u_k)-b(x,u))(u_k-u).
        \end{equation*}
        In fact, it is possible to show that $p^-_j\to 0$ strongly in $L^1(\R^d)$ due to Vitali's convergence theorem, which afterwords also yields that $p_k^+\to 0$ strongly in $L^1(\R^d)$ and consequently $p_k\to 0$ a.e. in $\R^d$. Using strictly monotonicity of $\boldsymbol{a}$ with respect to the last variable, we are able to conclude that in fact $D^s u_k\to D^s u$ a.e. in $\R^d$;
        \item[$iii)$] Once we have proven that both $\{u_k\}$ and $\{D^s u_k\}$ converge a.e. in $\R^d$, using Lebesgue theorem we obtain
        \begin{equation*}
            \langle\mathscr{A}(u_k),v\rangle\rightharpoonup \langle\mathscr{A}(u), v\rangle \mbox{ in } \Lambda^{-s,p'}(\Omega)\, \mbox{ for all } v\in \Lambda^{s,p}_0(\Omega) \quad \mbox{ and } \quad \lim_{k\to\infty}{\langle\mathscr{A}(u_k), u_k-u\rangle}=0,
        \end{equation*}
        which yields
        \begin{equation*}
            \liminf_{k\to\infty}{\langle\mathscr{A}(u_k), u_k-v\rangle}\geq \lim_{k\to\infty}{\langle\mathscr{A}(u_k), u_k-u\rangle}+\lim_{k\to\infty}{\langle\mathscr{A}(u_k), u-v\rangle}=\langle\mathscr{A}(u), u-v\rangle,
        \end{equation*}
        proving \eqref{ImplicationPseudomonotone}.
    \end{itemize}

\end{proof}

\begin{theorem}[Existence of solutions for fractional variational inequalities with $\Omega$ arbitrary] \label{ExistenceFractionalVariationalInequalityArbitraryDomains}
    Under the same assumptions of Lemma \ref{lemma:pseudomonotone_arbitrary_domain}, there exists a function $u\in K$ that solves \eqref{integralGeneralEquation}.
\end{theorem}
\begin{proof}
    By Lemma \ref{lemma:pseudomonotone_arbitrary_domain} we know that $\mathscr{A}_s$ is pseudomonotone and we may also apply \cite[Thm. 8.2, page 247]{lions1969Quelques} as in Theorem \ref{ExistenceFractionalVariationalInequality}. It is sufficient to conclude the coercivity condition \eqref{eq:coercivity} by noticing that, for any fixed $v_0\in K$, the Sobolev's embeddings and Young's inequalities allow us to get
    \begin{equation}
            \langle \mathscr{A}_s(v),v-v_0\rangle\geq \alpha'\|D^s v\|^p_{L^p(\R^d;\R^d)} +\beta'\|v\|^p_{L^p(\R^d)}-C(k,\gamma_1,\gamma_2,v_0).
    \end{equation}
    with $\alpha'>0$ and $\beta'>0$ depending on $\alpha$ and $\beta$ respectively.
\end{proof}

\begin{remark}
    When $K\subset\Lambda^{s,p}_0(\Omega)$ is bounded, we just need pseudomonotonicity of $\mathscr{A}_s$ to prove the existence of solutions. On the other hand, the coercivity condition \eqref{eq:strongCorcivity} can be replaced in Lemma \ref{lemma:pseudomonotone_arbitrary_domain} by
    \begin{equation*}
        \boldsymbol{a}(x,r,\xi)\cdot\xi+b(x,r,\xi)r\geq -C(|\xi|^q+|r|^q)-k(x),
    \end{equation*}
    with $q<p$ and $k\in L^1(\R^d)$, to prove the pseudomonotonicity as in \cite{landes1980On}.
\end{remark}

\begin{remark}
    Although Theorem \ref{ExistenceFractionalVariationalInequalityArbitraryDomains} is valid for $\Omega$ arbitrary and for a more general lower order term $b$, it does not contain in bounded domains Theorem \ref{ExistenceFractionalVariationalInequality} as a special case, even with the assumption \ref{eq:strongCorcivity} replaced by \ref{eq:weakCoercivity}, since this one requires that $\boldsymbol{a}$ is only monotone with respect to the last variable.
\end{remark}

\subsection{Monotone operators with $0\leq s\leq 1$}\label{sec:existence_monotone}
If we restrict $\boldsymbol{a}$ to be dependent only on $x$ and $\xi$, but not on $r$, and $b$ not dependent on $\xi$, such that $\mathscr{A}_s$ is a monotone operator, then we have an existence result with $\Omega$ arbitrary for each $s\in [0,1]$. In particular, we can also consider the case $s=0$ with $D^s$ corresponding to minus the vectorial Riesz transform.

\begin{theorem}[Existence of solutions for monotone operators]\label{thm:ExistenceFractionalVariationalInequalityMonotoneOperators}
    Let $\Omega\subset\R^d$ be any open set, $s\in[0,1]$ and $p\in(1,\infty)$. Let $K\subset\Lambda^{s,p}_0(\Omega)$ be a non-empty closed convex set and assume that $\boldsymbol{a}=\boldsymbol{a}(x,\xi):\R^d\times\R^d\to\R^d$ and $b=b(x,r):\R^d\times\R\to\R^d$ are monotone with respect to $\xi$ and $r$, respectively, satisfy the growth conditions \eqref{eq:upper_bound_for_A_when_omega_arbitrary} and \eqref{eq:upper_bound_for_B_when_omega_arbitrary} and the strong coercivity \eqref{eq:strongCorcivity}. Then, there exists a function $u\in K$ such that
    \begin{equation}\label{eq:VariationalInequalityWithS0}
        \int_{\R^N}{\boldsymbol{a}(D^s u)\cdot D^s(v-u)}\,dx+\int_\Omega{b(u)(v-u)}\,dx\geq 0, \quad \forall v\in K.
    \end{equation}
\end{theorem}
\begin{proof}
   This is an immediate consequence of well-known results (see for example, \cite[Lemma 2.1 and Corollary 2.2]{showalter2013monotone}) since the assumptions imply that $\mathscr{A}_s$ is a bounded, coercive, monotone, and hemicontinuous operator in the Banach space $\Lambda^{s,p}_0(\Omega)$. In particular, the hemicontinuity property is also satisfied for all $0\leq s\leq 1$, i.e. the real-valued function $t\mapsto \langle \mathscr{A}_s(u+tv),v\rangle$
is continuous for all $u,v\in\Lambda^{s,p}_0(\Omega)$, since the growth conditions \eqref{eq:upper_bound_for_A_when_omega_arbitrary} and \eqref{eq:upper_bound_for_B_when_omega_arbitrary} imply the continuity of $\Lambda^{s,p}_0(\Omega)\ni u\mapsto \boldsymbol{a}(D^s u)\in{L^p}'(\R^d;\R^d)$, by the continuity of the operator $D^s$ in $L^p(\R^d)$ for all $0\leq s\leq 1$.
\end{proof}

\begin{remark}\label{rem:limit_existence_with_poincare}
    In the limit case of $s=0$, since $\Lambda^{0,p}_0(\Omega)=L^p_0(\Omega)$, it is enough to have only one of the constants $\alpha$ or $\beta$ positive in the coercivity assumption \eqref{eq:strongCorcivity}. For $0<s\leq1$ when
    $\Omega$ is bounded, using Poincaré inequality, we can relax the coercivity assumption \eqref{eq:strongCorcivity} by letting $\beta>-\alpha(s/C_P)^p$, with $C_P>0$ given in the Poincaré inequality \eqref{eq:poincare_inequality}, and keeping $\alpha>0$.
\end{remark}

As in the classical case, the uniqueness of solutions can be shown under stricter hypothesis of strict monotony on $\boldsymbol{a}$ or $b$.

\begin{proposition}[Uniqueness through strict monotonicity of $\boldsymbol{a}$ or $b$]\label{UniquenessSolutionPointwisePrespective}
If we assume that $\boldsymbol{a}$ or $b$ is strictly monotone with respect to the last variable, then there is at most one solution of \eqref{eq:VariationalInequalityWithS0} in $K\subset\Lambda^{s,p}_0(\Omega)$.
\end{proposition}
\begin{proof}
    If $u_1$ and $u_2$ solve the variational inequality \eqref{eq:VariationalInequalityWithS0}, then
	\begin{equation}\label{inequalityForUniqueness}
		\int_{\R^n}{(\boldsymbol{a}(D^s u_1)-\boldsymbol{a}(D^s u_2))\cdot (D^s u_1-D^s u_2)}\,dx+\int_\Omega{(b(u_1)-b(u_2))(u_1-u_2)}\,dx\leq 0.
	\end{equation}
	Since $(\boldsymbol{a}(x,\xi)-\boldsymbol{a}(x,\eta))\cdot (\xi-\eta)\geq 0$ for all $\xi,\eta\in\R^d$, if $(b(x,r)-b(x,t))(r-t)>0$ for all $r,t\in\R$, we conclude that $u_1=u_2$ a.e. in $\Omega$. A similar argument can be applied when $\boldsymbol{a}$ is the one that is strictly monotone, since then it would follow $D^su_1=D^su_2$ a.e. in $\R^d$, implying also $u_1=u_2$.
\end{proof}

\subsubsection{The special case of $p$-Laplacian type operators}

In this subsection we shall consider the special class of operators 
\begin{equation}\label{eq:pLaplacianOperator}
    \mathscr{A}^s_p u=-D^s\cdot \boldsymbol{a}(x,D^s u)
\end{equation}
of $p$-Laplacian type, i.e., when $\boldsymbol{a}:\R^d\times\R^d\to\R^d$ is a Carathéodory function with $\boldsymbol{a}(x,0)=0$, satisfying the growth condition \eqref{eq:upper_bound_for_A_when_omega_arbitrary}, and there exists a constant $\alpha_p>0$ such that
\begin{equation}\label{eq:monotonicityPLaplacian}
    (\boldsymbol{a}(x, \xi)-\boldsymbol{a}(x, \eta))\cdot(\xi-\eta)
    \geq
    \begin{cases}
        \alpha_p|\xi-\eta|^p, &\mbox{ if } p\geq 2\\
        \alpha_p\frac{|\xi-\eta|^2}{(|\xi|+|\eta|)^{2-p}}, &\mbox{ if } 1<p< 2,
    \end{cases}
\end{equation}
for every $\xi,\eta\in\R^d$ and almost every $x\in\R^d$.

This class contains the fractional $(s,p)$-Laplacian 
\begin{equation}\label{eq:fractional_s_p_laplacian}
    -\Delta_p^su=-D^s\cdot (|D^s u|^{p-2}D^s u),
\end{equation}
which corresponds to choose $\boldsymbol{a}(\xi)= |\xi|^{p-2}\xi$, where the best constant $\alpha_p=2^{2-p}$ for $p\geq2$ and $\alpha_p=p-1$ for $1<p<2$, see \cite[Lemma~1.11]{edmunds2023Fractional} whose proof for the scalar case extends easily to the vectorial case $\R^d$. When $s=1$, it coincides with the usual classical $p$-Laplacian, i.e., $-\Delta^1_p=-\Delta_p$, while when $s=0$ is a new class of a nonlinear operator involving the vectorial Riesz transform.

\begin{remark}
    As it was observed in \cite{schikorra2018Regularity}, the fractional $(s,p)$-Laplacian \eqref{eq:fractional_s_p_laplacian} with the nonlocal $p$-Laplacian
    \begin{equation*}
        \left(-\Delta\right)^s_p u(x)=\kappa_{s,p,d}\,\,\,\mathrm{p.v.}\int_{\R^d}{\frac{|u(x)-u(y)|^{p-2}(u(x)-u(y))}{|x-y|^{N+sp}}}\,dy,
    \end{equation*}
    with $\kappa_{s,p,d}$ being the normalized constant computed in \cite{delTeso2021Three}. This is a consequence of the fact that the fractional $(s,p)$-Laplacian is the Euler-Lagrange in the Lions-Calderón spaces $\Lambda^{s,p}_0(\Omega)$, while the nonlocal $p$-Laplacian arises from the Euler-Lagrange equation associated to the Gagliardo seminorm of the Sobolev-Slobodeckij spaces $W^{s,p}_0(\Omega)$, which are algebraically and topologically different when $p\neq 2$, $0<s<1$. In the Hilbertian case $s=2$ for the Dirichlet boundary condition, as $W^{s,2}_0(\Omega)=\Lambda^{s,2}_0(\Omega)=H^s_0(\Omega)$, the corresponding variational problems for the fractional Laplacian are related, as observed in \cite{lo2021class}, and have had a great increase of interest, as can be seen, for instance, in 
    \cite{bucur2016Nonlocal, bisci2016Variational}.
\end{remark}

As a simple application of Theorem \ref{thm:ExistenceFractionalVariationalInequalityMonotoneOperators} and Proposition \ref{UniquenessSolutionPointwisePrespective}, we have the following simple existence and uniqueness result, which is stated here for completeness.

\begin{corollary}[Existence and uniqueness for $p$-Laplacian type of operators]\label{cor:uniqueness fractional pLaplacian without b}
    Let $\Omega\subset\R^N$ be an open and bounded set, $0\leq s\leq 1$ and $1<p<\infty$. Consider a closed, convex and non-empty set $K\subset \Lambda^{s,p}_0(\Omega)$, a $p$-Laplacian type operator $\mathscr{A}^s_p$ and any functional $F\in \Lambda^{-s,p'}(\Omega)$. Then, there exists a unique solution $u\in K$ to  the variational inequality
    \begin{equation}\label{eq:SolutionOperator}
        \langle \mathscr{A}^s_p u, v-u\rangle\geq \langle F, v-u\rangle, \quad \forall v\in K.
    \end{equation}
\end{corollary}

This last result shows that for any functional $F\in\Lambda^{-s,p'}(\Omega)$, if $u\in K$ denotes the corresponding solution of \eqref{eq:SolutionOperator}, the solution operator $\mathscr{A}_{K}^{-1} F=u\in K$ is well defined. Notice that when $K=\Lambda^{s,p}_0(\Omega)$, the operator $\mathscr{A}^{-1}_{K}$ is the inverse of the fractional $p$-Laplacian type operator $\mathscr{A}^s_p$ with Dirichlet homogeneous boundary conditions. This solution operator has interesting properties.

\begin{theorem}\label{thm:ContinuityInversePLaplacian}
    Let $\Omega\subset \R^d$ be a bounded open set, $0\leq s\leq 1$ and $1<p<\infty$. Consider a $p$-Laplacian type operator $\mathscr{A}^s_p$, a closed and convex set $K\subset \Lambda^{s,p}_0(\Omega)$, with $0\in K$, and the variational inequality \eqref{eq:SolutionOperator}. Then, setting  $u_1=\mathscr{A}_{K}^{-1}(F_1)$ and $u_2=\mathscr{A}_{K}^{-1}(F_2)$, the solution map $\mathscr{A}_{K}^{-1}:\Lambda^{-s,p'}(\Omega)\to K$ is
    \begin{itemize}
        \item $\frac{1}{p-1}$-Hölder continuous when $p\geq 2$, with
        \begin{equation}\label{eq:holder_continuity_p>2}
            \|D^s u_1 -D^su_2\|_{L^p(\R^d;\R^d)}\leq \alpha_p^{\frac{1}{1-p}}\|F_1-F_2\|^{\frac{1}{p-1}}_{\Lambda^{-s,p'}(\Omega)};
        \end{equation}
        
        \item locally Lipschitz continuous when $1<p<2$, with
        \begin{equation}\label{eq:local_holder_continuity_p<2}
            \|D^s u_1-D^s u_2\|_{L^p(\R^d;\R^d)}\leq C_{p}(F_1,F_2)\|F_1-F_2\|_{\Lambda^{-s,p'}(\Omega)},
        \end{equation}
        where $C_{p}(F_1, F_2)= 2^{\frac{(p-1)(2-p)}{p}}\alpha_p^{\frac{3-p}{1-p}}\left(\|F_1\|^{\frac{1}{1-p}}_{\Lambda^{-s,p'}(\Omega)}+\|F_2\|^{\frac{1}{1-p}}_{\Lambda^{-s,p'}(\Omega)}\right)^{2-p}$.
    \end{itemize}
\end{theorem}
\begin{proof}
	From Corollary \ref{cor:uniqueness fractional pLaplacian without b}, we have that $\mathscr{A}_{K}^{-1}:\Lambda^{-s,p'}(\Omega)\to K$ is well defined. 
	
    Applying Hölder's inequality and the inequality \eqref{eq:SolutionOperator} to the previous identities, testing them with $v=u_2$ and $v=u_1$ respectively, we get
	\begin{equation}\label{eq:UpperboundContinuityPLaplacianOperators}
        \langle \mathscr{A}^s_p u_1-\mathscr{A}^s_p u_2, u_1-u_2\rangle =\langle F_1-F_2,u_1-u_2\rangle\leq\|F_1-F_2\|_{\Lambda^{-s,p'}(\Omega)}\|D^s(u_1-u_2)\|_{L^p(\R^d;\R^d)}.
    \end{equation}
	On the other hand, by definition of a p-Laplacian type operator, we also have
	\begin{equation}\label{eq:LowerBoundPLaplacianType}
        \langle \mathscr{A}^s_p u_1-\mathscr{A}^s_p u_2, u_1-u_2\rangle\geq
        \begin{cases}
                \alpha_p\|D^s(u_1-u_2)\|^p_{L^p(\R^d;\R^d)}, &\mbox{ if } p\geq 2\\
                \alpha_p\int_{\R^d}{\frac{|D^s(u_1-u_2)|^2}{(|D^s u_1|+|D^s u_2|)^{2-p}}}, &\mbox{ if } 1<p< 2.
        \end{cases}
    \end{equation}
    Combining the estimates \eqref{eq:UpperboundContinuityPLaplacianOperators} and \eqref{eq:LowerBoundPLaplacianType}, we can easily see that for $p\geq 2$, we get
    \begin{equation*}
       \alpha_p\|D^s (u_1-u_2)\|^{p-1}_{L^p(\R^d;\R^d)}\leq \|F_1-F_2\|_{\Lambda^{-s,p'}(\Omega)}.
    \end{equation*}
	In the case $1<p<2$, we can use Hölder's inequality together with \eqref{eq:UpperboundContinuityPLaplacianOperators} and \eqref{eq:LowerBoundPLaplacianType} to get
    \begin{equation}\label{eq:almostRightBoundPLeq2}
        \frac{\|D^s(u_1-u_2)\|_{L^p(\R^d;\R^d)}}{\left(\|D^s u_1\|_{L^p(\R^d;\R^d)}+\|D^s u_2\|_{L^p(\R^d;\R^d)}\right)^{2-p}}\leq 2^{\frac{(p-1)(2-p)}{p}}\alpha_p^{\frac{1}{1-p}}\|F_1-F_2\|_{\Lambda^{-s,p'}(\Omega)}.
    \end{equation}
    But since $\boldsymbol{a}(x,0)=0$ and 
    \begin{equation*}
        \alpha_p\|D^s u_j\|^{p-1}_{L^p(\R^d;\R^d)}\leq \frac{\langle \mathscr{A}^s_p u_j, u_j\rangle}{\|D^s u_j\|_{L^p(\R^d;\R^d)}}\leq\|F_j\|_{\Lambda^{-s,p'}(\Omega)}
    \end{equation*}
    for $j=1,2$, we can improve the estimate \eqref{eq:almostRightBoundPLeq2} to
	\begin{equation*}
		\|D^s(u_1-u_2)\|_{L^p(\R^d;\R^d)}\leq 2^{\frac{(p-1)(2-p)}{p}}\alpha_p^{\frac{3-p}{1-p}}\|F_1-F_2\|_{\Lambda^{-s,p'}(\Omega)}\left(\|F_1\|^{\frac{1}{1-p}}_{\Lambda^{-s,p'}(\Omega)}+\|F_2\|^{\frac{1}{1-p}}_{\Lambda^{-s,p'}(\Omega)}\right)^{2-p}.
    \end{equation*}
\end{proof}

As a consequence of the Theorem \ref{CompactnessResultALaBellido}, we get also the compactness of $\mathscr{A}_{K}^{-1}:\Lambda^{-s,p'}(\Omega)\to \Lambda^{t,q}_0(\Omega)$:

\begin{corollary}\label{cor:compactness_solutions_p_Laplacian}
    Under the same hypothesis of Theorem \ref{thm:ContinuityInversePLaplacian}, we have that $\mathscr{A}_{K}^{-1}:\Lambda^{-s,p'}(\Omega)\to \Lambda^{t,q}_0(\Omega)$ is compact for $0<t<s$ and $1<p,q<\infty$ satisfying $\frac{1}{p}-\frac{s}{d}<\frac{1}{q}-\frac{t}{d}$.
\end{corollary}

\begin{remark}
   The above continuity and compactness results extend to the fractional setting with a convex constraint the results of \cite{simon1981Regularity}, which corresponds to the classical $p$-Laplacian equation, i.e. $K=\Lambda^{s,p}_0(\Omega)$ and $s=1$.
\end{remark}

Besides $p$-Laplacian type operators, we can also consider some operators, that constitute variations of these, for which we can also obtain existence and uniqueness results. We present first a case in which we have also a lower order term of the type $b(r)=|r|^{p-2}r$.

Let the Carathéodory function $b(x,r):\Omega\times\R\to\R$ satisfy the growth condition \eqref{eq:upper_bound_for_B_when_omega_arbitrary} and assume there exists a constant $\beta_p\geq 0$ such that
\begin{equation}\label{eq:monotonicityb}
    (b(x, \rho)-b(x, r))(\rho-r)
    \geq
    \begin{cases}
        \beta_p|\rho-r|^p, &\mbox{ if } p\geq 2\\
        \beta_p\frac{|\rho-r|^2}{(|\rho|+ |r|)^{2-p}}, &\mbox{ if } 1<p< 2,
    \end{cases}
\end{equation}
for every $\rho,r\in\R$ and almost every $x\in\Omega$.

\begin{corollary}[$p$-Laplacian type operator with lower order term] \label{UniquenessFractionalPLaplacian}
    Let $\Omega\subset\R^N$ be an open set, $0\leq s\leq 1$ and $1<p<\infty$ and a constant $\beta_p\geq 0$ (with $\beta_p>0$ when $\Omega$ is unbounded). Consider now the $p$-Laplacian type operator $\mathscr{A}^s_p$ with a lower order perturbation $b$ satisfying the assumption \eqref{eq:monotonicityb}, with the constant $\beta_p\geq 0$ (with $\beta_p>0$ when $\Omega$ is unbounded), a closed and convex set $K\subset\Lambda^{s,p}_0(\Omega)$, with $0\in K$, and a functional $F\in \Lambda^{-s,p'}(\Omega)$. Then the variational inequality
    \begin{equation*}
        \langle \mathscr{A}^s_p u+b(u), v-u\rangle\geq \langle F, v-u\rangle, \quad \forall v\in K,
    \end{equation*}
    has a unique solution $u\in K$. Moreover, the respective solution map $(\mathscr{A}^s_p+b)_K^{-1}:\Lambda^{-s,p'}(\Omega)\to K$ is $\frac{1}{p-1}$-Hölder continuous when $p\geq 2$ and locally Lipschitz continuous when $1<p<2$.
\end{corollary}
\begin{proof}
    In the case of bounded $\Omega$ the operator $\mathscr{A}^s_p+b$ still satisfies the inequalities \eqref{eq:LowerBoundPLaplacianType} and the same conclusions for $u_1=(\mathscr{A}^s_p+b)_K^{-1}(F_1)$ and $u_2=(\mathscr{A}^s_p+b)_K^{-1}(F_2)$
    hold as in Theorem \ref{thm:ContinuityInversePLaplacian}.
    In the case of unbounded $\Omega$ the condition $\beta_p>0$ in the assumption \eqref{eq:monotonicityb} with the the inequalities \eqref{eq:LowerBoundPLaplacianType}, a similar argument as in the proof of Theorem \ref{thm:ContinuityInversePLaplacian} provides estimates for the norms \eqref{eq:norm_lions_calderon_spaces}, similar to \eqref{eq:holder_continuity_p>2} and to \eqref{eq:local_holder_continuity_p<2}, for the $\frac{1}{p-1}$-Hölder continuity  and locally Lipschitz continuity of $(\mathscr{A}^s_p+b)_K^{-1}:\Lambda^{-s,p'}(\Omega)\to \Lambda^{s,p}_0(\Omega)$, respectively when $p\geq 2$ and when $1<p<2$, now with constants also depending on $\beta_p$.
\end{proof}

A second example, is given by a perturbation with  a Carathéodory function $\boldsymbol{e}:\R^d\times\R\to\R^d$, such that, for all $r,\rho\in\R$ and a.e. $x\in\R^d$,
\begin{equation}\label{eq:growth_condition_perturbation}
   |\boldsymbol{e}(x,r)-\boldsymbol{e}(x,\rho)|\leq \lambda|r-\rho| \quad \mbox{ and } \quad \boldsymbol{e}(x,0)=0,
\end{equation} 
with $\lambda>0$, in the possible nonlinear $p=2$ framework in $\Lambda^{s,2}_0(\Omega)=H^s_0(\Omega)$.

\begin{corollary}[Small perturbation of a $2$-Laplacian type operator]
    Let $\Omega\subset\R^d$ be an open and bounded set, and $0\leq s\leq 1$. Consider a function $\boldsymbol{a}:\R^d\times\R^d\to\R^d$ that generates a $2$-Laplacian type operator with constant $\alpha_2>0$, and a function $\boldsymbol{e}:\R^d\times\R\to\R^d$ satisfying \eqref{eq:growth_condition_perturbation} with $\lambda$ sufficiently small such that $\alpha_2>\lambda c_{2,s}$, where $c_{2,s}>0$ is the best Poincaré constant in $H^s_0(\Omega)$, as in Remark \ref{rem:uniform_boundness_poincare_constant_p2}. Then, for any $F_s\in\Lambda^{-s,2}(\Omega)=H^{-s}(\Omega)$ and any non-empty, closed and convex set $K\subset\Lambda^{s,2}_0(\Omega)=H^s_0(\Omega)$, there exists a unique function $u\in K$ such that
    \begin{equation}\label{eq:existence_p_laplacian_with_drift}
        \int_{\R^d}{(\boldsymbol{a}(D^s u)+\boldsymbol{e}(u))\cdot D^s(v-u)}\,dx\geq \langle F_s, v-u\rangle,\quad\forall v\in K.
    \end{equation}
\end{corollary}
\begin{proof}
    Since $\boldsymbol{a}=\boldsymbol{a}(x,\xi)$ generates a $2$-Laplacian type operator, it satisfies the inequality \eqref{eq:LowerBoundPLaplacianType} for $p=2$, and $\boldsymbol{e}=\boldsymbol{e}(x,r)$ is a Lipschitz function in $r$, under the assumptions of this Corollary, the function $(x,r,\xi)\mapsto \boldsymbol{a}(x,\xi)+\boldsymbol{e}(x,r)$
    defines a nonlinear operator $\mathscr{A}^s_2$, which is strictly monotone and coercive, and satisfies \eqref{eq:LowerBoundPLaplacianType} for some $\alpha'_2=\alpha_2-\lambda c_{2,s}>0$. Indeed, by Cauchy-Schwarz's and Poincaré's inequalities, we have
    \begin{align*}
        &\int_{\R^N}{(\boldsymbol{a}(D^s u_1)-\boldsymbol{a}(D^s u_2)+\boldsymbol{e}(u_1)-\boldsymbol{e}(u_2))\cdot D^s(u_1-u_2)}\,dx\\
        &\qquad\qquad\geq\alpha_2 \|D^s u_1-D^s u_2\|^2_{L^2(\R^N;\R^N)}-\int_{\R^N}{\lambda|u_1-u_2||D^su_1-D^su_2|}\,dx\\
        &\qquad\qquad\geq\alpha_2 \|D^s u_1-D^s u_2\|^2_{L^2(\R^N;\R^N)}-\lambda \|u_1-u_2\|_{L^2(\R^N)}\|D^s u_1-D^s u_2\|_{L^2(\R^N;\R^N)}\\
        &\qquad\qquad\geq\left(\alpha_2-\lambda c_{2,s}\right) \|D^s u_1-D^s u_2\|^2_{L^2(\R^N;\R^N)}.
    \end{align*}
\end{proof}

\section{Mosco convergence and continuous dependence}\label{sec:stability}
In this section, we consider the convergence of solutions of variational inequalities in the framework of Section \ref{sec:existence_uniqueness} with respect to the convergence of the convex sets, specially with respect to the variation of the fractional parameter $s\in[0,1]$. We start by recalling the Mosco's convergence for convex sets $K_n\xrightarrow{M} K$. We say that $K_n$, with $n\in\mathbb{N}$, converges in the sense of Mosco to $K$ in a Banach space $X$ if: (i) for every $u\in K$, there exists a sequence $u_n\in K_n$ such that $u_n\to u$ in $X$, i.e. $K\subset \liminf_n{K_n}$; and (ii) if there exists a sequence $u_n\in K_n$ and a function $u\in X$ such that $u_n\rightharpoonup u$ in $X$, then $u\in K$, i.e. $\text{w--}\limsup_n{K_n}\subset K$.

However, in this section we are considering convergences of $u_n\in K_n\subset \Lambda^{s_n,p}_0(\Omega)$ solutions of the variational inequalities
\begin{equation}
    \langle-D^{s_n}\cdot \boldsymbol{a}(u_n,D^{s_n} u_n)+b(u_n,D^{s_n}u_n
    )-F_n, v-u_n\rangle_{s_n}\geq 0, \quad \forall v\in K_n\subset \Lambda^{s_n,p}_0(\Omega).
\end{equation}
Here not only the convex sets $K_n$ depend on the parameter $s_n$ but also the space $\Lambda^{s_n,p}_0(\Omega)$ and the operator, including the $F_n$, depend on $s_n$. These terms, may be represented in the form
\begin{equation}
    \langle F_n, v\rangle_{s_n}=\langle {f_0}_n-D^{s_n}\cdot \boldsymbol{f}_n, v\rangle_{s_n}=\int_{\Omega}{{f_0}_n v}\,dx+\sum_{j=1}^d{\int_{\R^d}{{f_j}_n\frac{\partial^{s_n} v}{\partial x_j^{s_n}}}\,dx} \quad  \forall v\in\Lambda^{s_n,p}_0(\Omega)
\end{equation}
with ${f_0}_n\in {L^p}'(\Omega)$ and $\boldsymbol{f}_n=({f_1}_n,...,{f_d}_n)\in {L^p}'(\R^d;\R^d)$.

Since the space of the solutions depends on $n$ and we need to consider not only the convergence of the solutions but also of their fractional gradients, we are lead to the following adaptation of the notion of Mosco's convergence:
\begin{definition}
    Let us consider a sequence $\{s_n\}\subset[0,1]$ that converges to $\sigma\in[0,1]$, convex sets $K_n\subset \Lambda^{s_n,p}_0(\Omega)$ for each $n\in\N$, and $K\subset\Lambda^{\sigma,p}_0(\Omega)$, and define the linear map
	\begin{equation}\label{DefinitionPs}
		\begin{aligned}
			\Pi_s: \Lambda^{s,p}_0(\Omega)&\to \L^p(\Omega)=L_0^p(\Omega)\times L^p(\R^d;\R^d)\\
			u&\mapsto (u, D^s u).
		\end{aligned}
	\end{equation}
    We say that $K_n$ converges in the generalized sense of Mosco to $K$, denoted $K_n\xrightarrow{s_n-M} K$, when:
    \begin{enumerate}[label=\emph{M\arabic*)}, itemsep =2pt, topsep=3pt]
        \item for every $u\in K$, there exists a sequence $\{u_n\}$ with $u_n\in K_n$, such that $\Pi_{s_n}(u_n)\to \Pi_\sigma(u)$ in $\L^p(\Omega)$; and \label{M1}
        \item if there exists a sequence ${u_n}\in \Lambda^{s_n,p}_0(\Omega)$ with $u_n\in K_n$ and a function $u\in \Lambda^{\sigma,p}_0(\Omega)$, such that $\Pi_{s_n}(u_n)\rightharpoonup \Pi_\sigma(u)$ in $\L^p(\Omega)$, then $u\in K$.\label{M2}
    \end{enumerate}
\end{definition}
\begin{remark}\label{rem:classicalMoscoTheorem}
    When the sequence $\{s_n\}$ is a constant sequence, i.e., $s_n=\sigma$ for all $n\in\N$, the generalized Mosco convergence reduces to the classical notion of Mosco's convergence \cite{mosco1969convergence}.
\end{remark}

In the following subsections we shall consider sequences  $\{s_n\}\subset[0,1]$ converging to $\sigma\in[0,1]$ and a corresponding converging sequence of convex sets in the generalized sense of Mosco
\begin{equation}\label{eq:generalized_mosco_convergence}
    \Lambda^{s_n,p}_0(\Omega)\supset K_n\xrightarrow{s_n-M} K\subset \Lambda^{\sigma,p}_0(\Omega),
\end{equation}
and sequence of given data, such that
\begin{equation}\label{eq:convergence_data}
    {f_0}_n\to f_0 \mbox{ in } {L^p}'(\Omega) \quad \mbox{ and } \quad \boldsymbol{f}_n\to \boldsymbol{f} \mbox{ in } {L^p}'(\R^d;\R^d),
\end{equation}
which implies  $\Lambda^{-s_n,p'}(\Omega)\ni F_n\xrightarrow{s_n} F\in \Lambda^{-\sigma,p'}(\Omega)$ in the distributional sense, i.e. $\langle F_n, v\rangle_{s_n}\xrightarrow{s_n}\langle F, v\rangle_{\sigma}$ for all $v\in C^{\infty}_0(\Omega)$.

For simplicity, we shall consider the functions $\boldsymbol{a}$ and $b$, used in the different cases of the existence theorems of Section \ref{sec:existence_uniqueness}, with the different structural assumptions independent of the sequences. Nevertheless the corresponding operators $\mathscr{A}_{s_n}$ still depend on the sequences $s_n$. Since we are separating the nonhomogeneous terms $F_n$ and $F$ from the nonlinearities of $\boldsymbol{a}$ and $b$, in their growth conditions \eqref{eq:upperboundForA}, \eqref{eq:upperboundForB}, when $\Omega$ is bounded, and \eqref{eq:upper_bound_for_A_when_omega_arbitrary}, \eqref{eq:upper_bound_for_B_when_omega_arbitrary}, when $\Omega$ is arbitrary, it is not restrictive to assume the terms $\gamma_1$ and $\gamma_2$ equal to zero, which is assumed in this section. In the case of bounded $\Omega$, the powers $q_1$ and $q_2$ in \eqref{eq:constraints_exponents} can depend on $s_n$ because the Sobolev embedding are used in the proof of Theorem \ref{ExistenceFractionalVariationalInequality}. In this section, since the fractional parameter $s$ is allowed to change with $n$, in order to have $q_1$ and $q_2$ independent of $s_n$ we impose
\begin{equation}\label{eq:constraints_exponents_independent_s}
    1<q_1<\min\{p_s^*,p^2/(p-1)\}\quad \mbox{ and } \quad 1<q_2<\min\{p_s^*,p+1\}
\end{equation}
where $0<s^*<s_n$ for all $n\in\N$.

\subsection{The pseudomonotone case with $0<s\leq 1$ and $\Omega$ bounded}
The following Mosco type theorem yields the convergence of solutions in the general situation of Subsection \ref{sec:pseudomonotone_brezis}.

\begin{theorem}\label{theorem:weak_stability_variational_inequality}
    Let $\Omega\subset\R^d$ be a bounded open set, $1<p<+\infty$ and consider a sequence $\{s_n\}\subset(s^*,1]$ with $s^*>0$ and $s_n\to\sigma\in[s^*,1]$. Let $K_n$ and $K$ be convex sets satisfying \eqref{eq:generalized_mosco_convergence} and $F_n={f_0}_n-D^{s_n}\cdot \boldsymbol{f}_n\in \Lambda^{-s_n,p'}(\Omega)$ satisfying \eqref{eq:convergence_data}. If the two Carathéodory functions $\boldsymbol{a}:\R^d\times\R\times\R^d\to\R^d$ and $b:\R^d\times\R\to\R^d$ satisfy the growth conditions \eqref{eq:upperboundForA} and \eqref{eq:upperboundForB} with $q_1$ and $q_2$ as in \eqref{eq:constraints_exponents_independent_s}, the weak coercivity conditions \eqref{eq:weakCoercivity} with $1<q_3<p$, and the monotonicity condition \eqref{eq:monotonicity}, then from the sequence of solutions $\{u_n\}$, with $u_n\in K_n$ satisfying
    \begin{equation}\label{eq:VariationalProblemWeakConvergence}
        \langle -D^{s_n}\cdot\boldsymbol{a}(u_n, D^{s_n} u_n)+b(u_n)-F_n, v-u_n\rangle_{s_n}\geq 0,\quad \forall v\in K_n,
    \end{equation}
    one can extract a subsequence, still denoted by $u_n$, with the property 
    \begin{equation}
        u_n\to u \mbox{ in } \Lambda^{t,p}_0(\Omega), \mbox{ for } 0\leq t<s^*, \quad \mbox{ and } \quad D^{s_n}u_n\rightharpoonup D^\sigma u \mbox{ in } L^p(\R^d;\R^d),
    \end{equation}
    where,  for $F:=f_0-D^\sigma\cdot\boldsymbol{f}$, $u\in K$ solves    \begin{equation}\label{eq:limitVariationalProblemWeakConvergence}
        \langle -D^\sigma\cdot\boldsymbol{a}(u, D^\sigma u)+b(u)-F, v-u\rangle_\sigma\geq 0,\quad \forall v\in K.
    \end{equation}
\end{theorem}

To prove this theorem we first find a priori estimates that will allow to extract a convergence subsequence, which limit will be shown to be solution of the limit problem.

\begin{lemma}\label{lemma:ConstructionSolutionOmegaBounded}
    Under the hypotheses of the Theorem \ref{theorem:weak_stability_variational_inequality}, as $s_n\to\sigma\in[s^*,1]$  one can extract from $\{u_n\}$ a subsequence such that $u_n\to u$ in $\Lambda^{t,p}_0(\Omega)$ for $0\leq t<s^*$ and  $D^{s_n}u_n\rightharpoonup D^\sigma u$ in $L^p(\R^d;\R^d)$ for some $u\in K$.
\end{lemma}
\begin{proof}
    For an arbitrary function $v\in K\subset \Lambda^{\sigma,p}_0(\Omega)$, by assumption, there exists a sequence $\{v_n\}$, with $v_n\in K_n$, such that $\Pi_{s_n}(v_n)\to \Pi_\sigma(v)$ in $\L^p(\R^d)$ as $s_n\to \sigma$. With these functions in \eqref{eq:VariationalProblemWeakConvergence} and the weak coercivity of $\boldsymbol{a}$ and $b$, \eqref{eq:weakCoercivity}, we obtain
	\begin{equation}\label{InequalityForAPrioriBound}
		\begin{aligned}
                &\int_{\R^d}{\boldsymbol{a}(u_n,D^{s_n} u_n)\cdot D^{s_n} v_n}\,dx+\int_{\Omega}{b(u_n)v_n}\,dx+\langle F_n, u_n-v_n\rangle_{s_n}\\
                &\qquad\qquad\geq \int_{\R^d}{\boldsymbol{a}(u_n,D^{s_n} u_n)\cdot D^{s_n} u_n}\,dx+\int_{\Omega}{b(u_n)u_n}\,dx\\
                &\qquad\qquad\geq \alpha \|D^{s_n} u_n\|^p_{L^p(\R^d;\R^d)} -\beta C_{s^*}\|D^{s_n} u_n\|^{q_3}_{L^p(\R^d;\R^d)}-\|k\|_{L^1(\R^d)}.
		\end{aligned}
	\end{equation}
	For the left-hand side of the previous equations, we use the growth conditions of $\boldsymbol{a}$ and $b$, \eqref{eq:upperboundForA} and \eqref{eq:upperboundForB}, respectively, Hölder's inequality, Young's inequality with $\delta$ and Poincaré's inequality \eqref{eq:poincare_inequality}, with $s^*>0$, to obtain
	\begin{equation}\label{UpperBoundAprioriA}
        \begin{aligned}
            &\left|\int_{\R^d}{\boldsymbol{a}(u_n,D^{s_n} u_n)\cdot D^{s_n} v_n}\,dx+\int_\Omega{b(u_n)v_n}\,dx\right|\\
            &\quad\leq C(s^*,v_0)\left(\|D^{s_n} u_n\|^{\frac{q_1}{p'}}_{L^p(\R^d;\R^d)}+\|D^{s_n}u_n\|^{p-1}_{L^p(\R^d;\R^d)}+\|D^{s_n} u_n\|^{(q_2-1)}_{L^p(\R^d;\R^d)}\right)
        \end{aligned}
    \end{equation}
	and
	\begin{equation}\label{UpperBoundAprioriL}
		\begin{aligned}
                |\langle F_n, u_n-v_n\rangle_{s_n}|
                &\leq |\langle {f_0}_n, u_n\rangle_{s_n}|+|\langle {f_0}_n, v_n\rangle_{s_n}|+|\langle \boldsymbol{f}_n, D^{s_n} u_n\rangle_{s_n}|+|\langle \boldsymbol{f}_n, D^{s_n} v_n\rangle_{s_n}|\\
                &\leq \delta\|D^{s_n} u_n\|^{p}_{L^p(\R^d;\R^d)}+C(\delta, M_1, M_2)
		\end{aligned}
	\end{equation}
        where $M_1, M_2\geq 0$ are such that $\|v_n\|_{L^p(\Omega)}+\|D^{s_n} v_n\|_{L^p(\R^d;\R^d)}\leq M_1$, $\|{f_0}_n\|_{L^p(\Omega)}+\|\boldsymbol{f}_n\|_{{L^p}'(\R^d;\R^d)} \leq M_2$, and $\delta>0$ is a sufficiently small positive number. Combining these estimates, with 
        \begin{multline*}
            C(s^*,v_0)\left(\|D^{s_n} u_n\|^{\frac{q_1}{p'}}_{L^p(\R^d;\R^d)}+\|D^{s_n}u_n\|^{p-1}_{L^p(\R^d;\R^d)}+\|D^{s_n} u_n\|^{(q_2-1)}_{L^p(\R^d;\R^d)}\right)\\
            +\beta C_{s^*}\|D^{s_n} u_n\|^{q_3}_{L^p(\R^d;\R^d)}\leq \varepsilon \|D^{s_n} u_n\|^p_{L^p(\R^d;\R^d)} +C(\varepsilon, p, q_1, q_2, q_3, \beta C_{s^*}, C(s^*,v_0)),
        \end{multline*}
        for a sufficiently small $\varepsilon>0$, by recalling from \eqref{eq:constraints_exponents} that $q_1/p^{'}<p$, $q_2-1<p$ and $q_3<p$, we may conclude that $\|D^{s_n} u_n\|_{L^p(\R^d;\R^d)}$ is uniformly bounded by a constant independent of $s_n$.

        As a simple consequence of the uniform bound of $\|D^{s_n} u_n\|_{L^p(\R^d;\R^d)}$ we get, from Theorem \ref{CompactnessFromUniformBoundOnDs}, that there exists a subsequence still denoted by $\{u_n\}$ that converges strongly in $\Lambda^{t,p}_0(\Omega)$, with $0\leq t<s^*\leq\sigma$, to some function $u\in \Lambda^{\sigma,p}_0(\Omega)$ with the additional property that $D^{s_n} u_n\rightharpoonup D^\sigma u$ in $L^p(\R^d;\R^d)$. Since $K_n\subset \Lambda^{s_n,p}_0(\Omega)$ converges in the generalized sense of Mosco to $K\subset \Lambda^{\sigma,p}_0(\Omega)$, then by M2) and by the uniqueness of weak limits, we have that $u\in K$.
\end{proof}

\begin{proof}[Proof of Theorem \ref{theorem:weak_stability_variational_inequality}.]
        Let $u$ be given by Lemma \ref{lemma:ConstructionSolutionOmegaBounded} and $v\in K$. We start by considering the family of functions $u_\varepsilon=(1-\varepsilon)u+\varepsilon v\in K$ with $\varepsilon\in (0,1]$ and a sequence of functions $\{w_n\}$ with the property that $w_n\in K_n$ for each $s_n$ and with $\Pi_{s_n}(w_n)\to \Pi_\sigma(u)$ in $\L^p(\Omega)$ , as $K_n$ converges in the generalized sense of Mosco to $K$. Using $\int_{\R^d}{(\boldsymbol{a}(u_n, D^{s_n}u_n)-\boldsymbol{a}(u_n, D^\sigma u_\varepsilon))\cdot(D^{s_n}u_n-D^\sigma u_\varepsilon))}\geq 0$, with a simple calculation, we obtain
	\begin{align*}
		&\varepsilon\int_{\R^d}{\boldsymbol{a}(u_n,D^{s_n} u_n)\cdot D^\sigma (u-v)}\,dx
		\geq \varepsilon\int_{\R^d}{\boldsymbol{a}(u_n,D^\sigma u_\varepsilon)\cdot D^\sigma(u-v)}\,dx\\
		&\qquad\qquad+\int_{\R^d}{\boldsymbol{a}(u_n,D^\sigma u_\varepsilon)\cdot (D^{s_n} u_n-D^\sigma u)}\,dx
		+\int_{\R^d}{\boldsymbol{a}(u_n,D^{s_n} u_n)\cdot D^{s_n}(w_n-u_n)}\,dx\\
		&\qquad\qquad-\int_{\R^d}{\boldsymbol{a}(u_n,D^{s_n} u_n)\cdot (D^{s_n} w_n- D^\sigma u)}\,dx.
	\end{align*}
Hence, take the inferior limit on both sides of the previous inequality, we obtain
\begin{equation}\label{InequalityProofStability}
    \begin{aligned}
		&\varepsilon\liminf_{n\to\infty}{\int_{\R^d}{\boldsymbol{a}(u_n,D^{s_n} u_n)\cdot D^\sigma(u-v)}\,dx}
		\geq \varepsilon\lim_{n\to\infty}{\int_{\R^d}{\boldsymbol{a}(u_n,D^\sigma u_\varepsilon)\cdot D^\sigma (u-v)}\,dx}\\
		&\qquad\qquad+\lim_{n\to\infty}{\int_{\R^d}{\boldsymbol{a}(u_n,D^\sigma u_\varepsilon)\cdot (D^{s_n} u_n-D^\sigma u)}\,dx}
        +\liminf_{n\to\infty}{\int_{\R^d}{\boldsymbol{a}(u_n,D^{s_n} u_n)\cdot D^{s_n}(w_n-u_n)}\,dx}\\
		&\qquad\qquad-\lim_{n\to\infty}{\int_{\R^d}{\boldsymbol{a}(u_n,D^{s_n} u_n)\cdot (D^{s_n} w_n-D^\sigma u)}\,dx}.
    \end{aligned}
\end{equation}
The limits of the right hand side can be computed easily in the following way. First, with the help of Hölder's inequality together with the facts $\boldsymbol{a}(u_n,D^\sigma u_\varepsilon)\to \boldsymbol{a}(u,D^\sigma u_\varepsilon)$ in ${L^p}'(\R^d\times\R^d)$, $u_n\to u$ in $L^{q_1}(\R^d)$ and $D^{s_n} u_n\rightharpoonup D^\sigma u$ in $L^p(\R^d;\R^d)$, we deduce
\begin{equation}\label{EstimateForStability1}
\lim_{n\to\infty}{\int_{\R^d}{\boldsymbol{a}(u_n,D^\sigma u_\varepsilon)\cdot D^\sigma(u-v)}\,dx}=\int_{\R^d}{\boldsymbol{a}(u,D^\sigma u_\varepsilon)\cdot D^\sigma (u-v)}\,dx,
\end{equation}
and
\begin{equation}\label{EstimateForStability2}
    \begin{aligned}
		&\lim_{n\to\infty}{\left|\int_{\R^d}{\boldsymbol{a}(u_n,D^\sigma u_\varepsilon)\cdot (D^{s_n} u_n-D^\sigma u)}\,dx\right|}\\
		&\qquad\qquad\leq \lim_{n\to\infty}\biggl(\|\boldsymbol{a}(u_n, D^\sigma u_\varepsilon)-\boldsymbol{a}(u, D^\sigma u_\varepsilon)\|_{{L^p}'(\R^d;\R^d)}\|D^{s_n} u_n-D^\sigma u\|_{L^p(\R^d;\R^d)}\\
		&\qquad\qquad\qquad\qquad\qquad\qquad\qquad\qquad+\left|\int_{\R^d}{\boldsymbol{a}(u,D^\sigma u_\varepsilon)\cdot (D^{s_n} u_n-D^\sigma u)}\,dx\right|\biggr)=0.
    \end{aligned}
\end{equation}
In addition, since by the growth condition of $\boldsymbol{a}$, \eqref{eq:upperboundForA}, and the uniform boundedness of $\|u_n\|_{L^p(\R^d;\R^d)}$, and using the fact that $D^{s_n} w_n\to D^\sigma u$ in $L^p(\R^d;\R^d)$ we also obtain
\begin{equation}\label{EstimateForStability3}
	\lim_{n\to\infty}{\int_{\R^d}{\boldsymbol{a}(u_n,D^{s_n} u_n)\cdot (D^{s_n} w_n-D^\sigma u)}\,dx}=0.
\end{equation}
For the inferior limit of the right hand side of \eqref{InequalityProofStability}, we use the fact that each $u_n\in K_n$ is a solution of \eqref{eq:VariationalProblemWeakConvergence} and $w_n\to u$ in $L^p(\R^d)$ to get
\begin{equation}\label{EstimateForStability4}
    \liminf_{n\to\infty}{\int_{\R^d}{\boldsymbol{a}(u_n, D^{s_n} u_n)\cdot D^{s_n}(w_n-u_n)}}\,dx\geq \lim_{n\to\infty}{\langle F, w_{s_n}-u_{s_n}\rangle_{s_n}}=0.
\end{equation}
Applying now \eqref{EstimateForStability1}, \eqref{EstimateForStability2}, \eqref{EstimateForStability3} and \eqref{EstimateForStability4} into \eqref{InequalityProofStability} we obtain an inequality, which divided by $\varepsilon>0$, yields
\begin{equation}\label{EstimateRelatesSnWithClassical}
    \begin{aligned}
        \liminf_{n\to\infty}{\int_{\R^d}{\boldsymbol{a}(u_n,D^{s_n} u_n)\cdot D^\sigma(u-v)}\,dx}
        &\geq \lim_{\varepsilon\to 0}\int_{\R^d}{\boldsymbol{a}(u,D^\sigma u_\varepsilon)\cdot D^\sigma(u-v)}\,dx\\
        &=\int_{\R^d}{\boldsymbol{a}(u,D^\sigma u)\cdot D^\sigma(u-v)}\,dx.
    \end{aligned}
\end{equation}
Using the monotonicity of $\boldsymbol{a}$ with the convergence $\boldsymbol{a}(u_n,D^\sigma u)\to \boldsymbol{a}(u, D^\sigma u)$ in ${L^p}'(\R^d;\R^d)$, we deduce
\begin{equation}\label{LiminfInequalityForA}
	\begin{aligned}
		&\liminf_{n\to\infty}{\int_{\R^d}{\boldsymbol{a}(u_n,D^{s_n} u_n)\cdot (D^{s_n} u_n-D^\sigma v)}\,dx}\\
		&\geq \liminf_{n\to\infty}{\int_{\R^d}{\left(\boldsymbol{a}(u_n,D^{s_n} u_n)-\boldsymbol{a}(u_n,D^\sigma u)\right)\cdot (D^{s_n} u_n-D^\sigma u)}\,dx}\\
		&\qquad\quad+\lim_{n\to\infty}{\int_{\R^d}{\left(\boldsymbol{a}(u_n,D^\sigma u)-\boldsymbol{a}(u,D^\sigma u)\right)\cdot (D^{s_n} u_n-D^\sigma u)}\,dx}\\
		&\qquad\qquad\qquad+\lim_{n\to\infty}{\int_{\R^d}{\boldsymbol{a}(u,D^\sigma u)\cdot (D^{s_n} u_n-D^\sigma u)}\,dx}+\liminf_{n\to\infty}{\int_{\R^d}{\boldsymbol{a}(u_n,D^{s_n} u_n)\cdot D^\sigma(u-v)}\,dx}\\
		&\geq \lim_{n\to\infty}{\int_{\R^d}{\boldsymbol{a}(u_n,D^\sigma u)\cdot (D^{s_n} u_n-D^\sigma u)}\,dx}+\int_{\R^d}{\boldsymbol{a}(u,D^\sigma u)\cdot D^\sigma (u-v)}\,dx\\
		&=\int_{\R^d}{\boldsymbol{a}(u,D^\sigma u)\cdot D^\sigma (u-v)}\,dx,
	\end{aligned}
\end{equation}
by applying the inequality \eqref{EstimateRelatesSnWithClassical}.

On the other hand, by Lemma \ref{lemma:ConstructionSolutionOmegaBounded} the $u_n$ are bounded in $\Lambda^{s^*,p}_0(\Omega)$. By Sobolev embeddings we have then $u_n\rightharpoonup u$ in ${L^{{p^*_{s^*}}}}(\Omega)$ and, by compactness, also $u_n\to u$ in $L^{(q_2-1)({p^*_{s^*}})'}(\Omega)$, since $q_2<p^*_{s^*}=\frac{dp}{d-s^* p}$. This implies, by the growth assumption \eqref{eq:upperboundForB},  $b(u_n)\to b(u)$ in ${L^{({p^*_{s^*}})'}}(\Omega)$, and consequently
\begin{equation}\label{LimitIdentityForB}
	\lim_{n\to\infty}{\int_{\Omega}{b(u_n)(u_n-v)}\,dx}=\int_{\Omega}{b(u)(u-v)}\,dx.
\end{equation}
For each $v\in K$, let $v_n\in K_n$ be such that $\Pi_{s_n}(v_n)\to \Pi_{\sigma}(v)$. Combining the limits \eqref{LiminfInequalityForA} and \eqref{LimitIdentityForB} with the monotonicity and growth condition of $\boldsymbol{a}$, \eqref{eq:monotonicity} and \eqref{eq:upperboundForA}, respectively, and the fact that $u_n$ solves \eqref{eq:VariationalProblemWeakConvergence}, we conclude that $u\in K$ is a solution of \eqref{eq:limitVariationalProblemWeakConvergence}, because
\begin{equation}\label{eq:fromLiminfIniequalityToVariationalInequality}
    \begin{aligned}
        &\int_{\R^d}{\boldsymbol{a}(u,D^\sigma u)\cdot D^\sigma (u-v)\, dx}+\int_{\Omega}{b(u)(u-v)}\,dx\\
        &\qquad\qquad\leq \liminf_{n\to\infty}{\int_{\R^d}{\boldsymbol{a}(u_n,D^{s_n} u_n)\cdot (D^{s_n} u_n-D^\sigma v)\,dx}+\int_{\Omega}{b(u_n)(u_n-v)}\,dx}\\
        &\qquad\qquad\leq \liminf_{n\to\infty}{\left(\int_{\R^d}{\boldsymbol{a}(u_n,D^{s_n} u_n)\cdot (D^{s_n} v_n-D^\sigma v)\,dx}+\int_{\Omega}{b(u_n)(v_n-v)}\,dx-\langle F_n,u_n-v_n\rangle_{s_n}\right)}\\
        &\qquad\qquad=\langle F, u-v\rangle_\sigma, \qquad \forall v\in K.
    \end{aligned}
\end{equation}

\end{proof}

\subsection{The pseudomonotone case with $0<s\leq 1$ and arbitrary $\Omega$}

In order to extend Mosco type results to arbitrary domains for general Leray-Lions operators we develop Browder's argument of Section \ref{sec:existence_browder} in the framework of $\Lambda^{s_n,p}_0(\Omega)$ with $s_n\to\sigma$.

\begin{theorem}\label{theorem:weak_stability_variational_inequality_arbitrary_domains}
    Let $\Omega\subset\R^d$ be an open set, $1<p<+\infty$ and consider a sequence $\{s_n\}\subset(s^*,1]$ with $s^*>0$ and $s_n\to\sigma\in[s^*,1]$. Let $K_n$ and $K$ be convex sets satisfying \eqref{eq:generalized_mosco_convergence} and $F_n={f_0}_n-D^{s_n}\cdot\boldsymbol{f}_n\in \Lambda^{-s_n,p'}(\Omega)$ satisfying \eqref{eq:convergence_data}. If the two Carathéodory functions $\boldsymbol{a}:\R^d\times\R\times\R^d\to\R^d$ and $b:\R^d\times\R\to\R^d$ satisfy the growth conditions \eqref{eq:upper_bound_for_A_when_omega_arbitrary} and \eqref{eq:upper_bound_for_B_when_omega_arbitrary}, the strong coercivity condition \eqref{eq:strongCorcivity} and $\boldsymbol{a}$ is strictly monotone with respect to the last variable as in \eqref{eq:strictMonotonicity}, then, from the sequence of solutions $\{u_n\}$, with $u_n\in K_n$ satisfying
    \begin{equation}\label{eq:VariationalProblemWeakConvergenceArbitraryDomains}
        \langle -D^{s_n}\cdot\boldsymbol{a}(u_n, D^{s_n} u_n))+b(u_n, D^{s_n}u_n)-F_n, v-u_n\rangle_{s_n}\geq 0,\quad \forall v\in K_n,
    \end{equation}
    one can extract a subsequence, still denoted by $\{u_n\}$, satisfying
    \begin{equation*}
        u_n\rightharpoonup u \mbox{ in }\Lambda^{t,p}_0(\Omega), \mbox{ for } 0\leq t\leq s^*, \quad \mbox{ and } \quad D^{s_n}u_n\rightharpoonup D^\sigma u \mbox{ in } L^p(\R^d;\R^d)
    \end{equation*}
    where, for $F=f_0-D^\sigma\cdot\boldsymbol{f}$, $u\in K$ solves
    \begin{equation}\label{eq:limitVariationalProblemWeakConvergenceArbitraryDomains}
        \langle -D^\sigma\cdot\boldsymbol{a}(u, D^\sigma u)+b(u, D^\sigma u)-F, v-u\rangle_\sigma\geq 0,\quad \forall v\in K.
    \end{equation}
\end{theorem}
\begin{remark}\label{rem:local_strong_convergence_browder}
    The main difference between Theorem \ref{theorem:weak_stability_variational_inequality} and Theorem \ref{theorem:weak_stability_variational_inequality_arbitrary_domains} is the weak convergence of $u_n$ to $u$ in $\Lambda^{t,p}_0(\Omega)$, $0\leq t<s^*$, because $\Omega$ can be unbounded in the second one. In this framework we still have the local strong convergence of the sequence $u_n$, using Rellich-Kondrachov's Proposition \ref{prop:local_rellich_kondrachov}. In fact, if $s^*p\leq d$, for any subset with finite measure $\omega\subset\Omega$, $u_n\to u$ in $L^q(\omega)$ for $q<p^*_{s^*}$ $(p^*_{s^*}=+\infty$ if $s^*p=d)$. On the other hand, if $s^*p> d$, for any compact subset $\Tilde{K}\subset\Omega$, $u_n\to u$ in $C^{0,\beta}(\Tilde{K})$ for $\beta<s^*-\frac{d}{p}$.
\end{remark}

\begin{lemma}\label{lemma:ConstructionSolutionOmegaUnbounded}
	Under the hypotheses of the Theorem \ref{theorem:weak_stability_variational_inequality_arbitrary_domains},  we can extract a subsequence from $\{u_n\}$ (still denoted by $u_n$) such that $\Pi_{s_n}(u_n)\rightharpoonup \Pi_\sigma(u)$ in $\L^p(\Omega)$ for some $u\in K$.
\end{lemma}
\begin{proof}
    The proof of this Lemma is similar to the proof of Lemma \ref{lemma:ConstructionSolutionOmegaBounded}. In fact, for \eqref{InequalityForAPrioriBound}, the only difference is that we have to the strong coercivity hypothesis made on $\boldsymbol{a}$ and $b$, which gives us the estimate
    \begin{multline*}
        \int_{\R^d}{\boldsymbol{a}(u_n,D^{s_n} u_n)\cdot D^{s_n} v_n}\,dx+\int_\Omega{b(u_n)v_n}\,dx+\langle F_n, u_n-v_n\rangle_{s_n}\\
        \geq \alpha \|D^{s_n} u_n\|^p_{L^p(\R^d;\R^d)}+\beta\|u_n\|^p_{L^p(\Omega)}-\|k\|_{L^1(\R^d)}.
    \end{multline*}
    For the estimates \eqref{UpperBoundAprioriA} and \eqref{UpperBoundAprioriL}, we argue in the same way now with the norm \eqref{eq:norm_lions_calderon_spaces} with the necessary adaptations.
\end{proof}

\begin{proof}[Proof of Theorem \ref{theorem:weak_stability_variational_inequality_arbitrary_domains}.]
	Since we are dealing with potentially unbounded domains, we use a local version of the Rellich-Kondrachov theorem and we consider an increasing sequence of bounded sets $\{\omega_j\}\subset\R^d$ such that $\Omega=\bigcup_{j=1}^\infty{\omega_j}$. Then, $\Lambda^{s_n,p}_0(\Omega)\subset\Lambda^{s^*,p}_0(\Omega)\Subset L^p(\omega_j)$ for each $j\in\N$. Using a diagonal argument, we can construct a subsequence, still denoted by $\{u_n\}$, such that $u_n\to u$ in $L^p(\omega_j)$ for any $j\in\N$. Consequently, we have, up to a subsequence, that $u_n(x)\to u(x)$ a.e. $x\in\Omega$. 

    Let us now consider the function
    \begin{equation}\label{eq:pInTermsOfG}
        p_n(x)=\boldsymbol{a}(u_n, D^{s_n}u_n)\cdot D^{s_n} u_n(x)+b(u_n, D^{s_n} u_n)u_n(x)-g_n^1(x)-g_n^0(x).
    \end{equation}
    where
    \begin{equation*}
        g_n^1(x)=\boldsymbol{a}(u, D^\sigma u)\cdot(D^{s_n} u_n-D^\sigma u)(x)+\boldsymbol{a}(u_n, D^{s_n} u_n)\cdot D^\sigma u(x),
    \end{equation*}
    and
    \begin{equation*}
        g_n^0(x)=b(u, D^\sigma u)(u_n-u)(x)+b(u_n, D^{s_n} u_n)u(x).
    \end{equation*}

    Following \cite{browder1977PseudoMonotone} we study the convergence of $p_n$, by using Vitalli's theorem to show first $p_n^-\to 0$ in $L^1(\R^d)$, and afterwards deducing $p_n\to 0$ a.e. in $\R^d$.
	
    We start by observing that the $g_n^1$ and $g_n^0$ are both equi-integrable in $\R^d$. This is an easy consequence of the growth conditions of $\boldsymbol{a}$ and $b$, \eqref{eq:upper_bound_for_A_when_omega_arbitrary} and \eqref{eq:upper_bound_for_B_when_omega_arbitrary}, respectively, and the a priori estimates on $u_n$ and $D^{s_n} u_n$. 
 
    Moreover, the growth conditions and the strong coercivity hypothesis applied to \eqref{eq:pInTermsOfG} imply the existence of constants $C_2, C_3>0$  and a function $\tilde{k}\in L^1(\R^d)$ such that 
    \begin{equation}\label{eq:pointwiseLowerBoundPs}
		p_n (x)\geq C_2(|u_n(x)|^p+|D^{s_n} u_n(x)|^p)-C_3(|u(x)|^p+|D^\sigma u(x)|^p+\tilde{k}(x)).
    \end{equation}
    So, restricting to $x\in\mathrm{supp}(p_n^-)$, we have
    \begin{equation}\label{eq:pointwiseUniformEstimateInS}
		C_2 |D^{s_n} u_n(x)|^p\leq C_2(|D^{s_n} u_n(x)|^p+|u_n|^p(x))\leq C_3(|u(x)|^p+|D^\sigma u(x)|^p)+\tilde{k}(x)= \rho(x)
	\end{equation}
	where $\rho$ is a function independent of $s$ and finite a.e. in $\R^d$ (otherwise it would not be integrable). 
	
	Writing $p_n(x)=q_n(x)+r_n(x)+t_n(x)$, with
	\begin{align*}
		&q_n(x) =(\boldsymbol{a}(u_n, D^{s_n} u_n)-\boldsymbol{a}(u_n, D^\sigma u))\cdot(D^{s_n} u_n -D^\sigma u)(x)\\
		&r_n(x) =(\boldsymbol{a}(u_n(x),D^\sigma u)-\boldsymbol{a}(u(x),D^\sigma u))\cdot(D^{s_n} u_n-D^\sigma u)(x)\\
		&t_n(x) =(b(u_n, D^{s_n}u_n)-b(u, D^\sigma u)(x))(u_n(x)-u(x)),
	\end{align*}
	and denoting $\chi_s=\chi_{\{x:\,p_n^->0\}}$, we can write a.e. $x\in\R^d$.
	\begin{equation}\label{eq:decompositionPsWithCharacteristics}
		-p_n^-=\chi_s q_n+\chi_s r_n+\chi_s t_n.
	\end{equation}
	By the monotonicity of $\boldsymbol{a}$ we have $q_n(x)\geq 0$ and so $\chi_s q_n\geq 0$. Since $\boldsymbol{a}$ is a Carathéodory function and \eqref{eq:pointwiseUniformEstimateInS} holds in $\{x:\,p_n^->0\}$ we have $\chi_s r_n\to 0$, as $u_n(x)\to u(x)$ a.e. Similarly, the growth conditions of $b$ with \eqref{eq:pointwiseUniformEstimateInS} imply $\chi_s t_n\to 0$ a.e. in $\R^d$. Consequently we get $p_n^-\to 0$ a.e. in $\R^d$.

	To prove $p_n^-\to 0$ in $L^1(\R^d)$, we show that $p_n^-$ is equi-integrable and apply Vitali's convergence theorem. In fact, the strong coercivity assumption on $\boldsymbol{a}$ and $b$, \eqref{eq:strongCorcivity}, and \eqref{eq:pInTermsOfG}, yields the estimate
	\begin{equation*}
		p_n(x)\geq\alpha |D^{s_n} u_n(x)|^p+\beta|u_n(x)|^p+k_2(x)-g_n^1(x)-g_n^0(x)\geq -|k_2(x)|-|g_n^1(x)|-|g_n^0(x)|,
	\end{equation*}
	and consequently,
	\begin{equation*}
		0\leq p_n^-(x)\leq |k_2(x)|+|g_n^1(x)|+|g_n^0(x)|, \quad \mbox{ a.e. in } \R^d.
	\end{equation*}
	This implies that $p_n^-$ is equi-integrable, since both $\{g_n^0\}$ and $\{g_n^1\}$ are equi-integrable.

    Now, we can easily prove $p_n(x)\to 0$ for a.e. $x\in\R^d$. Indeed, using the monotonicity of $\boldsymbol{a}$ in the last variable, we have $p_n\leq 0$, and, since $p_n^-\to 0$ in $L^1(\R^d)$, we get
    \begin{equation*}
		0\leq \limsup{\int_{\R^d}{p_n^+}\,dx}=\limsup{\int_{\R^d}{p_n}\,dx}-\lim{\int_{\R^d}{p_n^-}\,dx}\leq 0.
    \end{equation*}
    Therefore $p_n^+\to 0$ in $L^1(\R^d)$, and consequently, up to a subsequence, also $p_n(x)\to 0$ for a.e. $x\in\R^d$. This pointwise convergence with the estimate \eqref{eq:pointwiseLowerBoundPs} yields a uniform bound on $|D^{s_n} u_n(x)|^p$ a.e. in $\R^d$ and we can define a function $\xi(x)$ as the a.e. pointwise limit of $D^{s_n} u_n(x)$ as $n\to\infty$. Moreover, since $p_n(x)\to 0$, $r_n(x)\to 0$ and $t_n(x)\to 0$ for a.e. $x\in\R^d$ as $n\to\infty$, we have 
    \begin{equation*}
		\lim{p_n(x)}=\lim{q_n(x)}=(\boldsymbol{a}(x,u(x), \xi(x))-\boldsymbol{a}(x,u(x),D^\sigma u(x)))\cdot (\xi(x)-D^\sigma u(x))=0,
    \end{equation*}
    which, together with the strict monotonicity of $\boldsymbol{a}$ in the third variable, implies $\xi=D^\sigma u$ a.e. $x\in\R^d$.

    We rewrite 
     \begin{equation*}
        p_n=(\boldsymbol{a}(u_n, D^{s_n} u_n)-\boldsymbol{a}(u, D^\sigma u))\cdot (D^{s_n}u_n- D^\sigma u)+(b(u_n, D^{s_n}u_n)-b(u, D^\sigma u))(u_n-u)
    \end{equation*}
    and use $u_n\to u$ a.e. in $\Omega$ and $D^{s_n} u_n\to D^\sigma u$ a.e. in $\R^d$ to obtain
    \begin{align*}
        &\liminf_{n\to\infty}{\left(\int_{\R^d}{\boldsymbol{a}(u_n,D^{s_n} u_n)\cdot (D^{s_n} u_n-D^\sigma v)}\,dx+\int_\Omega{b(u_n, D^{s_n}u_n)(u_n-v)}\,dx\right)}\\
        &\geq\lim_{n\to\infty}{\left(\int_{\R^d}{p_n}\,dx+\int_{\R^d}{\boldsymbol{a}(u, D^\sigma u)\cdot (D^{s_n}u_n- D^\sigma u)}\,dx+\int_\Omega{b(u, D^\sigma u)(u_n-u)}\,dx\right)}\\
        &\qquad\qquad\qquad\qquad+\liminf_{n\to\infty}{\left(\int_{\R^d}{\boldsymbol{a}(u_n,D^{s_n} u_n)\cdot D^\sigma (u-v)}\,dx+\int_\Omega{b(u_n, D^{s_n} u_n)(u-v)}\,dx\right)}\\
        &\geq \int_{\R^d}{\lim_{n\to\infty}{(\boldsymbol{a}(u_n,D^{s_n} u_n))}\cdot D^\sigma (u-v)}\,dx+\int_\Omega{\lim_{n\to\infty}{(b(u_n, D^{s_n}u_n))}(u-v)}\,dx\\
        &=\int_{\R^d}{\boldsymbol{a}(u,D^\sigma u)\cdot D^\sigma (u-v)+b(u, D^\sigma u)(u-v)}\,dx.
    \end{align*}
    We conclude the proof as in \eqref{eq:fromLiminfIniequalityToVariationalInequality}.
\end{proof}

\subsection{The monotone case with $0\leq s\leq 1$ and $\Omega$ arbitrary}

Without the strict monotonicity of Subsection \ref{sec:existence_monotone}, we can obtain a Mosco type theorem with a weak convergence of solutions in the case of operators simply monotone, up to the lower limit $s=0$.

\begin{theorem}\label{thm:general_stability_s_to_0}
    Let $\Omega\subset\R^d$ be an open set, $1<p<+\infty$ and consider a sequence $\{s_n\}\subset[0,1]$ with $s_n\to\sigma\in[0,1]$. Let $K_n$ and $K$ be convex sets satisfying \eqref{eq:generalized_mosco_convergence} and $F_n={f_0}_n-D^{s_n}\cdot\boldsymbol{f}_n\in \Lambda^{-s_n,p'}(\Omega)$ satisfying \eqref{eq:convergence_data}. If the two Carathéodory functions $\boldsymbol{a}=\boldsymbol{a}(x,\xi):\R^d\times\R^d\to\R^d$ and $b=b(x,r):\R^d\times\R\to\R^d$, both monotone with respect to their last variable, satisfy the growth conditions \eqref{eq:upper_bound_for_A_when_omega_arbitrary} and \eqref{eq:upper_bound_for_B_when_omega_arbitrary}, respectively, and the strong coercivity hypothesis \eqref{eq:strongCorcivity}, then from the sequence of solutions $\{u_n\}$, with $u_n\in K_n$ satisfying
    \begin{equation}\label{eq:approximating_problem_s_to_0}
        \langle -D^{s_n}\cdot\boldsymbol{a}( D^{s_n} u_n)+b(u_n)-F_n, v-u_n\rangle_{s_n}\geq 0,\quad \forall v\in K_n,
    \end{equation}
    one can extract a subsequence, still denoted by $\{u_n\}$, with the property that
    \begin{equation}\label{eq:convergence_solutions_s_to_0}
        u_n\rightharpoonup u \mbox{ in } L^p(\Omega) \quad \mbox{ and } \quad D^{s_n}u_n\rightharpoonup D^\sigma u \mbox{ in } L^p(\R^d;\R^d)
    \end{equation}
    where, with $F=f-D^\sigma\cdot\boldsymbol{f}$, $u\in K$ solves
    \begin{equation}\label{eq:LimitProblemSTo0}
        \langle -D^\sigma\cdot\boldsymbol{a}( D^\sigma u)+b(u)-F, v-u\rangle_\sigma\geq 0,\quad \forall v\in K.
    \end{equation}
\end{theorem}
\begin{remark}\label{rem:general_convergence_s_to_0}
    If $\sigma>s^*>0$, as in Theorem \ref{theorem:weak_stability_variational_inequality_arbitrary_domains}, the first weak convergence in \ref{eq:convergence_solutions_s_to_0} also holds weakly in $\Lambda^{t,p}_0(\Omega)$, for $0<t<s^*$. When $\Omega$ is bounded, using Poincaré inequality, we can relax the coercivity assumption \eqref{eq:strongCorcivity} by letting $\beta>-\alpha(s/C_P)^p$, as in Remark \ref{rem:limit_existence_with_poincare}, with $\alpha>0$. If $\sigma>s^*>0$, as in Theorem \ref{theorem:weak_stability_variational_inequality}, we can improve the first weak convergence in \ref{eq:convergence_solutions_s_to_0} by $u_n\to u$ in $\Lambda^{t,p}_0(\Omega)$, for $0<t<s^*$.
\end{remark}
\begin{remark}
    If $\sigma>0$ in unbounded domains a similar observation as in Remark \ref{rem:local_strong_convergence_browder} is also valid here. As for $\sigma=0$, in general one should not expect any strong convergence of $u_n$ to $u$ due to the lack of compactness, since the uniform estimates on the sequences ${u_n}$ and $D^{s_n}u_n$ are only in $\Lambda^{0,p}_0(\Omega)=L^p_0(\Omega)$ and in $L^p(\R^d;\R^d)$, respectively.
\end{remark}

\begin{proof}
    Using similar arguments to those employed in the proof of the Lemma \ref{lemma:ConstructionSolutionOmegaUnbounded} one can deduce that $\|u_n\|_{\Lambda^{s_n,p}_0(\Omega)}$ is uniformly bounded with respect to $s_n$. Consequently, there exist two functions $u\in L^p(\R^d)$ and $\eta\in L^p(\R^d;\R^d)$ such that
	\begin{equation}\label{eq:weakLimitsSolutionAndGradient}
        u_n\rightharpoonup u \mbox{  in } L^p(\R^d)\quad \mbox{ and }\quad D^{s_n} u_n\rightharpoonup \eta \mbox{  in } L^p(\R^d;\R^d).
    \end{equation}
    Notice that $\eta=D^\sigma u$ because
    \begin{equation}\label{eq:weak_continuity_gradient_through_divergence}
        \int_{\R^d}{D^{s_n} u_n \cdot\varphi}\,dx=-\int_{\R^d}{u_n D^{s_n}\cdot\varphi}\,dx\to-\int_{\R^d}{uD^\sigma\cdot\varphi}\,dx=\int_{\R^d}{D^\sigma u \cdot\varphi}\,dx, \quad \forall \varphi\in C^\infty_c(\R^d;\R^d).
    \end{equation}
    Hence we have $u\in K$, since $\Lambda^{s_n,p}_0(\Omega)\supset K_n\xrightarrow{s_n-M} K\subset \Lambda^{\sigma,p}_0(\Omega)$.
    
    Next, we prove that $u$ is a solution of \eqref{eq:LimitProblemSTo0}, by applying Minty's lemma. Indeed, let us consider an arbitrary function $w\in K$ and a sequence $\{w_n\}$ such that $w_n\in K_n$ and $\Pi_{s_n}(w_n)\to \Pi_\sigma(w)$ in $\L^p(\Omega)$. Due to the monotonicity of $\boldsymbol{a}$ and $b$, we have
    \begin{multline*}
        \int_{\R^d}{\boldsymbol{a}(D^{s_n} w_n)\cdot D^{s_n}(w_n-u_n)}\,dx+\int_\Omega{b(w_n)(w_n-u_n)}\,dx\\
        \geq \int_{\R^d}{\boldsymbol{a}(D^{s_n} u_n)\cdot D^{s_n} (w_n-u_n)}\,dx+\int_\Omega{b(u_n)(w_n-u_n)}\,dx=\langle F_n,w_n-u_n\rangle_{s_n}.
    \end{multline*}
    Using $\Pi_{s_n}(w_n)\to \Pi_\sigma(w)$ in $\L^p(\Omega)$, the continuity of the operators $\boldsymbol{a}:L^p(\R^d;\R^d)\to {L^p}'(\R^d;\R^d)$ and $b:L^p(\R^d)\to {L^p}'(\R^d)$, the convergence of $f_n\to f$ in ${\L^p}'(\R^d)$, and the weak limits \eqref{eq:weakLimitsSolutionAndGradient}, we get
    \begin{equation}\label{InequalityBeforeDensityGeneralOperator}
        \int_{\R^d}{\boldsymbol{a}(D^\sigma w)\cdot D^\sigma(w-u)}\,dx+\int_\Omega{b(w)(w-u)}\,dx\geq \langle F,\varphi-u\rangle_\sigma.
    \end{equation}
    Setting in \eqref{InequalityBeforeDensityGeneralOperator} $w=u+\theta(v-u)$, with $0<\theta<1 $ and an arbitrary $v\in K$, dividing by $\theta$ and letting $\theta\to 0$, we conclude that $u\in K$ is a weak solution of \eqref{eq:LimitProblemSTo0}:
    \begin{equation*}
        \int_{\R^d}{\boldsymbol{a}(D^\sigma u)\cdot D^\sigma(v-u)}\,dx+\int_\Omega{b(u)(v-u)}\,dx\geq \langle F,v-u\rangle_\sigma, \quad\forall v\in K.
    \end{equation*}
\end{proof}

\subsubsection{Stronger convergence for $p$-Laplacian type operators}

In the framework of Corollary \ref{UniquenessFractionalPLaplacian}, we can prove a strong convergence result for the fractional gradients in a Mosco type theorem for $p$-Laplacian operators with lower order terms, in addition to the conclusions of Remark \ref{rem:general_convergence_s_to_0}.

\begin{theorem}\label{theorem:StrongMosco}
    Let us assume the hypothesis of Theorem \ref{thm:general_stability_s_to_0} and let us also assume that $\boldsymbol{a}$ and $b$ satisfy \eqref{eq:monotonicityPLaplacian} and \eqref{eq:monotonicityb}, with $\beta_p>-\alpha_p(s/C_P)^p$  when $\Omega$ is bounded and with $\beta_p>0$ when $\Omega$ is unbounded, respectively. Then, the sequence of solutions $u_n\in K_n$ of \eqref{eq:approximating_problem_s_to_0} satisfies
    \begin{equation*}
        u_n\to u \mbox{ in } L^p(\Omega)\quad \mbox{ and }\quad D^{s_n}u_n\to D^\sigma u \mbox{ in } L^p(\R^d;\R^d),
    \end{equation*}
    where $u\in K$ is the unique solution of \eqref{eq:LimitProblemSTo0}.
\end{theorem}
\begin{proof}
    From Theorem \ref{theorem:weak_stability_variational_inequality} we know that there exists a subsequence $\{u_n\}$ such that $\Pi_{s_n}(u_n)\rightharpoonup \Pi_\sigma(u)$ in $\L^p(\Omega)$ where $u$ is the solution to \eqref{eq:LimitProblemSTo0}. We only need to prove that this weak convergence is a strong convergence in $\L^p(\Omega)$. 
    
    Using the assumption $K_n\xrightarrow{s_n-M}K$ we may consider a sequence $\{w_n\}$, with $w_n\in K_n$, such that $\Pi_{s_n}(w_n)\to \Pi_\sigma(u)$ in $\L^p(\Omega)$. Since $u_n$ are solutions to \eqref{eq:approximating_problem_s_to_0}, we have
    \begin{multline}\label{eq:upper_estimate_strong_convergence}
            \langle D^{s_n}\cdot\boldsymbol{a}(D^{s_n} u_n)-D^{s_n}\cdot\boldsymbol{a}(D^{s_n}w_n)+b(w_n)-b(u_n),w_n-u_n\rangle_{s_n}\\
            \leq\int_{\R^d}{(\boldsymbol{a}(D^{s_n}w_n)-\boldsymbol{f}_n)\cdot(D^s w_n -u_n)}\,dx+\int_\Omega{(b(w_n)-{f_0}_n)(w_n-u_n)}\,dx.
    \end{multline}
    On the other hand, since $\boldsymbol{a}$ and $b$ satisfy \eqref{eq:monotonicityPLaplacian} and \eqref{eq:monotonicityb}, respectively, we have, for $p\geq 2$,
    \begin{multline}\label{eq:lower_estimate_p_geq_2_strong_convergence}
        \langle D^{s_n}\cdot\boldsymbol{a}(D^{s_n} u_n)-D^{s_n}\cdot\boldsymbol{a}(D^{s_n}w_n)+b(w_n)-b(u_n),w_n-u_n\rangle_{s_n}\\
        \geq \alpha_p\|D^{s_n}(u_n-w_n)\|^p_{L^p(\R^d;\R^d)}+\beta_p\|u_n-w_n\|^p_{L^p(\Omega)},
    \end{multline}
    and, for $1<p<2$, also
    \begin{multline}\label{eq:lower_estimate_p_leq_2_strong_convergence}
        \langle D^{s_n}\cdot\boldsymbol{a}(D^{s_n} u_n)-D^{s_n}\cdot\boldsymbol{a}(D^{s_n}w_n)+b(w_n)-b(u_n),w_n-u_n\rangle_{s_n}\\
        \geq \frac{1}{2^{\frac{2-p}{p}}(C_u+C_w)^{\frac{2-p}{p}}}
        \left(\alpha_p\|D^{s_n}(u_n-w_n)\|^p_{L^p(\R^d;\R^d)}+\beta_p\|u_n-w_n\|^p_{L^p(\Omega)}\right).
    \end{multline}
    Here $C_u, C_w>0$ are constants, independent of $n$, such that $\|u_n\|^p_{L^p(\Omega)}+\|D^{s_n}u_n\|^p_{L^p(\R^d;\R^d)}\leq C_u$ and $\|w_n\|^p_{L^p(\Omega)}+\|D^{s_n}w_n\|^p_{L^p(\R^d;\R^d)}\leq C_w$.
    Combining \eqref{eq:upper_estimate_strong_convergence} with \eqref{eq:lower_estimate_p_geq_2_strong_convergence}, when $p\geq 2$, and with \eqref{eq:lower_estimate_p_leq_2_strong_convergence}, when $1<p<2$, and using $b(w_n)\to b(u)$ in $L^{p'}(\Omega)$ and $\boldsymbol{a}(D^{s_n} w_n)\to \boldsymbol{a}(D^\sigma u)$ in ${L^{p'}}(\R^d;\R^d)$, respectively by the continuity of $\boldsymbol{a}$ and $b$, together with the convergences \eqref{eq:convergence_data}, we obtain
    \begin{equation*}
        \lim_{n\to\infty}{\|u_n-w_n\|_{\Lambda^{s_n,p}_0(\Omega)}}=\lim_{n\to\infty}{\left(\beta_p\|u_n-w_n\|^p_{L^p(\Omega)}+\|D^{s_n}(u_n-w_n)\|^p_{L^p(\R^d;\R^d)}\right)^{1/p}}=0.
    \end{equation*}
    Hence, using Poincaré inequality when $\Omega$ is bounded (see Remark \ref{rem:limit_existence_with_poincare}),
    recalling  $\Pi_{s_n}(w_n)\to \Pi_\sigma(u)$, we conclude $\Pi_{s_n}(u_n)\to\Pi_\sigma(u)$ in $\L^p(\Omega)$ and the convergence of the whole sequence is a consequence of the uniqueness of the limit problem \eqref{eq:LimitProblemSTo0}, by Corollary \ref{UniquenessFractionalPLaplacian}.
\end{proof}

\begin{remark}
    The typical example of operators that satisfy the hypothesis of the last Theorem are those in which $D^{s_n}\cdot\boldsymbol{a}(x,D^{s_n} u_n)=D^{s_n}\cdot(\alpha(x)|D^{s_n} u|^{p-2}D^{s_n} u)$ is the heterogeneous fractional $(s_n,p)$-Laplacian operator, with $0<\alpha_*\leq \alpha(x)\leq \alpha^*$, and the lower term $b(x,u)=\beta(x) |u|^{p-2}u$, with $0<\beta_*\leq \beta(x)\leq \beta^*$, if $\Omega$ is unbounded, or $-\alpha_*(s/C_P)^p<\beta_*\leq \beta(x)\leq \beta^*$, if  $\Omega$ is bounded.
\end{remark}

\begin{remark}
    In the context of problems without unilateral constraints where $K_n=\Lambda^{s_n,p}_0(\Omega)$ and $\mathscr{A}_{s_n}$ is of potential type, like in Remark \ref{remark:3.3}, instead of using Mosco convergence, one can use $\Gamma$-convergence, as in \cite{bellido2020gamma}.
\end{remark}

\section{Applications to unilateral problems}\label{sec:applications}
\subsection{Examples of generalized Mosco convergence with $s$ in $[0,1]$}\label{subsec:examples_mosco}

In order to illustrate the Mosco type convergence theorems, in this section we provide some examples of families of convex sets that converge in the generalized sense of Mosco. The first example corresponds to the case in which the convex sets coincide with the Lions-Calderón spaces. The two examples of unilateral constraints, leading to variational inequalities, will be given by the translations of cones in those spaces, corresponding to obstacle problems, and by a special case of fractional gradient constraints.

\begin{theorem}[No obstacle]\label{thm:no_obstacle}
    Let $\Omega\subset\R^d$ be any open set, $p\in(1,+\infty)$ and a sequence $\{s_n\}\subset [0,1]$ with $s_n\to \sigma\in [0,1]$. Then $\Lambda^{s_n,p}_0(\Omega)$ converge in the generalized sense of Mosco to $\Lambda^{\sigma,p}_0(\Omega)$, i.e.
    \begin{equation}\label{eq:mosco_convergence_lions_calderon}
        \Lambda^{s_n,p}_0(\Omega)\xrightarrow{s_n-M}\Lambda^{\sigma,p}_0(\Omega) \quad \mbox{ as } s_n\to\sigma.
    \end{equation}
\end{theorem}
\begin{proof}
    Condition \ref{M2} follows immediately from the fact that $K=\Lambda^{\sigma,p}_0(\Omega)$.
    
    To prove \ref{M1}, let $u\in \Lambda^{\sigma,p}_0(\Omega)$ and $\{\varphi_j\}\subset C^\infty_c(\Omega)$ be a sequence such that $\varphi_j\to u$ in $\Lambda^{\sigma,p}_0(\Omega)$ as $j\to\infty$. For each $m\in\N$, there exists $q_1(m)\in\N$ such that $j\geq q_1(m)$, for which
    \begin{equation*}
        \|\varphi_j-u\|_{L^p(\R^d)}+\|D^\sigma \varphi_j- D^\sigma u\|_{L^p(\R^d;\R^d)}<1/m
    \end{equation*}
    For each $j\geq q_1(m)$, by Proposition \ref{prop:continuous_dependence_fractional_gradiente_on_s} there exists a positive integer $q_2(m,j)\geq j$ such that for all $k\geq q_2(m,j)$ and
    \begin{equation*}
        \|D^{s_k} \varphi_j- D^\sigma \varphi_j\|_{L^p(\R^d;\R^d)}\leq 1/m.
    \end{equation*}
    
    Define the sequence
    \begin{equation*}
        u_n=\begin{cases}
            \varphi_1 & \mbox{ if } n< q(1)\\
            \varphi_{q(k)} & \mbox{ if } q(k)\leq n<q(k+1), \quad \mbox{ for } k\in \N
        \end{cases}
    \end{equation*}
    where $q(k)=q_2(k,q_1(k))$. For $n\geq q_2(m)\geq q_1(m)$, then
    \begin{multline*}
        \|u_n-u\|_{L^p(\Omega)}+\|D^{s_n} u_n- D^\sigma u\|_{L^p(\R^d;\R^d)}\\
        \leq \|u_n-u\|_{L^p(\R^d)}+\|D^{s_n} u_n-D^\sigma u_n\|_{L^p(\R^d;\R^d)}+\|D^\sigma u_n- D^\sigma u\|_{L^p(\R^d;\R^d)}\leq 2/m,
    \end{multline*}
    which implies that $\Pi_{s_n}(u_n)\to \Pi_\sigma(u)$ in $\L^p(\Omega)$.
\end{proof}

\begin{remark}
    In the case $s_n\leq\sigma$, the proof of \eqref{eq:mosco_convergence_lions_calderon} is much simpler, since \ref{M1} holds directly from Proposition \ref{prop:continuous_dependence_fractional_gradiente_on_s} by considering the sequence $u_n\equiv u\in \Lambda^{\sigma,p}_0(\Omega)\subset\Lambda^{s_n,p}_0(\Omega)$. 
    
    In the case $s_n\geq\sigma$ we can also give a different proof of \eqref{eq:mosco_convergence_lions_calderon}. Let $u_n\in\Lambda^{s_n,p}_0(\Omega)$ be the solution of 
    \begin{equation}\label{AuxProblemForConvergence}
        \int_{\R^d}{|D^{s_n} u_n|^{p-2}D^{s_n} u_n\cdot D^{s_n} v}\,dx+\int_\Omega{|u_n|^{p-2}u_n v}\,dx = \int_{\Omega}{f_0 v}\,dx+\int_{\R^d}{\boldsymbol{f}\cdot D^{s_n} v}\,dx, \quad \forall v\in \Lambda^{s_n,p}_0(\Omega).
    \end{equation}
    Here the functions $f_0$ and $\mathbf {f}=(f_1,...,f_d)$ are given such that
    \begin{equation*}
        \langle-\Delta^\sigma_p u+|u|^{p-2}
        _\sigma u, v\rangle_\sigma = \langle F, v\rangle_\sigma=\int_{\Omega}{f_0 v}\,dx+\int_{\R^d}{\boldsymbol{f}\cdot D^\sigma v}\,dx \quad \forall v\in\Lambda^{\sigma,p}_0(\Omega),
    \end{equation*}
    where the functional $F=-\Delta^\sigma_p u+|u|^{p-2}u\in \Lambda^{-\sigma,p'}(\Omega)\subset \Lambda^{-s_n,p'}(\Omega)$ and the $f_0,f_1,...,f_d\in {L^p}'(\R^d)$ can be obtained from Theorem \ref{CharacterizationDualInTermsOfFractionalDerivatives} to characterize $F$.
   
   We prove that, up to a subsequence, $\Pi_{s_n}(u_n)\to \Pi_\sigma(u)$ in $\L^p(\Omega)$. Testing \eqref{AuxProblemForConvergence} with $v=u_n$ we get
	\begin{equation}\label{APrioriEstimateToyProblem}
            \|u_n\|^p_{\Lambda^{s_n,p}_0(\Omega)}= \int_{\R^d}{|D^{s_n} u_n|^{p-2}D^{s_n} u_n\cdot D^{s_n} u_n}\,dx+\int_\Omega{|u_n|^{p}}\,dx= \int_{\Omega}{f_0 u_n}\,dx+\int_{\R^d}{\boldsymbol{f}\cdot D^{s_n} u_n}\,dx.
	\end{equation}
	This identity, together with Hölder's inequality, allow us to obtain a uniform estimate of $\|u_n\|_{\Lambda^{s_n,p}_0(\Omega)}$ with respect to $n$, and so, there exist two functions $\Tilde{u}\in L^p(\R^d)$ and $\eta\in L^p(\R^d;\R^d)$ such that
	\begin{equation}\label{weakLimitsSolutionAndGradient}
        u_n\rightharpoonup \Tilde{u} \mbox{  in } L^p(\R^d)\quad \mbox{ and }\quad D^{s_n} u_n\rightharpoonup \eta \mbox{  in } L^p(\R^d;\R^d).
    \end{equation}
	We conclude that $\eta=D^\sigma \Tilde{u}$ as in \eqref{eq:weak_continuity_gradient_through_divergence}. Using Minty's lemma and the uniqueness of the solution of the limit problem, we deduce that $\Tilde{u}=u$. 

    Finally, since $u$ is a weak-solution of \eqref{AuxProblemForConvergence} when $s_n$ is replaced by $\sigma$, and $u_n\rightharpoonup u$ in $L^p(\Omega)$, from \eqref{APrioriEstimateToyProblem}, we obtain
    \begin{equation*}
        \|D^{s_n} u_n\|^p_{L^p(\R^d)}+\|u_n\|^p_{L^p(\Omega)}=\langle F, u_n\rangle_{s_n}\to \langle F, u\rangle_\sigma=\|D^\sigma u\|_{L^p(\R^d;\R^d)}+\|u\|_{L^p(\Omega)}.
    \end{equation*}
    Consequently, since $\L^p(\Omega)$ is uniformly convex for $p\in(1,\infty)$, we conclude that $\Pi_{s_n} (u_n)\to \Pi_\sigma (u)$ in $\L^p(\Omega)$ strongly.

    In the general case $s_n\to\sigma$ we can replace the proof of \ref{M1} in Theorem \ref{thm:no_obstacle} by choosing $u_n=u$ if $s_n\leq\sigma$, and $u_n$ as the solution of \eqref{AuxProblemForConvergence} if $s_n\geq\sigma$.
\end{remark}

As a simple consequence of the Theorems of Section \ref{sec:existence_uniqueness} and \ref{thm:no_obstacle}, we get the following general stability result for fractional partial differential equations.
\begin{corollary}
    Let $\Omega\subset \R^d$ be an open set, $p\in(1,\infty)$ and $\{s_n\}\subset (0,1)$ be a sequence such that $s_n\to \sigma\in(0,1]$, under the assumptions of Theorem \ref{theorem:weak_stability_variational_inequality}, \ref{theorem:weak_stability_variational_inequality_arbitrary_domains} or \ref{thm:general_stability_s_to_0}, with $F_n={f_0}_n-D^{s_n}\cdot\boldsymbol{f}_n\in \Lambda^{-s_n,p'}(\Omega)$ satisfying \eqref{eq:convergence_data}.  If $u_n$ are solutions to
    \begin{equation}
        \begin{cases}
            -D^{s_n}\cdot\boldsymbol{a}(u_n,D^{s_n} u_n)+b(u_n,D^{s_n} u_n)=F_n & \mbox{ in } \Omega\\
            u=0 &\mbox{ on } \R^d\setminus \Omega,
        \end{cases}
        \label{eq:stability_partial_differential_equations}
    \end{equation}
    then one can extract a subsequence, still denoted by $\{u_n\}$, such that $\Pi_{s_n}(u_n)\rightharpoonup\Pi_\sigma(u)$ in $\L^p(\Omega)$, i.e.
    \begin{equation*}
        u_n\rightharpoonup u \mbox { in } L^p(\Omega) \quad \mbox{ and } \quad D^{s_n}u_n\rightharpoonup D^\sigma u \mbox{ in } L^p(\R^d;\R^d),
    \end{equation*}
    where $u\in\Lambda^{\sigma,p}_0(\Omega)$ is a solution of
    \begin{equation}
        \begin{cases}
            -D^\sigma\cdot\boldsymbol{a}(u,D^\sigma u)+b(u,D^{\sigma} u)=F & \mbox{ in } \Omega\\
            u=0 &\mbox{ on } \R^d\setminus \Omega.
        \end{cases}
        \label{eq:limit_stability_partial_differential_equations}
    \end{equation}
\end{corollary}

\begin{remark}\label{rem:strong_convergence_solutions_pdes_p_laplacian}
    If in the previous corollary we assume the hypothesis of Theorem \ref{thm:general_stability_s_to_0} for the functions $\boldsymbol{a}$ and $b$, we can then make also take sequences $s_n$ tending to $0$. Similarly, also under the stronger assumptions of Theorem \ref{theorem:StrongMosco} for fractional operators of $p$-Laplacian type we have $\Pi_{s_n}(u_n)\to\Pi_\sigma(u)$ in $\L^p(\Omega)$ strongly, for the whole sequence $s_n\to \sigma\in[0,1]$.
\end{remark}

For obstacle problems, we consider families of convex sets of the following type
\begin{equation*}
	K^s_{\geq\psi}:=\{v\in \Lambda^{s,p}_0(\Omega): v\geq \psi\},
\end{equation*}
with their convergence in the generalized sense of Mosco. For $\psi\in\Lambda^{s,p}_0(\Omega)$, the condition $\psi^+=\sup\{\psi,0\}\in\Lambda^{s,p}_0(\Omega)$ is sufficient for the convex set $K^s_{\geq\psi}$ to be non-empty. We start with the case, inspired in \cite[{Proposition~3 (1)}]{boccardo2021some}, where we have an upper bound on the obstacles $\psi_n$ by its limit $\psi$ and the fractional parameters $s_n$ are non-decreasing.

\begin{proposition}[With bound on the obstacles and order in the fractional parameters]\label{prop:GeneralizedWeak}
    Let $\Omega\subset\R^d$ be any open set, $p\in(1,+\infty)$ and a non-decreasing sequence $\{s_n\}\subset(0,1)$ with $s_n\nearrow \sigma\in (0,1]$. For a sequence of functions $\{\psi_n\}$ with $\psi_n\in \Lambda^{s_n,p}_0(\Omega)$ for all $n\in\N$, such that $\Pi_{s_n}(\psi_n)\rightharpoonup \Pi_\sigma(\psi)$ weakly in $\L^p(\Omega)$ for some $\psi\in \Lambda^{\sigma,p}_0(\Omega)$ and $\psi_{s_n}\leq \psi$, then $K^{s_n}_{\geq \psi_{s_n}}$ converges in the generalized sense of Mosco to $K^\sigma_{\geq \psi}$.
\end{proposition}
\begin{proof}
    To verify \ref{M1}, we just need to observe that for every $v\in K^\sigma_{\geq \psi}$ we can simply consider the sequence $v_n\equiv v$ for each $n$ since $K^\sigma_{\geq \psi}\subset K^{s_n}_{\geq \psi_n}$. 
    
    On the other hand, to verify that \ref{M2} holds, we start by considering a sequence $\{v_n\}$ with $v_n\in K^{s_n}_{\geq \psi_n}$ for each $n$, and a function $v\in \Lambda^{\sigma,p}_0(\Omega)$ such that $\Pi_{s_n}(v_n)\rightharpoonup \Pi_\sigma(v)$ in $\L^p(\Omega)$ as $s_n\to\sigma$. Now consider two compact sets $\omega$ and $\omega'$ such that $\omega\Subset\omega'\Subset\Omega$ and consider also a function $\phi\in C^\infty(\omega')$ such that $\phi\equiv 1$ in $\omega$. Moreover, from Proposition \ref{prop:EstimateProductTestFunction}, we know that $\|D^{s_n} (\phi u)\|_{L^p(\R^d;\R^d)}\leq \frac{C}{s}\|u\|_{\Lambda^{s_n,p}_0(\Omega)}$, which together with Theorem \ref{CompactnessFromUniformBoundOnDs}, implies the existence of two subsequences $\{\phi v_n\}$ and $\{\phi\psi_n\}$ such that 
    \begin{equation*}
    \begin{cases}
        \phi v_n\to \phi v\\
        \phi\psi_n\to \phi\psi
    \end{cases}
    \mbox{ in } L^p(\omega')\qquad \mbox{ and } \qquad
    \begin{cases}
        v_n\to v\\
        \psi_n\to \psi
    \end{cases}
    \mbox{ in } L^p(\omega).
    \end{equation*}
    Since this is valid for every compact $\omega\Subset\Omega$, and since $v_n\geq \psi_n$ for each $n$, we can easily deduce that $v\in K^\sigma_{\geq \psi}$.
\end{proof}

There are also some examples of sets that converge in the generalized sense of Mosco that do not require the order hypothesis on the obstacles, nor the order of the fractional parameters. As a first example, we consider $\psi_n=0$, where the convex sets are the cones of non negative functions.

\begin{theorem}[Positive cones]\label{prop:MoscoConvergencePositiveCone}
    Let $\Omega\subset\R^d$ be an open and bounded set, $p\in(1,\infty)$ and a sequence $\{s_n\}\subset (0,1)$ with $s_n\to\sigma\in[0,1]$. Then $K^{s_n}_{\geq 0}$ converges in the generalized sense of Mosco to $K^\sigma_{\geq 0}$.
\end{theorem}
\begin{proof}
    For \ref{M1} we use the same arguments as in the proof of the Theorem \ref{thm:no_obstacle}, but instead of using the density of $C^\infty_c(\Omega)$ in $\Lambda^{s_n,p}_0(\Omega)$, we use the density of the cone of non negative functions $C^\infty_c(\Omega)$ in the cone of non negative functions of $\Lambda^{s_n,p}_0(\Omega)$ of Theorem \ref{thm:densityPositiveCone}.

    For \ref{M2}, let us consider a sequence $\{v_n\}_{n\in \N}$ with $v_{n}\in K^{s_n}_{\geq 0}$ for each $n\in\N$, and a function $v\in \Lambda^{\sigma,p}_0(\Omega)$ such that $\Pi_{s_n}(v_n)\rightharpoonup \Pi_\sigma(v)$ in $\L^p(\Omega)$. As a consequence, we have
    \begin{equation*}
        0\leq \int_{\Omega}{v_n\varphi}\,dx\to \int_{\Omega}{v\varphi}\,dx \quad \forall \varphi\in C^\infty_c(\Omega) \mbox{ with } \varphi\geq 0,
    \end{equation*}
    which implies $v\geq 0$ a.e. in $\Omega$ and, consequently, $v\in K^\sigma_{\geq 0}$.
\end{proof}

This result is the basis for other examples of sequences of convex sets of obstacle type that converge in the generalized sense of Mosco. These examples can be obtained when we combine Theorem \eqref{prop:MoscoConvergencePositiveCone} with the following result, adapted from \cite[Lemma~1.6 on p.~529]{mosco1969convergence}, concerning the convergence of translations of the convex sets, for the $\L^p(\Omega)$ framework. 

\begin{proposition}[Translations of convex sets]\label{GeneralizedMoscoPlusStrong}
    Let $\Omega\subset\R^d$ be any open set, $p\in(1,+\infty)$ and a sequence $s_n$ in $(0,1)$ with $s_n\to\sigma\in[0,1]$. Consider a family of non-empty, closed convex sets $\{K_n\}$ such that $K_n\subset \Lambda^{s_n,p}_0(\Omega)$ converging in the sense of Mosco to $K\subset \Lambda^{\sigma,p}_0(\Omega)$, and also a sequence of functions $\{\psi_n\}$ such that $\psi_n\in \Lambda^{s_n,p}_0(\Omega)$ for all $n$ and $\Pi_{s_n}(\psi_n)\to \Pi_\sigma(\psi)$ in $\L^p(\Omega)$ for some function $\psi\in \Lambda^{\sigma,p}_0(\Omega)$. Then, the convex set $K_n+\psi_n$ converges in the generalized sense of Mosco to $K+\psi$.
\end{proposition}
\begin{proof}
    In order to check that \ref{M1} holds, let us consider a function $v\in K+\psi$, i.e., $v=\bar{v}+\psi$ for some $\bar{v}\in K$. Since $K_m$ converges in the generalized sense of Mosco to $K$, then there exists $\bar{v}_n\in K_n$ such that $\Pi_{s_n}(\bar{v}_n)\to \Pi_\sigma(\bar v)$ in $\L^p(\Omega)$. This allow us to construct the functions $v_n=\bar{v}_n+\psi_n\in K_n+\psi_n$ with the property that $\Pi_{s_n}(v_n)=\Pi_{s_n}(\bar{v}_n)+\Pi_{s_n}(\psi_n)\to \Pi_\sigma(\bar{v})+\Pi_\sigma(\psi)=\Pi_\sigma(v)$ in $\L^p(\Omega)$ as $s\to \sigma$.
    
    For the condition \ref{M2}, we start by considering a sequence $\{v_n\}$ with $v_n\in K_n+\psi_n$ and a function $v\in \Lambda^{\sigma,p}_0(\Omega)$ such that $\Pi_{s_n}(v_n)\rightharpoonup \Pi_\sigma(v)$ in $\L^p(\Omega)$ as $n\to\infty$. Since $v_n-\psi_n\in K_n$ for each $n$, $\Pi_{s_n}(v_n-\psi_n)\rightharpoonup \Pi_\sigma(v-\psi)$ in $\L^p(\Omega)$, and $K_n$ converges in the generalized sense of Mosco to $K$, then $v-\psi\in K$.
\end{proof}

As consequence of the last proposition, we can adapt {\cite[Lemma~1.7 on p.~531]{mosco1969convergence}} to prove a result about the generalized Mosco convergence of convex sets when the obstacles, and their fractional gradients, converge strongly.

\begin{corollary}[Without order between obstacles]\label{cor:ObstacleNoOrder}
    Let $\Omega\subset\R^d$ be a bounded open set, $p\in(1,+\infty)$ and $\{s_n\}\subset(0,1)$ a sequence with $s\to\sigma\in[0,1]$. Consider $\{\psi_n\}$ to be a family of measurable functions with $\psi_n\in \Lambda^{s_n,p}_0(\Omega)$ for each $n$.
    If $\Pi_{s_n}(\psi_n)\to \Pi_\sigma(\psi)$ in $\L^p(\Omega)$ for some $\psi\in \Lambda^{\sigma,p}_0(\Omega)$, then $K^{s_n}_{\geq\psi_n}$ converges in the generalized sense of Mosco to $K^\sigma_{\geq \psi}$.
\end{corollary}
\begin{proof}
    As $K^{s_n}_{\geq 0}\xrightarrow{s_n-M} K^{\sigma}_{\geq 0}$ by Theorem \ref{prop:MoscoConvergencePositiveCone}, and $\Pi_{s_n}(\psi_n)\to \Pi_\sigma(\psi)$ in $\L^p(\Omega)$, Proposition \ref{GeneralizedMoscoPlusStrong} immediately yields  $K^{s_n}_{\geq \psi_n}=K^{s_n}_{\geq 0}+\psi_n\xrightarrow{s_n-M} K^{\sigma}_{\geq 0}+\psi=K^{\sigma}_{\geq \psi}$.
\end{proof}

\begin{remark}\label{remark:generalizationForSConstant}
    The results in Proposition \ref{prop:GeneralizedWeak} and Corollary \ref{cor:ObstacleNoOrder} are extensions of the Mosco convergence of convex sets of obstacle type in the spaces $\Lambda^{s,p}_0(\Omega)$ with varying $s$, of classical results of \cite{mosco1969convergence} corresponding to constant $s_n=\sigma\in(0,1]$, in particular for $W^{1,p}_0(\Omega)$ \cite{boccardo2021some}.
\end{remark}

\begin{remark}
     It was proved in \cite{antil2021on} that, as a consequence of the compactness of the positive cone of $\Lambda^{-\sigma,2}(\Omega)$ in $\Lambda^{-\sigma,q}(\Omega)$ for $1<q<2$, if $\psi_n\rightharpoonup\psi$ in $\Lambda^{\sigma,p}_0(\Omega)$ for some $p>2$, then $H^\sigma_0(\Omega)\supset K^\sigma_{\geq \psi_n}\xrightarrow{M} K^\sigma_{\geq \psi}\subset H^\sigma_0(\Omega)$. It would be interesting to extend this property to $K^{s_n}_{\geq \psi_n}\xrightarrow{s_n-M} K^{\sigma}_{\geq \psi}$, even in the Hilbertian framework of $H^s_0(\Omega)$ with varying $s$.
\end{remark}
\begin{remark}\label{rem:obstacle_above}
    Theorem \ref{prop:MoscoConvergencePositiveCone} and Corollary \ref{cor:ObstacleNoOrder} are also valid for convex sets of the obstacle type of the form
    \begin{equation*}
        K^s_{\leq\phi}=\{v\in \Lambda^{s,p}_0(\Omega): v\leq \phi\}\neq\emptyset,
    \end{equation*}
    with the appropriate adaptations of the proofs.
\end{remark}

Another type of admissible sets that one can consider are the ones that correspond to constraints on the fractional gradient, namely convex sets of the form
\begin{equation}\label{eq:constraintConvex}
	\mathcal{K}^s_{g}=\{v\in \Lambda^{s,p}_0(\Omega): |D^s v|\leq g \mbox{ a.e. in }\R^d\}\neq \emptyset.
\end{equation}
Concerning the generalized Mosco convergence, the case of the fractional gradient constraints are, in general, more difficult to deal with than those with obstacle constraints on the function. However, we present an interesting example of a family of convex sets of that type that converge in the generalized sense of Mosco, using the approximation of the Riesz operator $I_\alpha$ to the identity as $\alpha\to 0$ (\cite{kurokawa1981On}).

\begin{theorem}[Constraint in the fractional gradient]
    Let $\Omega\subset\R^d$ any open set, $p\in(1,+\infty)$, $\sigma\in (0,1]$ and an increasing sequence $\{s_n\}\subset(0,1)$ such that $s_n\nearrow \sigma\in (0,1]$. Consider a function $g\in L^p(\R^d)\cap L^q(\R^d)$ with $1<q<p$ and $g\geq 0$ a.e., and a sequence of functions $g_n= I_{\sigma-s_n}g$. Then $\mathcal{K}^{s_n}_{g_n}$, as defined in \eqref{eq:constraintConvex}, converges in the generalized sense of Mosco to $\mathcal{K}^\sigma_{g}$.
\end{theorem}
\begin{proof}
    We start by noting that $g_n\to g$ in $L^p(\R^d)$ by \cite[Theorem~C]{kurokawa1981On}. Moreover, $\mathcal{K}^\sigma_{g}\subset \mathcal{K}^{s_n}_{g_n}$ for every $n$. In fact, if $u\in \mathcal{K}^\sigma_{g}$, then Proposition \ref{prop:fractionalGradientInTermsOfRieszPotentialAndFractionalGradient} allow us to obtain the estimate
    \begin{equation*}
        |D^{s_n} u|=|I_{\sigma-s_n} D^\sigma u|\leq I_{\sigma-s_n}|D^\sigma u|\leq I_{\sigma-s_n}g=g_n.
    \end{equation*}
    Consequently, we can set $u_n\equiv u\in \mathcal{K}^\sigma_{g}\subset \mathcal{K}^{s_n}_{g_n}$ for all $n$, and \ref{M1} holds.
    
    For the condition \ref{M2}, consider a sequence of functions $u_n\in\mathcal{K}{s_n}_{g_n}$ such that $\Pi_{s_n}(u_n)\rightharpoonup \Pi_\sigma(u)$ in $\L^p(\Omega)$. Then, $D^{s_n} u_n\rightharpoonup D^\sigma u$ in $L^p(\omega;\R^d)$ for all measurable sets $\omega\subset \R^d$, because $\Pi_{s_n}(u_n)\rightharpoonup \Pi_\sigma(u)$ in $\L^p(\Omega)$, and we have
    \begin{equation*}
        \int_\omega{|D^\sigma u|^p}\,dx\leq \liminf_{s_n\nearrow\sigma}{\int_\omega{|D^{s_n}u_n|^p}}\,dx\leq \liminf_{s_n\nearrow \sigma}{\int_\omega{|g_n|^p}}\,dx=\int_\omega{|g|^p}\,dx.
    \end{equation*}
    Since this inequality is valid for all measurable sets $\omega\subset \R^d$, we conclude that $|D^\sigma u|\leq g$ a.e. in $\R^d$.
\end{proof}

\subsection{Quasi-variational Inequalities}\label{subsec:quasivariational}

As a natural application of the Mosco convergence in variational inequalities, quasi-variational inequalities correspond to unilateral problems where the convex sets depend on the solutions themselves. This type of problems with $s=1$ arise in different models of control theory, solid mechanics or superconductivity, for instance. In this section, we provide examples of quasi-variational inequalities involving the $s$-fractional gradient. In particular, we study the existence of solutions, and their continuous dependence with respect to the fractional parameter $s$, to a quasi-variational inequality of obstacle type, where the obstacle is itself a solution to a coupled problem. We also study the existence of solutions to a quasi-variational inequality of $s$-gradient constraint type.

\subsubsection{Quasi-variational inequality of obstacle type}

Consider the problem of finding a solution $u\in K^s_r(u)$, such that
\begin{equation}\label{eq:quasivariationalObstacleAuxiliarProblem}
    \int_{\R^d}{\boldsymbol{a}(D^s u)D^s(w-u)}\,dx+\int_\Omega{b(u)(w-u)}\,dx\geq \langle F, w-u\rangle_s,\quad \forall w\in K^s_t(u),
\end{equation}
where the convex set is of obstacle type 
\begin{equation}\label{eq:convex_set_qvi}
    K^s_t(u)=\{w\in\Lambda^{s,p}_0(\Omega):\, w\geq \Psi(u)\},
\end{equation}
with the obstacle $\psi=\Psi(u)$ being itself a solution to the problem, depending on $u$, given by
\begin{equation}\label{eq:auxiliary_problem_for_qvi}
    \begin{cases}
        -\Delta^{t}_p \psi+|\psi|^{p-2}\psi=T(u, D^s u) & \mbox{in }\Omega,\\
        \psi=0 &\mbox{on } \R^d\setminus \Omega,
    \end{cases}
\end{equation}
where $T$ is a suitable continuous operator.
The implicit obstacle problem \eqref{eq:quasivariationalObstacleAuxiliarProblem}, \eqref{eq:convex_set_qvi} and \eqref{eq:auxiliary_problem_for_qvi} may be solved with the general abstract existence theory for quasi-variational inequalities of \cite{joly1979propos} or \cite[Theorem 2.2]{kano2008existence}. However, for this type of application it is simpler to use directly a combination of the Mosco continuous dependence results with a compactness argument. To illustrate also the continuous dependence on the fractional parameter in the case of monotone operators, including of $p$-Laplacian type, we present the following theorem, which proof we develop for completeness. 

\begin{theorem}\label{thm:quasiVariationalPseudoMonotone}
    Let $\Omega\subset\R^d$ be any open set, $s\in[0,1]$ and $p\in(1,\infty)$. Let us consider two Carathéodory functions $\boldsymbol{a}=\boldsymbol{a}(x,\xi):\R^d\times\R^d\to\R^d$ and $b=b(x,r):\R^d\times\R\to\R^d$, both monotone with respect to their last variable, and at least on of them being strictly monotone, satisfying the growth conditions \eqref{eq:upper_bound_for_A_when_omega_arbitrary} and \eqref{eq:upper_bound_for_B_when_omega_arbitrary} and the strong coercivity \eqref{eq:strongCorcivity}. Consider also a family of closed convex sets $\{K(v):\, v\in\Lambda^{s,p}_0(\Omega)\}$ such that
    \begin{itemize}
        \item[i)] there exists a bounded set $E\subset\Lambda^{s,p}_0(\Omega)$ with the property that $K(v)\cap E\neq\empty\emptyset$ for all $v\in\Lambda^{s,p}_0(\Omega)$;
        \item[ii)] whenever $v_n\rightharpoonup v$ in $\Lambda^{s,p}_0(\Omega)$, we have $K(v_n)\to K(v)$ in the (classical) sense of Mosco.
    \end{itemize}
    Then, there exists a function $u\in K(u)$, such that
    \begin{equation}\label{QuasiVariationalInequality}
        \int_{\R^d}{\boldsymbol{a}(D^s u)\cdot D^s(w-u)}\,dx+\int_\Omega{b(u)(w-u)}\,dx\geq 0, \quad \forall w\in K(u).
    \end{equation}
\end{theorem}
\begin{proof}
    For each $v\in \Lambda^{s,p}_0(\Omega)$, we can apply Theorem \ref{ExistenceFractionalVariationalInequality} and Proposition \ref{UniquenessSolutionPointwisePrespective}, to get the existence and uniqueness of the solution $u_v=S(v)\in K(v)$ to the problem
    \begin{equation}\label{inequalityForSolutionMap}
        \int_{\R^d}{(\boldsymbol{a}(D^s u_v))\cdot D^s(w-u_v)}\,dx+\int_\Omega{b(u_v)(w-u_v)}\,dx\geq 0, \quad \forall w\in K(v).
    \end{equation}
    We show that $S$ has a fixed point. We start by obtaining an a priori estimate for the functions $u_v$. Choose, for each $v\in\Lambda^{s,p}_0(\Omega)$, a function $g_v\in E\cap K(v)$. With these functions, we can apply the coercivity condition to get
    \begin{align*}
        \int_{\R^d}{\boldsymbol{a}(D^s u_v)\cdot D^s g_v}\,dx+\int_\Omega{b(u_v)g_v}\,dx&\geq \int_{\R^d}{\boldsymbol{a}(D^s u_v)\cdot D^s u_v}\,dx+\int_\Omega{b(u_v)u_v}\,dx\\
        &\geq \alpha \|D^s u_v\|^p_{L^p(\R^d;\R^d)}+\beta\|u_v\|^p_{L^p(\Omega)}-\|k\|_{L^1(\R^d)}.
    \end{align*}
    On the other hand, from the growth conditions, we have
    \begin{align*}
        &\int_{\R^d}{\boldsymbol{a}(D^s u_v)\cdot D^s g_v}\,dx+\int_\Omega{b(u_v)g_v}\,dx\\
        &\qquad \leq \|\gamma_1\|_{{L^p}'(\R^d)}\|D^s g_v\|_{L^p(\R^d;\R^d)}+C_1\|D^s u_v\|^{p-1}_{L^p(\R^d;\R^d)}\|D^s g_v\|_{L^p(\R^d;\R^d)}\\
        &\qquad\qquad+\|\gamma_2\|_{{L^p}'(\Omega)}\|g_v\|_{L^p(\Omega)}+C_2 \|u_v\|^{p-1}_{L^p(\R^d)}\|g_v\|_{L^p(\R^d)}.
    \end{align*}
    Combining this two estimates, and using the fact that $E$ is bounded in $\Lambda^{s,p}_0(\Omega)$, we deduce that there exists $R>0$, such that $\|u_v\|_{\Lambda^{s,p}_0(\Omega)}\leq R$, and so
    \begin{equation}\label{eq:inclusion_between_balls_for_qvi}
        S\left(\overline{B_R}\right)\subset S(\Lambda^{s,p}_0(\Omega))\subset \overline{B_R},
    \end{equation}
    where $\overline{B_R}=\{v\in\Lambda^{s,p}_0(\Omega):\, \|v\|_{\Lambda^{s,p}_0(\Omega)}\leq R\}$. Since $\overline{B_R}$ is weakly compact, then there exists $v,z\in \overline{B_R}$ such that
    \begin{equation*}
        v_n\rightharpoonup v \mbox{ in }\Lambda^{s,p}_0(\Omega) \qquad \mbox{ and } \qquad S(v_n)\rightharpoonup z \mbox{ in }  \Lambda^{s,p}_0(\Omega).
    \end{equation*}
    The weak convergence of $v_n$, implies $K(v_n)\to K(v)$ in the sense of Mosco by assumption \textit{ii)}. As a consequence of Theorem \ref{thm:general_stability_s_to_0}, which in this case is equivalent to the classical Mosco's theorem because $s_n\equiv \sigma$, see Remark \ref{rem:classicalMoscoTheorem}, we also get $S(v_n)\rightharpoonup S(v)$ in $\Lambda^{s,p}_0(\Omega)$.

    With these properties, we can apply Tychonov's fixed point theorem, \cite[Theorem 9.16]{baiocchi1984Variational} to deduce that $S$ has a fixed point in $\overline{B_R}$, completing the proof.
\end{proof}

As an application of this type of fixed point argument, we have the following existence result for \eqref{eq:quasivariationalObstacleAuxiliarProblem}, \eqref{eq:convex_set_qvi} and \eqref{eq:auxiliary_problem_for_qvi}, where we can relax the continuity assumption on the dependence $v\mapsto K^s_r(v)$.

\begin{theorem}\label{prop:existenceQuasivariationalObstacleAuxiliarProblem}
    Let $\Omega\subset\R^d$ be an open bounded set, $0<s< t \leq1$ and $1<p<\infty$. Let $\boldsymbol{a}$ and $b$ satisfy the same assumptions as in Theorem \ref{thm:quasiVariationalPseudoMonotone}. Let $T:\L^p(\Omega)\mapsto L^{(p^*_s)'}(\Omega)$ be a continuous operator such that $\|T(g)\|_{L^{(p^*_s)'}(\R^d)}\leq M<\infty$ for all $g\in \L^p(\Omega)$. Then, the quasi-variational inequality \eqref{eq:quasivariationalObstacleAuxiliarProblem}, \eqref{eq:convex_set_qvi} and \eqref{eq:auxiliary_problem_for_qvi} has a solution.
\end{theorem}
\begin{proof}
    We apply Schauder's fixed theorem to $v\mapsto S(v)=Q\circ \Psi(v)$, where $\Psi=\overline{\Psi}\circ T\circ \Pi_s(v)$, being $T\circ \Pi_s: \Lambda^{s,p}_0(\Omega)\to\Lambda^{-t,p'}(\Omega)$ a compact operator since $L^{(p^s_*)'}(\Omega)\subset \Lambda^{-s,p'}(\Omega)\Subset \Lambda^{-t,p'}(\Omega)$ by Corollary \ref{cor:fracRellichKondrachovFixedP}, and with $\overline{\Psi}:\Lambda^{-t,p'}(\Omega)\ni T_v\mapsto\psi\in\Lambda^{t,p}_0(\Omega)$ being the solution of \eqref{eq:auxiliary_problem_for_qvi} given by Corollary \ref{UniquenessFractionalPLaplacian}, and $Q:\Lambda^{t,p}_0(\Omega)\subset \Lambda^{s,p}_0(\Omega)\to \Lambda^{s,p}_0(\Omega)$ is the solution to the variational inequality \eqref{eq:quasivariationalObstacleAuxiliarProblem} with the convex $K^s_t(v)$ given by \eqref{eq:convex_set_qvi}.
    
    In order to obtain the inclusion \eqref{eq:inclusion_between_balls_for_qvi} for some $R>0$, we need the following a priori estimate. For every $v\in\Lambda^{s,p}_0(\Omega)$, let $\Psi(v)$ denote the solution to \eqref{eq:auxiliary_problem_for_qvi} with $T_v=T(v, D^s v)$, so, recalling $\Lambda^{t,p}_0(\Omega)\subset\Lambda^{s,p}_0(\Omega)$, we have
    \begin{multline}\label{eq:estimate_for_size_ball_qvi}
        C'\|\Psi(v)\|^p_{\Lambda^{s,p}_0(\Omega)}
        \leq \|\Psi(v)\|^p_{\Lambda^{t,p}_0(\Omega)}=\langle -\Delta^{t}_p \Psi(v)+|\Psi(v)|^{p-2}\Psi(v),\Psi(v)\rangle_\sigma=\int_\Omega{T_v\Psi(v)}\,dx\\
        \leq \|T_v\|_{L^{(p^*_s)'}(\Omega)} \|\Psi(v)\|_{L^{p^*_s}(\Omega)}\leq C'_s\|T_v\|_{L^{(p^*_s)'}(\Omega)} \|\Psi(v)\|_{\Lambda^{s,p}_0(\Omega)}.
    \end{multline}
    Since $\|T_v\|_{{L^{p_s^*}}'(\Omega)}\leq M$, then we have
    \begin{equation}\label{eq:apriori_estimate_solution_coupled_problem}
        \|\Psi(v)\|_{\Lambda^{s,p}_0(\Omega)}\leq (M/C_s)^{\frac{1}{p-1}}=R,
    \end{equation}
    and we have a fixed point $u=Su\in \overline{B_R}\cap K^s_t(u)$.
\end{proof}

\begin{remark}
    An example may be given by the Uryson operator $T:\L^p(\Omega)\mapsto L^{(p^*_s)'}(\Omega)$ given by
    \begin{equation}
        T(g)(x)=\int_{\mathscr{O}}{\tau(x, y, g(y), D^s g(y))}\,dy,\quad \mbox{ for a.e. } x\in\Omega,
    \end{equation}
    where $\mathscr{O}$ is a bounded subset of $\R^d$ and $\tau:\Omega\times \mathscr{O}\times\R\times\R^d\to\R$ is a Carathéodory function, continuous in $(v,\xi)$ and measurable for a.e. $(x,y)$, satisfying the inequality
    \begin{equation*}
        |\tau(x,y,v,\xi)|\leq \phi(x,y), \mbox{ for a.e. } (x,y)\in\Omega\times\mathscr{O} \mbox{ and for all } (v,\xi)\in \R\times\R^d
    \end{equation*}
    with $\phi\in L^{(p^*_s)'}(\Omega; L^\infty(\mathscr{O}))$.  Here we may choose $M=\left(\int_\Omega{\left(\int_{\mathscr{O}}{\phi(x,y)}\,dy\right)^{(p^*_s)'}}\,dx\right)^\frac{1}{(p^*_s)'}$ in Theorem \ref{prop:existenceQuasivariationalObstacleAuxiliarProblem}. For the continuity of $T:\L^p(\Omega)\mapsto L^{(p^*_s)'}(\Omega)$, we observe that for every $(g_n, E_n)\in\L^p(\Omega)$ converging to $(g,E)\in\L^p(\Omega)$ as $n\to\infty$ we have $\tau(x,y, g_n(y), E_n(y))\to \tau(x,y,g(y),E(y))$ in measure as $n\to\infty$, for a.e. $(x,y)\in\Omega\times \mathscr{O}$, because $\tau$ is a Carathéodory function. Moreover, for $\mathcal{C}\subset \mathscr{O}$ we have
    \begin{equation*}
        \lim_{|\mathcal{C}|\to 0}{\int_{\mathcal{C}}{|\tau(x, y, g_n(y), E_n(y))|}\,dy}\leq \lim_{|\mathcal{C}|\to 0}{\int_{\mathcal{C}}{\phi(x, y)}\,dy}=0 
    \end{equation*}
    uniformly with respect to $n$. From Vitalli's theorem, we obtain that $\tau(x, \cdot, g_n, E_n)\to \tau(x, \cdot, g, E)$ in $L^1(\mathscr{O})$ for a.e. $x\in\Omega$.
    The continuity of $T$ then follows by Lebesgue's dominated convergence theorem.
\end{remark}

\begin{remark}\label{rem:qvi_obstacle_with_T_depending_only_on_u}
    If in Theorem \ref{prop:existenceQuasivariationalObstacleAuxiliarProblem} we assume that $T$ does not depend on $D^s u$, i.e. $T:L^q(\Omega)\to L^{(p^*_s)'}(\Omega)$, with $1\leq q<p^*_s$, and keep the same assumptions on $T$, we can still prove the existence of a solution to \eqref{eq:quasivariationalObstacleAuxiliarProblem}, \eqref{eq:convex_set_qvi} and \eqref{eq:auxiliary_problem_for_qvi} when $0<s=t\leq 1$, as a Corollary of Theorem \ref{thm:quasiVariationalPseudoMonotone}. Indeed, by the fractional Rellich-Kondrachov's theorem, Theorem \ref{CompactnessResultALaBellido}, we can use the compactness of $\Lambda^{s,p}_0(\Omega)\Subset L^q(\Omega)$ for any sequence $v_n\rightharpoonup v$ in $\Lambda^{s,p}_0(\Omega)$,  and the continuity of $T: \L^q(\Omega)\to L^{(p^*_s)'}(\Omega)$, to get $T_{v_n}\to T_v$ in $L^{(p^*_s)'}(\Omega)$, and consequently, by Corollary \ref{UniquenessFractionalPLaplacian}, we have $\Psi(v_n)\to \Psi(v)$ in $\Lambda^{s,p}_0(\Omega)$. Hence, Corollary \ref{cor:ObstacleNoOrder}, with $s_n\equiv\sigma$, yields $K^s_t(v_n)\to K^s_t(v)$ in the sense of Mosco and the assumption \textit{ii)} of Theorem \ref{thm:quasiVariationalPseudoMonotone} is satisfied. The assumption \textit{i)} is satisfied with $E=\overline{B_R}$ by the a priori estimate \eqref{eq:estimate_for_size_ball_qvi}.
\end{remark}

\begin{remark}
    An example of a continuous operator $T:L^q(\Omega)\to L^{(p^*_s)'}(\Omega)$ satisfying the hypothesis stated in Remark \ref{rem:qvi_obstacle_with_T_depending_only_on_u} is given by
    \begin{equation}
        T(g)(x)=(g(x)\wedge k(x))\vee -k(x),
    \end{equation} 
    with $k\geq 0$ a.e. and $k\in L^{(p^*_s)'}(\Omega)$. In fact, we may choose $M=\|k\|_{{L^{(p^*_s)}}'(\Omega)}$, and since
    \begin{equation*}
        |T(g)(x)|\leq k(x) \quad \mbox{ a.e. } x\in\Omega, \quad \mbox { for all } g\in L^q(\Omega),
    \end{equation*}
    using Vitalli's theorem, the continuity follows
    \begin{equation*}
        T(g_n)\to T(g) \mbox{ in }L^{(p^*_s)'}(\Omega)\quad \mbox{ for all } g_n\to g \mbox{ in } L^q(\Omega).
    \end{equation*}
\end{remark}

In this framework it is also possible to show the continuous dependence, with respect to $s$, of the solutions to the quasi-variational inequality of obstacle type \eqref{eq:quasivariationalObstacleAuxiliarProblem}, \eqref{eq:convex_set_qvi} and \eqref{eq:auxiliary_problem_for_qvi}. This is a consequence of the general stability results obtained in Section \ref{sec:stability} and the convergence, in the generalized sense of Mosco, of convex sets of obstacle type as obtained in Subsection \ref{subsec:examples_mosco}, combined with compactness methods. For simplicity, we assume that $T$ does not depend on $D^s$ and $q=p$, so that we can use the fractional Poincaré's inequality, for which the dependence of the constant with respect to the fractional parameter $s$ can be controlled.

\begin{theorem}\label{thm:convergenceSolutionsQuasivariationalInequality}
    Let $\Omega$ be a bounded open set in $\R^d$, $1<p<\infty$ and consider a sequence $\{s_n\}\subset (s_*,\sigma]$ with $0<s_*<\sigma\leq 1$ such that $s_n\to \sigma$. Let $\boldsymbol{a}$ and $b$ have the same hypothesis as in Proposition \ref{prop:existenceQuasivariationalObstacleAuxiliarProblem} and $F_n={f_0}_n-D^{s_n}\cdot\boldsymbol{f}_n\in \Lambda^{-s_n,p'}(\Omega)$ satisfying \eqref{eq:convergence_data}. Assume that $T:L^p(\Omega)\mapsto L^{p'}(\Omega)$ be a continuous operator such that $\|T(g)\|_{L^{p'}(\R^d)}\leq M<\infty$ for all $g\in L^p(\Omega)$. If $u_n\in K^{s_n}_{\sigma}(u_n)$ is a solution of \eqref{eq:quasivariationalObstacleAuxiliarProblem}, \eqref{eq:convex_set_qvi} and \eqref{eq:auxiliary_problem_for_qvi} for each $n\in\N$, then we can extract a subsequence such that 
    \begin{equation*}
        u_n\to u \mbox{ in } \Lambda^{t,p}(\Omega), \mbox{ for } 0\leq t\leq s^* \quad \mbox{ and } \quad D^{s_n}u_n\rightharpoonup D^\sigma u \mbox{ in } L^p(\R^d;\R^d). 
    \end{equation*}
    where $u\in K^\sigma_\sigma(u)$ solves
    \begin{equation}\label{eq:limit_qvi_problem}
        \int_{\R^d}{\boldsymbol{a}(D^\sigma u)\cdot D^\sigma(w-u)\,dx+\int_\Omega b(u)(w-u)}\,dx\geq \langle F, w-u\rangle_\sigma,\quad \forall w\in K^\sigma_\sigma(u).
    \end{equation}
\end{theorem}
\begin{proof}
    From Theorem \ref{thm:general_stability_s_to_0}, we just need to prove that $K^{s_n}_{\sigma}(u_n)\to K^\sigma_\sigma(u)$ in the generalized sense of Mosco.
	
    Using the same argument as in \eqref{eq:estimate_for_size_ball_qvi} with $s=s_n$, $t=\sigma$, and using Corollary \ref{cor:order_between_Ds} and the fractional Poincaré's inequlity \eqref{eq:poincare_inequality} to make explicit the dependence of the constant in terms of $s_*<s_n\leq \sigma$, we obtain
    \begin{equation}\label{eq:uniform_bound_for_qvi}
        \begin{split}
            \|D^{s_n}\Psi(u_n)\|^{p-1}_{L^p(\R^d;\R^d)}
            \leq \left(\frac{C}{s_*^{1+1/p}}\right)^{p-1}\|D^\sigma \Psi(u_n)\|^{p-1}_{L^p(\R^d;\R^d)}\leq \left(\frac{C}{s_*^{2+1/p}}\right)^{p-1}\|\psi(u_n)\|^{p-1}_{\Lambda^{\sigma,p}(\R^d)}\\
            \leq \left(\frac{C}{s_*^{2+1/p}}\right)^{p-1}\|\psi(u_n)\|^{-1}_{\Lambda^{\sigma,p}(\R^d)}\langle-\Delta^{r}_p \Psi(u_n)+|\Psi(u_n)|^{p-2}\Psi(u_n),\Psi(u_n)\rangle_\sigma
            \leq M\left(\frac{C}{s_*^{2+1/p}}\right)^{p-1}.
        \end{split}
    \end{equation}
    This estimate is useful to obtain an estimate for $u_n$ in $\Lambda^{s_n,p}_0(\Omega)$, independent of $s_n$. In fact, testing \eqref{eq:quasivariationalObstacleAuxiliarProblem} with $\Psi(u_n)\in K^{s_n}_{\sigma}(u_n)$, and using an adaptation of the arguments employed in \eqref{InequalityForAPrioriBound}, \eqref{UpperBoundAprioriA} and \eqref{UpperBoundAprioriL} with the coercivity and growth assumptions on $\boldsymbol{a}$ and $b$ and the bound \eqref{eq:uniform_bound_for_qvi}, we get the uniform estimate,
    \begin{equation*}
        \|u_n\|_{\Lambda^{s_n,p}_0(\Omega)}\leq C
    \end{equation*}
    where $C>0$ is a positive constant depending on $\Omega,\alpha,\beta, M$ and $s_*$ but not on $n$.

    From Theorem \ref{CompactnessFromUniformBoundOnDs} we get $u_n\to u$ in $L^p(\Omega)$, which by continuity of $T$ implies $T(u_n)\to T(u)$ in ${L^p}'(\Omega)$, and consequently, from Corollary \ref{UniquenessFractionalPLaplacian} we get $\Psi(u_n)\to \Psi(u)$ in $\Lambda^{\sigma,p}_0(\Omega)$. Moreover, we also have $D^{s_n} \Psi(u_n)\to D^\sigma \Psi(u)$ in $L^p(\R^d;\R^d)$ because
    \begin{align*}
        &\lim_{n\to\infty}{\|D^{s_n} \Psi(u_n)-D^\sigma \Psi(u)\|_{L^p(\R^d;\R^d)}}\\
        &\leq \lim_{n\to\infty}{\left(\|D^{s_n}\Psi(u_n)-D^{s_n}\Psi(u)\|_{L^p(\R^d;\R^d)}+\|D^{s_n} \Psi(u)- D^\sigma \Psi(u)\|_{L^p(\R^d;\R^d)}\right)}\\
        &\leq \lim_{n\to\infty}{\left(C_{s_*}\|\Psi(u_n)-\Psi(u)\|_{\Lambda^{\sigma,p}_0(\Omega)}+\|D^{s_n} \Psi(u)- D^\sigma \Psi(u)\|_{L^p(\R^d;\R^d)}\right)}=0.
    \end{align*}
    To conclude, we just need to apply Corollary \ref{cor:ObstacleNoOrder}.
\end{proof}

\begin{remark}\label{rem:convergence_solutions_qvi_sn_equal_rn}
    When $T$ depends only on $u$ and $q=p$, a similar argument can be used to prove that the solutions $u_n\in K^{s_n}_{s_n}(u_n)$ to \eqref{eq:quasivariationalObstacleAuxiliarProblem}, \eqref{eq:convex_set_qvi} and \eqref{eq:auxiliary_problem_for_qvi}, for each $n\in\N$, converge to a solution $u\in K^{\sigma}_{\sigma}(u)$ of \eqref{eq:limit_qvi_problem}, as $s_n\to\sigma$. The only difference between the proofs is that we do not need to use Corollary \ref{cor:order_between_Ds} to get an estimate of the type of \eqref{eq:uniform_bound_for_qvi}, and instead of using Corollary \ref{UniquenessFractionalPLaplacian} for the convergence of $\Psi(u_n)$ to $\Psi(u)$ in $\Lambda^{\sigma,p}_0(\Omega)$, we use Remark \ref{rem:strong_convergence_solutions_pdes_p_laplacian} for the convergence of $\Pi_{s_n}(u_n)$ to $\Pi_\sigma(u)$ in $\L^p(\Omega)$.
\end{remark}

\begin{remark}
    In Theorem \ref{thm:convergenceSolutionsQuasivariationalInequality} and in Remark \ref{rem:convergence_solutions_qvi_sn_equal_rn} we can recover the solutions $u\in K^1_1(u)$ to the classical quasi-variational problem associated to \eqref{eq:quasivariationalObstacleAuxiliarProblem}, \eqref{eq:convex_set_qvi} and \eqref{eq:auxiliary_problem_for_qvi} with $s=r=1$, which is a localization result of the fractional quasi-variational inequality as $s\to1$.
\end{remark}

\subsubsection{Quasi-variational inequality of $s$-gradient constraint type}

Consider the problem of finding $u\in \mathcal{K}^s_{G(u)}$, such that
\begin{equation}\label{eq:qvi_gradient_constraint}
    \int_{\R^d}{\boldsymbol{a}(D^s u)\cdot D^s(v-u)}\,dx+\int_\Omega{b(u)(v-u)}\,dx\geq \langle F_s,v-u\rangle_s,\quad \forall v\in \mathcal{K}^s_{G(u)}
\end{equation}
where $G:L^{p^*_s}(\Omega)\to L^\infty_\nu(\R^d)$ is a continuous and bounded operator, with
\begin{equation*}
    L^\infty_\nu(\R^d)=\{g\in L^\infty(\R^d):\, g\geq \nu\mbox{ a.e. in }\R^d\}
\end{equation*}
for some $\nu>0$, and
\begin{equation}\label{eq:convex_set_qvi_gradient_constraint}
    \mathcal{K}^s_{G(u)}=\{w\in \Lambda^{s,p}_0(\Omega):\, |D^s w|\leq G(u) \mbox{ a.e. in }\R^d\}.
\end{equation}
This extends a result of \cite{rodrigues2019On} to the nonlinear fractional setting, inspired in examples of the classic case $s=1$ surveyed in \cite{rodrigues2019Variational}.

\begin{theorem}
    Let $\Omega\subset\R^d$ be a bounded open set, $s\in(0,1]$, $p\in(1,\infty)$ and $\nu>0$. Let $\mathscr{A}_s(u)=-D^s\cdot\boldsymbol{a}(D^su)$ be a $p$-Laplacian type operator, let $b:\R^d\times\R\to\R^d$ satisfy the growth condition \eqref{eq:upper_bound_for_B_when_omega_arbitrary} and the monotonicity condition \eqref{eq:monotonicityb}, with the constant $\beta_p\geq 0$, and let $F\in \Lambda^{-s,p'}(\Omega)$. If $G:L^{p^*_s}_0(\Omega)\to L^\infty_\nu(\R^d)$ is a continuous and bounded operator, then there exists a solution $u\in \mathcal{K}^s_{G(u)}$ to \eqref{eq:qvi_gradient_constraint}.
\end{theorem}
\begin{proof}
    Let $g\in L^\infty_\nu(\R^d)$. We know from Corollary \ref{UniquenessFractionalPLaplacian} that there exists a unique solution $u_g=S(F, g)\in \mathcal{K}^s_g$ such that
    \begin{equation*}
        \int_{\R^d}{\boldsymbol{a}(D^s u_g)\cdot D^s(v-u_g)}\,dx+\int_\Omega{b(u)(v-u_g)}\,dx\geq \langle F, v-u_g\rangle_s, \quad\forall v\in \mathcal{K}^s_g.
    \end{equation*}
    Moreover, since $0\in \mathcal{K}^s_g$, we have
    \begin{align*}
        \alpha_p\|D^s u_g\|_{L^p(\R^d;\R^d)}^p+\beta_p\|u_g\|_{L^p(\Omega)}
        &\leq \int_{\R^d}{\boldsymbol{a}(D^s u_g)\cdot D^s u_g}\,dx+\int_\Omega{b(u_g)u_g}\,dx\\
        &\leq\langle F, u_g\rangle_s\leq \|\boldsymbol{f}\|_{{L^p}'(\R^d;\R^d)}\|D^s u_g\|_{L^p(\R^d;\R^d)},
    \end{align*}
    which implies the existence of a positive constant $C_{F}$, independent of $g$, such that
    \begin{equation*}
        \|u_g\|_{L^{p_s^*}(\Omega)}\leq C_s\|D^s u_g\|_{L^p(\R^d;\R^d)}\leq C_{F}.
    \end{equation*}
    
    Let us consider the operator $T=S(F,\cdot)\circ G$. Notice that for every $w\in \overline{B_{C_F}}=\{v\in L^{p^*_s}_0(\Omega):\, \|v\|_{L^{p^*_s}(\Omega)}\leq C_F\}$, we have $T(w)=u_{G(w)}\in \overline{B_{C_F}}\cap \mathcal{K}^s_{G(w)}$, so by proving that $T$ is continuous and compact in $L^{p^*_s}_0(\Omega)$, we can apply Schauder's fixed point theorem on $\overline{B_{C_F}}$ and deduce that there exists a function $u\in \overline{B_{C_F}}$ such that $T(u)=u$.
    
    To prove that $T:L^{p^*_s}_0(\Omega)\to L^{p^*_s}_0(\Omega)$ is continuous, we only need to prove that $S(F_s,\cdot):L^\infty(\Omega)\to L^{p^*_s}_0(\Omega)$ is continuous since $G:L^{p^*_s}(\Omega)\to L^\infty_\nu(\R^d)$ is by hypothesis a continuous operator. Consider two functions $g_1,g_2\in L^{p^*}(\Omega)$ and define $u_i=S(F,g_i)$ as well as $\Tilde{u}_i=\frac{\nu}{\nu+\eta}u_i$, where $\eta:=\|g_1-g_2\|_{L^\infty(\R^d)}$, for $i=1,2$. Notice that $\Tilde{u}_1\in K^s(g_2)$ and $\Tilde{u}_2\in K^s(g_1)$ because
    \begin{equation*}
        |D^s\Tilde{u}_1|\leq \frac{\nu}{\nu+\eta}|g_1|\leq \frac{\nu}{\nu+\eta}(|g_1-g_2|+|g_2|)\leq \frac{\eta}{\nu+\eta}\nu+\frac{\nu}{\nu+\eta}|g_2|\leq |g_2|,
    \end{equation*}
    and similarly $|D^s\Tilde{u}_2|\leq |g_1|$. Moreover, these functions also satisfy
    \begin{equation*}
        |u_i-\Tilde{u}_i|=\frac{\eta}{\nu+\eta}|u_i|\leq\frac{\eta}{\nu}|u_i|\quad \mbox{ and }\quad |D^s(u_i-\Tilde{u}_i)|=\frac{\eta}{\nu+\eta}|D^s u_i|\leq\frac{\eta}{\nu}|D^s u_i|.
    \end{equation*}
    These estimates, in conjugation with Hölder's inequality, the fractional Poincaré inequality and the fact that $u_{g_i}=S(F_s, g_i)$, allows us to obtain
    \begin{align*}
        &\int_{\R^d}{(\boldsymbol{a}(D^s u_1)-\boldsymbol{a}(D^s u_2))\cdot D^s(u_1-u_2)}\,dx+\int_\Omega{(b(u_1)-b(u_2))(u_1-u_2)}\,dx\\
        &\leq \langle F, (u_1-\Tilde{u}_1)+(u_2-\Tilde{u}_2))\rangle_s+\int_{\R^d}{\boldsymbol{a}(D^s u_1)\cdot D^s(\Tilde{u}_2-u_2)}\,dx+\int_{\R^d}{\boldsymbol{a}(D^s u_2)\cdot D^s(\Tilde{u}_1-u_1)}\,dx\\
        &\qquad\qquad+\int_{\Omega}{b(u_1)(\Tilde{u}_2-u_2)}\,dx+\int_{\Omega}{b(u_2)\cdot (\Tilde{u}_1-u_1)}\,dx\\
        &\leq \frac{\eta}{\nu}\|f\|_{{L^p}'(\Omega)}(\|u_1\|_{L^p(\Omega)}+\|u_2\|_{L^p(\Omega)})+2C\frac{\eta}{\nu}\left(\|D^s u_1\|^{p-1}_{L^p(\R^d;\R^d)}\|D^s u_2\|_{L^p(\R^d;\R^d)}+\|u_1\|^{p-1}_{L^p(\Omega)}\|u_2\|_{L^p(\Omega)}\right)\\
        &\leq C\|g_1-g_2\|_{L^\infty(\R^d)}.
    \end{align*}
    Moreover, if we combine this estimate with the coercivity conditions on $\boldsymbol{a}$ and $b$, \eqref{eq:monotonicityPLaplacian} and \eqref{eq:monotonicityb} respectively, we have for $p>2$ 
    \begin{equation*}
        C_{\alpha_p,\beta_p}\|u_1-u_2\|^p_{\Lambda^{s,p}(\Omega)}\leq C\|g_1-g_2\|_{L^\infty(\R^d)}
    \end{equation*}
    and for $1<p<2$
    \begin{multline*}
        C_{\alpha_p,\beta_p}\|u_1-u_2\|^2_{\Lambda^{s,p}(\Omega)}\leq C\|g_1-g_2\|_{L^\infty(\R^d)}\left(\|D^s u_1\|_{L^p(\R^d;\R^d)}+\|D^s u_2\|_{L^p(\R^d;\R^d)}\right)^{2-p}\\
        \leq (2C_{F})^{2-p}C\|g_1-g_2\|_{L^\infty(\R^d)},
    \end{multline*}
    concluding the continuity of $S(F_s,\cdot):L^\infty(\Omega)\to L^{p^*_s}(\Omega)$.
    
    To prove the compactness of $T$, we use the boundedness of $G:L^{p^*_s}(\Omega)\to L^\infty_\nu(\R^d)$ and the estimate 
    \begin{multline*}
        \|D^s u_{G(w)}\|^q_{L^q(\R^d;\R^d)}=\int_{\R^d}{|D^s u_{G(w)}|^q}\,dx=\int_{\R^d}{|D^s u_{G(w)}|^{q-p}|D^s u_{G(w)}|^p}\,dx\\
        \leq \|D^s u_{G(w)}\|^{q-p}_{L^\infty(\R^d;\R^d)}\|D^s u_{G(w)}\|_{L^p(\R^d;\R^d)}^p\leq \|G(w)\|^{q-p}_{L^\infty(\R^d;\R^d)}\|D^s u_{G(w)}\|_{L^p(\R^d;\R^d)}^p
    \end{multline*}
    with $q>\frac{N}{s}$ $q>p$ and $\gamma=s-\frac{N}{q}$, to derive $T(\overline{B_{C_F}})\subset C^{0,\gamma}(\overline{\Omega})\Subset C^0(\overline{\Omega})\subset L^{p^*_s}(\Omega)$, by using Ascoli-Arzelà's theorem.
\end{proof}

\begin{remark}
    Similarly to \cite[Example 4.1]{rodrigues2019Variational}, we can consider $G:L_0^{p^*_s}(\Omega)\to\R^d$ defined as
    \begin{equation*}
        G(u)(x)=F(x,w_u(x)),
    \end{equation*}
    where $F:\R^d\times\R\to\R$ is a bounded function in $x\in\R^d$ and continuous in $w\in \R$ uniformly with respect to $x\in\R^d$, satisfying for some $\nu>0$,
    \begin{equation*}
        0<\nu\leq F(x,w)\leq \varphi(|w|)\quad \mbox{ a.e. } x\in\R^d
    \end{equation*}
    for some monotone increasing function $\varphi$. Moreover, for $u\in L_0^{p^*_s}(\Omega)$, we set
    \begin{equation*}
        w_u(x)=\int_\Omega{\theta(x,y)u(y)}\,dy
    \end{equation*}
    with $\theta\in L^\infty(\R^d_x; L^{(p^*_s)'}(\Omega_y))$.
\end{remark}

The case of the convergence of the solutions of the quasi-variational problems with gradient constraints is much more delicate, since the criteria of the generalized Mosco convergence of the convex sets $K(u)$ is not satisfied in general for the quasi-variational solutions $u$. However, using a result of \cite{azevedo2022transport} on the localization of the solutions of variational inequalities with gradient constraint as $s\to 1$ in the Hilbertian framework $\Lambda^{s,2}_0(\Omega)=H^s_0(\Omega)$, we can obtain the following result.

\begin{theorem}
    Let $\Omega\subset\R^d$ be a bounded open set, $0<s_*<s<1$ and $F_s={f_0}_s-D^s\cdot\boldsymbol{f}_s\in H^{-s}(\Omega)$ satisfying \eqref{eq:convergence_data}  as $s\nearrow 1$ with $p'=2$. Consider the linear operator $\mathscr{A}_{s}u=-D^s\cdot(A(x)D^{s} u)$ with $A$ being a bounded, measurable and strictly elliptic matrix-valued function, satisfying
    \begin{equation*}
        \alpha|\xi|^2\leq A(x)\xi\cdot\xi\quad \mbox{ and } \quad A(x)\xi\cdot\eta\leq C|\xi||\eta|,
    \end{equation*}
    and a continuous operator $G:L^2(\Omega)\to L^\infty_\nu(\R^d)$, such that $0<g_*\leq G(v)\leq g^*$ for all $v\in L^2(\Omega)$. From the set of $u_s\in \mathcal{K}^{s}_{G(u_s)}\subset H^s_0(\Omega)$ that are solutions of \eqref{eq:qvi_gradient_constraint}-\eqref{eq:convex_set_qvi_gradient_constraint} for each $s$, we can extract a generalized subsequence such that, as $s\nearrow 1$,
    \begin{equation*}
        u_s\to u \mbox{ in } \Lambda^{s_*,p}_0(\Omega)\cap C^{0,\lambda}(\overline{\Omega}) \quad \mbox{ and } \quad D^{s}u_s\overset{\ast}{\rightharpoonup} D u \mbox{ in } L^\infty(\R^d;\R^d),
    \end{equation*}
    for all $0\leq\lambda<s_*$ and all $1<p<\infty$, where $u\in K^1_{G(u)}\subset H^1_0(\Omega)\cap C^{0,1}(\overline{\Omega}) $ solves
    \begin{equation*}
        \int_{\Omega}{A D u\cdot D(w-u)}\,dx\geq \langle F, w-u\rangle,\quad \forall w\in K^1_{G(u)}.
    \end{equation*}
\end{theorem}
\begin{proof}
    This is a direct application of \cite[Theorem 5.1]{azevedo2022transport}, because the assumption and the uniform estimate of $D^{s}u_s$ in $L^{\infty}(\Omega)$ imply the uniform convergence $u_s\to u$ and, consequently also, $G(u_s)\to G(u)$ in $L^\infty(\Omega)$ as $s\nearrow 1$.
\end{proof}

%%%%%%%%%% Acknowledgments. %%%%%%%%%%

\textbf{Acknowledgments.} The first and second authors' research was done under the framework of CMAFcIO, FCT project: UIDB/04561/2020 and UIDP/04561/2020. The first author was also supported by the PhD FCT-grant UI/BD/152276/2021.

\newcommand{\etalchar}[1]{$^{#1}$}


\begin{thebibliography}{dTGCV21}

\bibitem[Ada75]{adams1975Sobolev}
R.~A. Adams.
\newblock {\em Sobolev spaces}.
\newblock Pure and Applied Mathematics, Vol. 65. Academic Press [Harcourt Brace
  Jovanovich, Publishers], New York-London, 1975.

\bibitem[AH96]{adams1996Function}
D.~R. Adams and L.~I. Hedberg.
\newblock {\em Function spaces and potential theory}, volume 314 of {\em
  Grundlehren der mathematischen Wissenschaften [Fundamental Principles of
  Mathematical Sciences]}.
\newblock Springer-Verlag, Berlin, 1996.

\bibitem[AR18]{antil2018Fractional}
H.~Antil and C.~N. Rautenberg.
\newblock Fractional elliptic quasi-variational inequalities: theory and
  numerics.
\newblock {\em Interfaces Free Bound.}, 20(1):1--24, 2018.

\bibitem[ARS21]{antil2021on}
H.~Antil, C.~N. Rautenberg, and A.~Schikorra.
\newblock On a fractional version of a {M}urat compactness result and
  applications.
\newblock {\em SIAM J. Math. Anal.}, 53(3):3158--3187, 2021.

\bibitem[ARS22]{azevedo2022transport}
A.~Azevedo, J.~F. Rodrigues, and L.~Santos.
\newblock Nonlocal {L}agrange multipliers and transport densities.
\newblock {\em arXiv preprint arXiv:2208.14274}, 2022.
\newblock To appear in Bull. Math. Sci.

\bibitem[ARS23]{azevedo2023constrain}
A.~Azevedo, J.~F. Rodrigues, and L.~Santos.
\newblock On a class of nonlocal problems with fractional gradient constraint.
\newblock In {\em European {C}ongress of {M}athematics}, pages 885--906. EMS
  Press, Berlin, 2023.

\bibitem[AS59]{aronszajn1959theory}
N.~Aronszajn and K.~T. Smith.
\newblock {Theory of Bessel potentials. Number 26. University of Kansas, Dept.
  of Mathematics}, 1959.

\bibitem[AS61]{aronszajn1961theory}
N.~Aronszajn and K.~T. Smith.
\newblock {Theory of Bessel potentials. I}.
\newblock In {\em Annales de l'institut Fourier}, volume~11, pages 385--475,
  1961.

\bibitem[BC84]{baiocchi1984Variational}
C.~Baiocchi and A.~Capelo.
\newblock {\em Variational and quasivariational inequalities}.
\newblock A Wiley-Interscience Publication. John Wiley \& Sons, Inc., New York,
  1984.
\newblock Applications to free boundary problems, Translated from the Italian
  by Lakshmi Jayakar.

\bibitem[BCCS22]{brue2022AsymptoticsII}
E.~Bru\`e, M.~Calzi, G.~E. Comi, and G.~Stefani.
\newblock A distributional approach to fractional {S}obolev spaces and
  fractional variation: asymptotics {II}.
\newblock {\em C. R. Math. Acad. Sci. Paris}, 360:589--626, 2022.

\bibitem[BCMC20]{bellido2020piola}
J.~C. Bellido, J.~Cueto, and C.~Mora-Corral.
\newblock Fractional {P}iola identity and polyconvexity in fractional spaces.
\newblock {\em Ann. Inst. H. Poincar\'{e} C Anal. Non Lin\'{e}aire},
  37(4):955--981, 2020.

\bibitem[BCMC21]{bellido2020gamma}
J.~C. Bellido, J.~Cueto, and C.~Mora-Corral.
\newblock {$\Gamma$-convergence of polyconvex functionals involving
  s-fractional gradients to their local counterparts}.
\newblock {\em Calculus of Variations and Partial Differential Equations},
  60(1):\text{Paper No. 7}, 2021.

\bibitem[BGL73]{bensoussan1973controle}
A.~Bensoussan, M.~Goursat, and J.-L. Lions.
\newblock Contr\^{o}le impulsionnel et in\'{e}quations quasi-variationnelles
  stationnaires.
\newblock {\em C. R. Acad. Sci. Paris S\'{e}r. A-B}, 276:A1279--A1284, 1973.

\bibitem[BL82]{bensoussan1982Controlle}
A.~Bensoussan and J.-L. Lions.
\newblock {\em Contr\^{o}le impulsionnel et in\'{e}quations quasi
  variationnelles}, volume~11 of {\em M\'{e}thodes Math\'{e}matiques de
  l'Informatique [Mathematical Methods of Information Science]}.
\newblock Gauthier-Villars, Paris, 1982.

\bibitem[BM18]{brezis2018Gagliardo}
H.~Brezis and P.~Mironescu.
\newblock Gagliardo-{N}irenberg inequalities and non-inequalities: the full
  story.
\newblock {\em Ann. Inst. H. Poincar\'{e} C Anal. Non Lin\'{e}aire},
  35(5):1355--1376, 2018.

\bibitem[Boc21]{boccardo2021some}
L.~Boccardo.
\newblock Some new results about {M}osco convergence.
\newblock {\em J. Convex Anal.}, 28(2):387--394, 2021.

\bibitem[BPS16]{brasco2016Stability}
L.~Brasco, E.~Parini, and M.~Squassina.
\newblock Stability of variational eigenvalues for the fractional
  {$p$}-{L}aplacian.
\newblock {\em Discrete Contin. Dyn. Syst.}, 36(4):1813--1845, 2016.

\bibitem[Bre68]{brezis1968equations}
H.~Brezis.
\newblock \'{E}quations et in\'{e}quations non lin\'{e}aires dans les espaces
  vectoriels en dualit\'{e}.
\newblock {\em Ann. Inst. Fourier (Grenoble)}, 18(fasc. 1):115--175, 1968.

\bibitem[Bro77]{browder1977PseudoMonotone}
F.~E. Browder.
\newblock Pseudo-monotone operators and nonlinear elliptic boundary value
  problems on unbounded domains.
\newblock {\em Proc. Nat. Acad. Sci. U.S.A.}, 74(7):2659--2661, 1977.

\bibitem[BV16]{bucur2016Nonlocal}
C.~Bucur and E.~Valdinoci.
\newblock {\em Nonlocal diffusion and applications}, volume~20 of {\em Lecture
  Notes of the Unione Matematica Italiana}.
\newblock Springer, [Cham]; Unione Matematica Italiana, Bologna, 2016.

\bibitem[Cal61]{calderon1961Lebesgue}
A.-P. Calder\'{o}n.
\newblock Lebesgue spaces of differentiable functions and distributions.
\newblock In {\em Proc. {S}ympos. {P}ure {M}ath., {V}ol. {IV}}, pages 33--49.
  American Mathematical Society, Providence, R.I., 1961.

\bibitem[Cam21]{campos2021Lions}
P.~Campos.
\newblock Lions-{C}alderón spaces and applications to nonlinear fractional
  partial differential equations.
\newblock Master's thesis, Faculdade de Ciências da Universidade de Lisboa,
  \url{https://repositorio.ul.pt/handle/10451/52753}, 2021.

\bibitem[CR23]{campos2023hyperbolic}
P.~M. Campos and J.~F. Rodrigues.
\newblock On fractional and classical hyperbolic obstacle-type problems.
\newblock {\em Discrete and {C}ontinuous {D}ynamical {S}ystems - {S}}, 2023.
\newblock \url{https://doi.org/10.3934/dcdss.2023164}.

\bibitem[CS19]{comi2019BlowUp}
G.~E. Comi and G.~Stefani.
\newblock A distributional approach to fractional {S}obolev spaces and
  fractional variation: existence of blow-up.
\newblock {\em J. Funct. Anal.}, 277(10):3373--3435, 2019.

\bibitem[CS22]{comi2022asymptoticsI}
G.~E. Comi and G.~Stefani.
\newblock A distributional approach to fractional {S}obolev spaces and
  fractional variation: asymptotics {I}.
\newblock {\em Revista Matem{\'a}tica Complutense}, pages 1--79, 2022.

\bibitem[DNPV12]{diNezza2012Hitchhiker}
E.~Di~Nezza, G.~Palatucci, and E.~Valdinoci.
\newblock Hitchhiker's guide to the fractional {S}obolev spaces.
\newblock {\em Bull. Sci. Math.}, 136(5):521--573, 2012.

\bibitem[dTGCV20]{delTeso2020Estimates}
F.~del Teso, D.~G\'{o}mez-Castro, and J.~L. V\'{a}zquez.
\newblock Estimates on translations and {T}aylor expansions in fractional
  {S}obolev spaces.
\newblock {\em Nonlinear Anal.}, 200:111995, 12, 2020.

\bibitem[dTGCV21]{delTeso2021Three}
F.~del Teso, D.~G\'{o}mez-Castro, and J.~L. V\'{a}zquez.
\newblock Three representations of the fractional {$p$}-{L}aplacian: semigroup,
  extension and {B}alakrishnan formulas.
\newblock {\em Fract. Calc. Appl. Anal.}, 24(4):966--1002, 2021.

\bibitem[EE23]{edmunds2023Fractional}
D.~E. Edmunds and W.~D. Evans.
\newblock {\em Fractional {S}obolev spaces and inequalities}, volume 230 of
  {\em Cambridge Tracts in Mathematics}.
\newblock Cambridge University Press, Cambridge, 2023.

\bibitem[FJW22]{feulefack2022Small}
P.~A. Feulefack, S.~Jarohs, and T.~Weth.
\newblock Small order asymptotics of the {D}irichlet eigenvalue problem for the
  fractional {L}aplacian.
\newblock {\em J. Fourier Anal. Appl.}, 28(2):Paper No. 18, 44, 2022.

\bibitem[GJKR22]{gwinner2022uncertainty}
J.~Gwinner, B.~Jadamba, A.~A. Khan, and Fabio Raciti.
\newblock {\em Uncertainty quantification in variational inequalities---theory,
  numerics, and applications}.
\newblock CRC Press, Boca Raton, FL, 2022.

\bibitem[Gra14a]{grafakos2014Classical}
L.~Grafakos.
\newblock {\em Classical {F}ourier analysis}, volume 249 of {\em Graduate Texts
  in Mathematics}.
\newblock Springer, New York, third edition, 2014.

\bibitem[Gra14b]{grafakos2014Modern}
L.~Grafakos.
\newblock {\em Modern {F}ourier analysis}, volume 250 of {\em Graduate Texts in
  Mathematics}.
\newblock Springer, New York, third edition, 2014.

\bibitem[Hor59]{horvath1959On}
J.~Horv\'{a}th.
\newblock On some composition formulas.
\newblock {\em Proc. Amer. Math. Soc.}, 10:433--437, 1959.

\bibitem[JK95]{jerison1995Inhomogeneous}
D.~Jerison and C.~E. Kenig.
\newblock The inhomogeneous {D}irichlet problem in {L}ipschitz domains.
\newblock {\em J. Funct. Anal.}, 130(1):161--219, 1995.

\bibitem[JM79]{joly1979propos}
J.-L. Joly and U.~Mosco.
\newblock \`a propos de l'existence et de la r\'{e}gularit\'{e} des solutions
  de certaines in\'{e}quations quasi-variationnelles.
\newblock {\em J. Functional Analysis}, 34(1):107--137, 1979.

\bibitem[Kac60]{kachurovskii1960monotone}
R.~I. Kachurovskii.
\newblock Monotone operators and convex functionals.
\newblock {\em Uspekhi Matematicheskikh Nauk}, 15(4):213--215, 1960.

\bibitem[KKM08]{kano2008existence}
R.~Kano, N.~Kenmochi, and Y.~Murase.
\newblock Existence theorems for elliptic quasi-variational inequalities in
  {B}anach spaces.
\newblock In {\em Recent advances in nonlinear analysis}, pages 149--169. World
  Sci. Publ., Hackensack, NJ, 2008.

\bibitem[KS22]{kreisbeck2022Quasiconvexity}
C.~Kreisbeck and H.~Sch\"{o}nberger.
\newblock Quasiconvexity in the fractional calculus of variations:
  characterization of lower semicontinuity and relaxation.
\newblock {\em Nonlinear Anal.}, 215:Paper No. 112625, 26, 2022.

\bibitem[Kur81]{kurokawa1981On}
T.~Kurokawa.
\newblock On the {R}iesz and {B}essel kernels as approximations of the
  identity.
\newblock {\em Sci. Rep. Kagoshima Univ.}, (30):31--45, 1981.

\bibitem[Kwa17]{kwasnicki2017Ten}
M.~Kwa\'{s}nicki.
\newblock Ten equivalent definitions of the fractional {L}aplace operator.
\newblock {\em Fract. Calc. Appl. Anal.}, 20(1):7--51, 2017.

\bibitem[Lio60]{lions1960Une}
J.-L. Lions.
\newblock Une construction d'espaces d'interpolation.
\newblock {\em C. R. Acad. Sci. Paris}, 251:1853--1855, 1960.

\bibitem[Lio69]{lions1969Quelques}
J.-L. Lions.
\newblock {\em Quelques m\'{e}thodes de r\'{e}solution des probl\`emes aux
  limites non lin\'{e}aires}.
\newblock Dunod, Paris; Gauthier-Villars, Paris, 1969.

\bibitem[LM61]{lions1961problemi}
J.-L. Lions and E.~Magenes.
\newblock Problemi ai limiti non omogenei (iii).
\newblock {\em Annali della Scuola Normale Superiore di Pisa-Classe di
  Scienze}, 15(1-2):41--103, 1961.

\bibitem[LM80]{landes1980On}
R.~Landes and V.~Mustonen.
\newblock On pseudomonotone operators and nonlinear noncoercive variational
  problems on unbounded domains.
\newblock {\em Math. Ann.}, 248(3):241--246, 1980.

\bibitem[LP15]{linares2015Introduction}
F.~Linares and G.~Ponce.
\newblock {\em Introduction to nonlinear dispersive equations}.
\newblock Universitext. Springer, New York, second edition, 2015.

\bibitem[LR23a]{lo2021class}
C.~W.~K. Lo and J.~F. Rodrigues.
\newblock On a class of nonlocal obstacle type problems related to the
  distributional {R}iesz fractional derivative.
\newblock {\em Port. Math}, 80(1/2):157--205, 2023.

\bibitem[LR23b]{lo2023Stefan}
C.~W.~K. Lo and J.~F. Rodrigues.
\newblock On an anisotropic fractional {S}tefan-type problem with {D}irichlet
  boundary conditions.
\newblock {\em Math. Eng.}, 5(3):Paper No. 047, 38, 2023.

\bibitem[MBRS16]{bisci2016Variational}
G.~Molica~Bisci, V.~D. Radulescu, and R.~Servadei.
\newblock {\em Variational methods for nonlocal fractional problems}, volume
  162 of {\em Encyclopedia of Mathematics and its Applications}.
\newblock Cambridge University Press, Cambridge, 2016.

\bibitem[Min62]{minty1962Monotone}
G.~J. Minty.
\newblock Monotone (nonlinear) operators in {H}ilbert space.
\newblock {\em Duke Math. J.}, 29:341--346, 1962.

\bibitem[Mos69]{mosco1969convergence}
U.~Mosco.
\newblock Convergence of convex sets and of solutions of variational
  inequalities.
\newblock {\em Advances in Math.}, 3:510--585, 1969.

\bibitem[Mos76]{mosco1976Implicit}
U.~Mosco.
\newblock Implicit variational problems and quasi variational inequalities.
\newblock In {\em Nonlinear operators and the calculus of variations ({S}ummer
  {S}chool, {U}niv. {L}ibre {B}ruxelles, {B}russels, 1975)}, Lecture Notes in
  Math., Vol. 543, pages 83--156. Springer, Berlin-New York, 1976.

\bibitem[Pon16]{ponce2016Elliptic}
A.~C. Ponce.
\newblock {\em Elliptic {PDE}s, measures and capacities}, volume~23 of {\em EMS
  Tracts in Mathematics}.
\newblock European Mathematical Society (EMS), Z\"{u}rich, 2016.

\bibitem[Rod87]{rodrigues1987obstacle}
J.~F. Rodrigues.
\newblock {\em Obstacle problems in mathematical physics}, volume 134 of {\em
  North-Holland Mathematics Studies}.
\newblock Elsevier, 1987.

\bibitem[Rou13]{roubivcek2013nonlinear}
T.~Roub{\'\i}{\v{c}}ek.
\newblock {\em {Nonlinear Partial Differential Equations with Applications}},
  volume 153 of {\em Int. Series Numerical Math.}
\newblock Birkh\"{a}user, Basel, 2nd edition, 2013.

\bibitem[RS19a]{rodrigues2019On}
J.~F. Rodrigues and L.~Santos.
\newblock On nonlocal variational and quasi-variational inequalities with
  fractional gradient.
\newblock {\em Appl. Math. Optim.}, 80(3):835--852, 2019.
\newblock correction in Appl. Math. Optim. 84:3565–3567, 2021.

\bibitem[RS19b]{rodrigues2019Variational}
J.~F. Rodrigues and L.~Santos.
\newblock Variational and quasi-variational inequalities with gradient type
  constraints.
\newblock In {\em Topics in applied analysis and optimisation---partial
  differential equations, stochastic and numerical analysis}, CIM Ser. Math.
  Sci., pages 319--361. Springer, Cham, 2019.

\bibitem[Sch21]{schonberger2021characterization}
H.~M.~J. Sch{\"o}nberger.
\newblock Characterization of lower semicontinuity and relaxation of fractional
  integral and supremal functionals.
\newblock Master's thesis, 2021.

\bibitem[SEW{\etalchar{+}}07]{silling2007States}
S.~A. Silling, M.~Epton, O.~Weckner, J.~Xu, and E.~Askari.
\newblock Peridynamic states and constitutive modeling.
\newblock {\em J. Elasticity}, 88(2):151--184, 2007.

\bibitem[Sho13]{showalter2013monotone}
R.~E. Showalter.
\newblock {\em {Monotone Operators in Banach Space and Nonlinear Partial
  Differential Equations}}, volume~49 of {\em Mathematical Surveys and
  Monographs}.
\newblock American Mathematical Soc., 2013.

\bibitem[Sil00]{silling2000Reformulation}
S.~A. Silling.
\newblock Reformulation of elasticity theory for discontinuities and long-range
  forces.
\newblock {\em J. Mech. Phys. Solids}, 48(1):175--209, 2000.

\bibitem[Sim81]{simon1981Regularity}
J.~Simon.
\newblock R\'{e}gularit\'{e} de la solution d'un probl\`eme aux limites non
  lin\'{e}aires.
\newblock {\em Ann. Fac. Sci. Toulouse Math. (5)}, 3(3-4):247--274 (1982),
  1981.

\bibitem[SNs]{sobolev1963Embedding}
S.~L. Sobolev and S.~M. Nikol\textquotesingle~ski\u{\i}.
\newblock Embedding theorems.
\newblock In {\em Amer. Math. Soc. Trans. (2), 1970}, pages 147--173.
  Translation of Proc. 4th {A}ll-{U}nion {M}ath. {C}ongr. ({L}eningrad, 1961),
  {V}ol. {I}, pages 227--242. Izdat. Akad. Nauk SSSR, Leningrad, 1963.

\bibitem[SS15]{shieh2015On}
T.-T. Shieh and D.~E. Spector.
\newblock On a new class of fractional partial differential equations.
\newblock {\em Adv. Calc. Var.}, 8(4):321--336, 2015.

\bibitem[SS18]{shieh2018On}
T.-T. Shieh and D.~E. Spector.
\newblock On a new class of fractional partial differential equations {II}.
\newblock {\em Adv. Calc. Var.}, 11(3):289--307, 2018.

\bibitem[SSS18]{schikorra2018Regularity}
A.~Schikorra, T.-T. Shieh, and D.~E. Spector.
\newblock Regularity for a fractional {$p$}-{L}aplace equation.
\newblock {\em Commun. Contemp. Math.}, 20(1):1750003, 6, 2018.

\bibitem[Ste70]{stein1970Singular}
E.~M. Stein.
\newblock {\em Singular integrals and differentiability properties of
  functions}.
\newblock Princeton Mathematical Series, No. 30. Princeton University Press,
  Princeton, N.J., 1970.

\bibitem[Str67]{strichartz1967Multipliers}
R.~Strichartz.
\newblock Multipliers on fractional {S}obolev spaces.
\newblock {\em J. Math. Mech.}, 16:1031--1060, 1967.

\bibitem[Tri83]{triebel1983Theory}
H.~Triebel.
\newblock {\em Theory of function spaces}, volume~78 of {\em Monographs in
  Mathematics}.
\newblock Birkh\"{a}user Verlag, Basel, 1983.

\bibitem[\v{S}16]{silhavy2016Beyond}
M.~\v{S}ilhav\'{y}.
\newblock Beyond fractional laplacean: fractional gradient and divergence.
\newblock
  \url{https://www.researchgate.net/publication/295095188_Beyond_fractional_laplaceanfractional_gradient_and_divergence},
  2016.

\bibitem[\v{S}20]{silhavy2020Fractional}
M.~\v{S}ilhav\'{y}.
\newblock Fractional vector analysis based on invariance requirements (critique
  of coordinate approaches).
\newblock {\em Contin. Mech. Thermodyn.}, 32(1):207--228, 2020.

\end{thebibliography}
\end{document}